\documentclass[EJP]{ejpecp_old}

\usepackage[T1]{fontenc}
\usepackage[english]{babel}
\usepackage[latin9]{inputenc}

\usepackage{multirow, multicol}	

\usepackage{amsmath}
\usepackage{array}
\usepackage{amsthm}
\usepackage{amssymb}
\usepackage{esint}
\usepackage{tikz}
 \usepackage{mathtools}
\usepackage{xcolor}
\PassOptionsToPackage{normalem}{ulem}
\usepackage{ulem}
\PassOptionsToPackage{unicode}{hyperref}
\PassOptionsToPackage{naturalnames}{hyperref}
\allowdisplaybreaks

\SHORTTITLE{Limits of semi-discrete directed polymers} 

\TITLE{Intermediate disorder limits for multi-layer semi-discrete directed polymers} 

\AUTHORS{%
  Mihai~Nica\footnote{University of Guelph. Supported by NSF grant DMS-1209165, the MacCracken Fellowship from New York University, and the NSERC postdoctoral fellowship. \EMAIL{nicam@uoguelph.ca} }}


\AMSSUBJ{82D60} 

\SUBMITTED{July 26, 2019} 
\ACCEPTED{March 29, 2021} 

\ARXIVID{1609.00298} 

\VOLUME{0}
\YEAR{2016}
\PAPERNUM{0}
\DOI{}

\ABSTRACT{We show that the partition function of the multi-layer semi-discrete directed polymer converges in the intermediate disorder regime to the partition function for the multi-layer continuum polymer introduced by O'Connell and Warren in \cite{OConnellWarren2015}. This verifies, modulo a previously hidden constant, an outstanding conjecture proposed by Corwin and Hammond \cite{CorHam15}. A consequence is the identification of the KPZ line ensemble as logarithms of ratios of consecutive layers of the continuum partition function. Other properties of the continuum partition function, such as continuity, strict positivity and contour integral formulas to compute mixed moments, are also identified from this convergence result.  }

\makeatletter
\theoremstyle{plain}
\newtheorem{thm}{\protect\theoremname}[section]
  \theoremstyle{definition}
  \newtheorem{defn}[thm]{\protect\definitionname}
  \theoremstyle{remark}
  \newtheorem{rem}[thm]{\protect\remarkname}
  \theoremstyle{plain}
  \newtheorem{lem}[thm]{\protect\lemmaname}
  \theoremstyle{plain}
  \newtheorem{prop}[thm]{\protect\propositionname}
  \theoremstyle{plain}
  \newtheorem{cor}[thm]{\protect\corollaryname}

\usepackage{hyperref}
\usepackage{verbatim}

\makeatother

\usepackage{babel}
  \providecommand{\corollaryname}{Corollary}
  \providecommand{\definitionname}{Definition}
  \providecommand{\lemmaname}{Lemma}
  \providecommand{\propositionname}{Proposition}
  \providecommand{\remarkname}{Remark}
\providecommand{\theoremname}{Theorem}

\newcommand\xCa{0.02}
\newcommand\xCb{0.08}
\newcommand\xCc{0.12}
\newcommand\xCd{0.19}
\newcommand\xCe{0.22}
\newcommand\xCf{0.48}
\newcommand\xCg{0.72}

\newcommand\xBa{0.12}
\newcommand\xBb{0.15}
\newcommand\xBc{0.28}
\newcommand\xBd{0.44}
\newcommand\xBe{0.65}
\newcommand\xBf{0.73}
\newcommand\xBg{0.90}

\newcommand\xAa{0.16}
\newcommand\xAb{0.24}
\newcommand\xAc{0.42}
\newcommand\xAd{0.50}
\newcommand\xAe{0.73}
\newcommand\xAf{0.79}
\newcommand\xAg{0.95}

\newcommand\slopex{-4.5}
\newcommand\scaley{0.5}
\newcommand\offx{1.7}
\newcommand\offy{5}

\newcommand{\iintt}[2]{\iintop_{\mathclap{\substack{#1\\#2}}}}
\newcommand{\sintt}[2]{\quad\mathclap{\intop_{\mathclap{\substack{#1\\#2}}}}\mathclap{\textstyle\sum}\quad}

\newcommand\p{\mathbb{P}}

\newcommand\var{\mathbf{Var}}

\newcommand{\norm}[1]{\left\Vert #1\right\Vert }
\newcommand{\abs}[1]{\left|#1\right|}
\newcommand{\given}[1]{\left|#1\right.}

\newcommand{\floor}[1]{\left\lfloor #1\right\rfloor }

\newcommand\dequal{\stackrel{d}{=}}

\newcommand\ld{\ldots}
\newcommand\di{\partial}
\newcommand\defequal{\stackrel{{\scriptscriptstyle \Delta}}{=}}
\newcommand\pr{\prime}

\newcommand\sgn{\text{sgn}}
\newcommand\dd{\text{d}}

\newcommand\To{\Rightarrow}
\newcommand\half{\frac{1}{2}}
\newcommand\oo[1]{\frac{1}{#1}}
\newcommand\nogt{\,0}
\newcommand\dfac{\alpha}

\newcommand\Y{\vec{Y}}
\newcommand\X{\vec{X}}

\newcommand\y{\vec{y}}
\newcommand\x{\vec{x}}
\newcommand\e{\mathbb{E}}
\newcommand\one{\mathbf{1}}
\newcommand\B{\vec{B}}
\newcommand\D{\vec{D}}

\newcommand\xf{{x^{\star}}}

\newcommand\zf{{z^{\star}}}
\newcommand\tf{{t^{\star}}}
\newcommand\nf{{n^{\star}}}

\newcommand\taf{{\ta^{\star}}}

\newcommand\bN{\mathbb{N}}

\newcommand\bR{\mathbb{R}}
\newcommand\bS{\mathbb{S}}

\newcommand\bW{\mathbb{W}}

\newcommand\bZ{\mathbb{Z}}

\newcommand\cE{\mathcal{E}}
\newcommand\cF{\mathcal{F}}

\newcommand\cH{\mathcal{H}}

\newcommand\cL{\mathcal{L}}

\newcommand\cP{\mathcal{P}}

\newcommand\cX{\mathcal{X}}

\newcommand\cZ{\mathcal{Z}}

\newcommand\al{\alpha}
\newcommand\be{\beta}
\newcommand\ga{\gamma}
\newcommand\Ga{\Gamma}
\newcommand\de{\delta}
\newcommand\De{\Delta}
\newcommand\ep{\epsilon}

\newcommand\et{\eta}

\newcommand\la{\lambda}

\newcommand\rh{\rho}
\newcommand\si{\sigma}

\newcommand\ta{\tau}
\newcommand\ph{\phi}

\newcommand\vp{\varphi}

\newcommand\ps{\psi}

\newcommand\Om{\Omega}

\begin{document}



\tableofcontents{}
\section{ Introduction }

\subsection{Background and main results}

Let $d\in \bN = \left\{1,2,3\ldots \right\}$, $\be >0 $, $\tf >0$ and $\zf \in \bR$. In this work, we will use the superscript $\star$ to denote quantities related to the endpoint of bridges; for example $(\tf,\zf)$ will denote the endpoint of $d$ non-intersecting Brownian bridges. O'Connell and Warren \cite{OConnellWarren2015} define a continuum partition function which is given by the following white noise chaos series:  
\begin{equation}
\label{eq:cZ_def}
\cZ_{d}^{\be}(\tf,\zf) \defequal \rh(\tf,\zf)^{d} \sum_{k=0}^{\infty}\be^{k} \; \; \mathclap{\intop_{\De_k(0,\tf)}} \quad \; \; \intop_{\bR^k} \ps_{k}^{(\tf,\zf)}\Big((t_1,z_1),\ld,(t_k,z_k)\Big)\xi(\dd t_1,\dd z_1) \cdots \xi(\dd t_k, \dd z_k) ,
\end{equation}
where $\xi(t,z)$ is 1+1 dimensional Gaussian white noise, $\De_k(a,b)$ denotes the set of ordered $k$-tuples of time coordinates, $\rh(t,z)$ is the heat kernel:
\begin{align}\label{eq:simplex}
\De_k(a,b) &\defequal \{a<t_1<\ld<t_k<b\},\\
\rh(t,z) &\defequal (2\pi t)^{-1/2}\exp(-z^2/2t), \nonumber
\end{align}
and $\ps_{k}^{(\tf,\zf)}$ is the $k$-point correlation function for $d$ non-intersecting Brownian bridges, each of which starts at $0$ at time $0$ and ends at $\zf$ at time $\tf$ (see Definition \ref{def:NIBb} for a precise definition). 


 In the case $d=1$, $\cZ^{\be}_{1}$ is a solution to the stochastic heat equation with multiplicative white noise
\begin{equation*}
\di_t \cZ^{\be}_{1} = \half \di_{xx} \cZ^{\be}_{1} + \beta \xi \cZ^{\be}_{1},
\end{equation*}
with delta initial data \cite{AKQ_2014}.  Moreover, $\cZ^{\be}_{1}$ was shown to be the universal scaling limit of the partition function for discrete directed polymers in the \emph{intermediate disorder regime} introduced by Alberts, Khanin and Quastel \cite{AKQ_2014}. In this scaling limit the strength of the random environment is scaled to zero in a critical way as the size of the discrete system grows to infinity. Similarly, when $d>1$, $\cZ^{\be}_{d}$ was shown to be the universal limit in the intermediate disorder regime for discrete directed polymers consisting of $d$ non-intersecting simple symmetric random walks in \cite{CorwinNica16}.

Another random polymer model that has received recent attention is the O'Connell-Yor semi-discrete directed polymer introduced in \cite{OCYor}, where the polymers are in continuous time but in discrete space. 
It was shown in \cite{Oconnell12} that the multi-layer version of this, which involves several non-intersecting polymer paths, has an algebraic structure related to Whittaker functions and the quantum Toda lattice. This multi-layer semi-discrete partition function is the main object of study in this paper and is defined precisely below.

\begin{defn}
\label{def:semi_discrete_parition} An up/right path in $\bR\times\bN$ is an increasing path
which either proceeds to the right or jumps up by exactly one unit. For any
$\taf>0$ and any $\xf\in\bN$, each sequence $0<\ta_{1}<\ld<\ta_{\xf}<\taf$ is associated to an up/right path $X^{(\taf,\xf)}(\cdot)$ which travels from the lattice point $(0,1)$
to $(\taf,\xf+1)$, and which jumps between the points $(\ta_{i},i)$ and
$(\ta_{i},i+1)$ for $1\leq i \leq \xf$ and otherwise always travels to the right. The list $\vec{\ta}\in \De_{\xf}(0,\taf) \subset \bR^{\xf}$ can be thought of as the ``jump times'' of the up/right path; the list of jump times is in bijection with the up/right path $X^{(\taf,\xf)}(\cdot)$ and we therefore conflate the two notions with the convention that the paths are cadlag. 

Let $\left\{ B_{i}(\cdot)\right\}_{i=1}^{\infty}$ be an infinite family
of independent standard Brownian motions on a probability space $\left(\Om,\cF,\cP\right)$.
Define the energy of the up/right path $X^{(\taf,\xf)}(\cdot)$ to be the following random
variable on the probability space $\Om$:
\begin{align}
H\left(X^{(\taf,\xf)}\right) & \defequal B_{1}(\ta_{1})+\left(B_{2}(\ta_{2})-B_{2}(\ta_{1})\right)+\cdots+\left(B_{\xf+1}\left(\taf\right)-B_{\xf+1}\left(\ta_{\xf}\right)\right) \nonumber \\
&=\sum_{j=1}^{\xf+1}\intop_{0}^{\taf}\one\left\{ j = X^{(\taf,\xf)}(s)\right\} \dd B_{j}(s) = \intop_0^\taf \dd B_{X^{(\taf,\xf)}(s)}(s). \label{eq:energy_def}
\end{align}
We can think of $X^{(\taf,\xf)}(\cdot)$ as a random up/right path in a
natural way by taking the probability measure on the set of jump times $\vec{\ta} \in \De_{\xf}(0,\taf)\subset \bR^{\xf}$ which is proportional to the Lebesgue measure on $\bR^{\xf}$. If we denote by $\e$ the expectation with respect to this measure, we define for any $\be>0$, the directed polymer partition function $Z_{1}^{\be}(\taf,\xf)$, which is a random variable on the probability space $\Om$, by
\[
Z_{1}^{\be}(\taf,\xf)\defequal\e\left[\exp\left(\be H(X^{(\taf,\xf)})\right)\right].
\]


We generalize this for $d > 1$ by taking multiple up/right paths as follows. Let $\vec{X}^{(\taf,\xf)}(\cdot)=\left(X_{1}^{(\taf,\xf)}(\cdot),\ld X_{d}^{(\taf,\xf)}(\cdot)\right)$ be a collection of $d$ up/right paths with initial points $X_{i}^{(\taf,\xf)}(0)=i$ and final points $X_{i}^{(\taf,\xf)}(\taf)=\xf+i$ for $1\leq i \leq d$ which are \emph{non-intersecting}. More specifically, the non-intersecting condition that we require is that $X_i^{(\taf,\xf)}(\ta) < X_j^{(\taf,\xf)}(\ta)$ for all $i < j$ and for all times $\ta \in (0,\taf)$.

\begin{figure} 
\begin{center}
\begin{tikzpicture}[yscale = 0.8,
declare function={ my(\x,\y)                 = \y*\scaley+\x*\slopex+\offy; 
                  },xscale = 8,yscale=0.5     
]			

			\node[anchor = north, gray] at (0,1) (noname) {$0$};
			\node[anchor = north, gray] at (1,1) (noname) {$\taf$};

			\foreach \y in {1,2,3,4}
   			    \draw[thin,gray](1,\y) -- (0,\y)
   			     node[anchor=east] {$\y$};
   			     
   			\foreach \y in {5,6}
   			    \draw[thin,gray](1,\y) -- (0,\y)
   			     node[anchor=east] {$\vdots$};

   			\foreach \y in {7}
   			    \draw[thin,gray](1,\y) -- (0,\y)
   			     node[anchor=east] {$\xf$};   			
   			
   			\foreach \y in {1,2,3}
   			    \draw[thin,gray](1,7+\y) -- (0,7+\y)
   			     node[anchor=east] {$\xf + \y$};

   			
   			    
   			     
   			
   			\draw [black,thick] (0,1) -- (\xAa,1);
   			\draw [black,dotted] (\xAa,1) -- (\xAa,2);
   			\draw [black,thick] (\xAa,2) -- (\xAb,2);
   			\draw [black,dotted] (\xAb,2) -- (\xAb,3);
   			\draw [black,thick] (\xAb,3) -- (\xAc,3);
   			\draw [black,dotted] (\xAc,3) -- (\xAc,4);
   			\draw [black,thick] (\xAc,4) -- (\xAd,4);
   			\draw [black,dotted] (\xAd,4) -- (\xAd,5);
   			\draw [black,thick] (\xAd,5) -- (\xAe,5);
   			\draw [black,dotted] (\xAe,5) -- (\xAe,6);
   			\draw [black,thick] (\xAe,6) -- (\xAf,6);
   			\draw [black,dotted] (\xAf,6) -- (\xAf,7);
   			\draw [black,thick] (\xAf,7) -- (\xAg,7);
   			\draw [black,dotted] (\xAg,7) -- (\xAg,8);
   			\draw [black,thick] (\xAg,8) -- (1,8);
   			
   			\draw [black,thick] (0,2) -- (\xBa,2);
   			\draw [black,dotted] (\xBa,2) -- (\xBa,3);
   			\draw [black,thick] (\xBa,3) -- (\xBb,3);
   			\draw [black,dotted] (\xBb,3) -- (\xBb,4);
   			\draw [black,thick] (\xBb,4) -- (\xBc,4);
   			\draw [black,dotted] (\xBc,4) -- (\xBc,5);
   			\draw [black,thick] (\xBc,5) -- (\xBd,5);
   			\draw [black,dotted] (\xBd,5) -- (\xBd,6);
   			\draw [black,thick] (\xBd,6) -- (\xBe,6);
   			\draw [black,dotted] (\xBe,6) -- (\xBe,7);
   			\draw [black,thick] (\xBe,7) -- (\xBf,7);
   			\draw [black,dotted] (\xBf,7) -- (\xBf,8);
   			\draw [black,thick] (\xBf,8) -- (\xBg,8);
   			\draw [black,dotted] (\xBg,8) -- (\xBg,9);
   			\draw [black,thick] (\xBg,9) -- (1,9);
   			
   			\draw [black,thick] (0,3) -- (\xCa,3);
   			\draw [black,dotted] (\xCa,3) -- (\xCa,4);
   			\draw [black,thick] (\xCa,4) -- (\xCb,4);
   			\draw [black,dotted] (\xCb,4) -- (\xCb,5);
   			\draw [black,thick] (\xCb,5) -- (\xCc,5);
   			\draw [black,dotted] (\xCc,5) -- (\xCc,6);
   			\draw [black,thick] (\xCc,6) -- (\xCd,6);
   			\draw [black,dotted] (\xCd,6) -- (\xCd,7);
   			\draw [black,thick] (\xCd,7) -- (\xCe,7);
   			\draw [black,dotted] (\xCe,7) -- (\xCe,8);
   			\draw [black,thick] (\xCe,8) -- (\xCf,8);
   			\draw [black,dotted] (\xCf,8) -- (\xCf,9);
   			\draw [black,thick] (\xCf,9) -- (\xCg,9);
   			\draw [black,dotted] (\xCg,9) -- (\xCg,10);
   			\draw [black,thick] (\xCg,10) -- (1,10);
\end{tikzpicture}
\end{center}
\caption{$d$ non-intersecting up/right paths $\vec{X}^{(\taf,\xf)}$ started from $X_i(0)=i$ and ended at $X_i(\taf)=\xf+i$ for $1\leq i \leq d$. In this example $d = 3$.} \label{fig:up_right_paths}
		\end{figure}

Notice now that all the jump times for the $d$ up/right paths taken together can be thought of as a vector in $\bR^{d\xf}$. We can think of $\vec{X}^{(\taf,\xf)}(\cdot)$ as a random process by taking the probability measure on this list of jump times proportional to the Lebesgue measure on the subset of $\bR^{d\xf}$ of allowed configuration. (The process $\vec{X}^{(\taf,\xf)}$ defined in this way also has a natural interpretation as certain Poisson walkers conditioned to be non-intersecting: see Remark \ref{rem:Poisson_bridges}). Figure \ref{fig:up_right_paths} shows a typical realization of these paths. Denoting by $\e$ the expectation with respect to this measure, we define the partition function:
\[
Z_{d}^{\be}\left(\taf,\xf\right)\defequal\e\left[\exp\left(\be\sum_{i=1}^{d}H\left(X_{i}^{(\taf,\xf)}\right)\right)\right].
\]
\end{defn}

In the case $d=1$, it is shown in \cite{QuastelInPrep} that the semi-discrete partition function $Z_1^{\be}$ converges to $\cZ^{\be}_{1}$ from equation \eqref{eq:cZ_def} in the intermediate disorder regime. An expository presentation of this proof is given in \cite{MSRInotestwo}. The main result of this article, Theorem \ref{thm:main_thm}, is to extend this to $d > 1$: a convergence result for semi-discrete polymers consisting of $d$ non-intersecting up/right paths that start and end grouped together. 


\begin{thm}
\label{thm:main_thm}Fix $d \in \bN$, $\tf>0$, $\zf \in\bR$, and $\be > 0$.  Recall for any $\taf >0, \xf \in \bN$, that $Z_{d}^{\be}(\taf,\xf)$ denotes the semi-discrete partition function for $d$ non-intersecting up/right paths as in Definition \ref{def:semi_discrete_parition}. For any sequence $\be_N$ with $N^{\oo{4}} \be_N  \to \be$ as $N \to \infty$, we have the following convergence in distribution as $N\to\infty$:
\begin{equation}\label{eq:main_thm_eq}
Z_{d}^{\be_{N}}\left(\tf N,\floor{\tf N+\zf \sqrt{N}}\right)\exp\left(-\frac{d}{2} \tf N \be_N^{2}\right)\To\frac{\cZ_{d}^{\be}\left(\tf, \zf \right)}{\rho\left(\tf,\zf \right)^{d}}.
\end{equation}

Moreover, it is possible to find a coupling of the probability space on which $Z_{d}^{\be_{N}}$ is defined to the probability space on which $\cZ_{d}^{\be}$ is defined, so that the convergence is in $L^{p}$ for any $p \geq 1$. 

Finally, if one treats the LHS and RHS of equation \eqref{eq:main_thm_eq} as stochastic processes indexed by $d \in \bN, \tf \in (0,\infty), \zf \in \bR$, then the convergence also holds for the finite dimensional distributions of these processes.
\end{thm}

\begin{rem} \label{rem:e_rem} The expected value of $Z^\be_d(\taf,\xf)$ is always $\exp(\frac{d}{2} \taf \be^2)$ irrespective of $\xf$ (see Lemma \ref{lem:chaos_series_for_Z}). This explains the scaling on the LHS of equation \eqref{eq:main_thm_eq}, which is exactly needed to normalize to unit expectation. \end{rem}

\begin{rem}
A similar convergence result for non-intersecting ensembles of simple symmetric random walks in discrete time and discrete space was proven in \cite{CorwinNica16}. The interest behind the result for discrete random walks was that the convergence did not depend on the ``disordered environment'' used to define the partition function: the convergence was universal in this sense. 

The interest in the semi-discrete result is different. Instead of a universality-type result, Theorem \ref{thm:main_thm} is only about convergence of the specific model in Definition \ref{def:semi_discrete_parition} where the environment is given by independent Brownian motions. The reason Theorem \ref{thm:main_thm} is useful is that many properties of the model in Definition \ref{def:semi_discrete_parition} are already known, and these properties can be carried over in the limit. Therefore, Theorem \ref{thm:main_thm} gives a way to  extract information about the limiting partition function $\cZ_d^\be$. In particular, this can be used to prove a connection between $\cZ_d^\be$ and the KPZ line ensemble defined in \cite{CorHam15}, which was the original motivation for the convergence result. Section \ref{sec:applications} details some of these applications of Theorem \ref{thm:main_thm}. 
\end{rem}

The main technical tool in the proof of Theorem
\ref{thm:main_thm} is the $L^2$ convergence of the $k$-point correlation functions of the non-intersecting up/right paths to those of the non-intersecting Brownian bridges  under the diffusive scaling $\left(\ta,x\right)\approx\big(Nt,Nt+\sqrt{N}z\big)$. This is encapsulated in the following convergence result:

\begin{thm}
\label{thm:L2_convergence_psiN_to_psi} For any $t>0$, $z\in\bR$, and $k\in\bN$, let $\ps_{k}^{(N),(t,z)}:{\De_k(0,t)\times\bR^k}\to\bR$
be the $k$-point correlation functions for $d$ non-intersecting up/right paths which start and end grouped together and are rescaled diffusivly in space and time; see Definition \ref{def:rescaled_NIPb} for a precise definition.
Let $\ps_{k}^{(t,z)}:\De_k(0,t)\times\bR^k\to\bR$ be the $k$-point correlation function for $d$ non-intersecting Brownian
bridges given in Definition \ref{def:NIBb}. Then we have the following
convergence as $N\to\infty$:
\begin{equation}
\lim_{N\to\infty}\norm{\ps_{k}^{(N),(t,z)}-\ps_{k}^{(t,z)}}_{L^{2}\left(\De_k(0,t)\times\bR^k\right)}=0.\label{eq:L2_thm}
\end{equation}
\end{thm}

The argument which shows that Theorem \ref{thm:L2_convergence_psiN_to_psi} implies Theorem \ref{thm:main_thm} is carried out in Section \ref{sec:pf} and uses the theory of Gaussian Hilbert spaces to directly connect the semi-discrete polymer and the continuum polymer. This is different than the method of polynomial chaos series developed in \cite{caravenna_sun_zygouras_2015polynomial} which was used as an intermediate step for the convergence of discrete non-intersecting random walks established in \cite{CorwinNica16}. 

The proof of Theorem \ref{thm:L2_convergence_psiN_to_psi} is carried out in Section \ref{sec:pf_of_L2}, with technical lemmas deferred to later sections. The proof goes by extending the methods introduced in \cite{CorwinNica16} which were used to prove a similar $L^2$ convergence result for discrete non-intersecting random walks. An additional complication that must be handled here is due to the exponentially rare event that a continuous-time random process takes many steps in a short amount of time. We extend the method of exponential moment control used in \cite{CorwinNica16} in order to handle this type of rare event. Another complication is that the discrete Tanaka formula used in \cite{CorwinNica16} does not apply to the continuous-time random processes studied here. To handle this, it is necessary to first ``de-Poissonize'' the processes before proving certain bounds, and then ``re-Poissonize'' to get back to the original model; this is carried out in Section \ref{sec:dePoiss}.

\subsection{Applications of Theorem \ref{thm:main_thm}}\label{sec:applications}

We connect the notation from Definition \ref{def:semi_discrete_parition} to other work in the literature and present some applications of Theorem \ref{thm:main_thm} in the corollaries below. Many of these corollaries were conjectured in Sections 2.3.3 and 2.3.4 of \cite{CorHam15}. (Note that the numbering of equations in \cite{CorHam15} refers to the published version and may differ from the latest arXiv version of that paper.)

\begin{defn}[Following Definitions 3.1 and 3.5 of \cite{CorHam15}] \label{def:KPZ_line} 
For each $M > 0, t >0$ define
\[ 
C(M,t,z) \defequal \exp\left\{ M + \frac{\sqrt{tM}+z}{2} + zt^{-\half}M^{\half} \right\}(t^\half M^{-\half})^M.
\]
Let $D^{(M)}_d(t)\subset \bR^{d(M-d)}$ denote the set of jump times for $d$-tuples of non-intersecting up/right paths with initial points $\delta(0) = (0,1),\ld,(0,d)$ and endpoints $(t,M-d+1),\ld(t,M)$. Identifying up/right paths with their jump times as we did in Definition \ref{def:semi_discrete_parition}, define $Z_d^{(M)}(t)$ as the integral (w.r.t Lebesgue measure) over all such $d$-tuples of non-intersecting up/right paths:
\[
Z_d^{(M)}(t) \defequal \intop_{D^{(M)}_d(t)} e^{\sum_{i=1}^d H(\ph_i)} \dd \ph_1 \ld \dd \ph_d.
\]
Note that $Z_d^{(M)}(t)$ is proportional to $Z_d^\be(\taf,\xf)$ from Definition \ref{def:semi_discrete_parition} when $\be =1$, $\taf = t$ and $\xf = M -d$; indeed they are related by the Lebesgue measure of the set $D^{(M)}_d(t)$:
\begin{equation} \label{eq:oldZ_newZ_relation}
Z_d^{(M)}(t) = \cL\Big( D^{(M)}_d(t) \Big) Z_d^{1}(t,M-d).
\end{equation}
The Lebesgue measure of this set is explicitly calculated in Lemma \ref{lem:Leb_measure}. For $M \in \bN$, $d \in \bN$, $z \in \bR$ so that $z > -\sqrt{tM}$ define $\cZ^{t,M}_{d}(z)$ by: 
\[
\cZ^{t,M}_{d}(z) \defequal \frac{Z_d^{(M)}(\sqrt{tM}+z)}{C(M,t,z)^d}.
\]
\end{defn}


\begin{cor}[Conjecture 2.18 of \cite{CorHam15} modulo the constant $c_{d,t}$] \label{cor:KPZ_line_ensemble_corr}
 For any $d \in \bN$, $t > 0$, define the constant $c_{d,t}$ by
\begin{equation} \label{eq:constant}
c_{d,t}^{-1} \defequal t^{-\half d(d-1)}\prod_{i=0}^{d-1}i!. 
\end{equation} 
For any fixed $t > 0$ and $z \in \bR$, we have the convergence as $M\to \infty$:
\begin{equation} \label{eq:KPZ_line_ensemble_convergence}
\cZ_d^{t,M}(z) \To c_{d,t}^{-1} \cZ_d^1(t,z).
\end{equation}
Moreover, thinking of the LHS and RHS of equation \eqref{eq:KPZ_line_ensemble_convergence} as stochastic process indexed by $d \in \bN$, and $z \in \bR$, the convergence holds for finite dimensional distributions of these processes and the convergence holds for the $p$-th moment of these processes for any $p \geq 1$.
\end{cor}

Corollary \ref{cor:KPZ_line_ensemble_corr} follows from Theorem \ref{thm:main_thm} after recognizing $\cZ^{t,M}_d$ in terms of $Z^{\be_N}_d$ as in equation \eqref{eq:oldZ_newZ_relation} and applying Stirling's formula to estimate the Lebesgue measure of the set $D^{(M)}_d(s)$. These calculations are deferred to Section \ref{sec:cor_pf}.

\begin{rem} \label{rem:missing_constants} Note that in the original Conjecture 2.18 of \cite{CorHam15}, the constant $c_{d,t}^{-1}$ on the RHS of equation \eqref{eq:KPZ_line_ensemble_convergence} was absent. This constant is nontrivial only for $d > 1$, and arises due to the ``squeezing'' together of the start and end points of the polymer in the diffusive scaling limit. $c_{d,t}^{-1}$ can be related to the constant that appears due to this same squeezing in equation (12) of \cite{LunWarren15}. The normalization $c_{d,t}^{-1} \cZ^{t,M}_d(t,z)$ is also the correct normalization so that the time evolution of this object is related to the 2D Toda equations, see \cite{LunWarren16}.

Note that the omission of this constant in Conjecture 2.18 of \cite{CorHam15} does not effect any of the  analysis of the KPZ equation carried out there, since these applications are based on studying $\cH^t_1$, defined below in Corollary \ref{cor:unique_KPZ}, for which the constant $c_{1,t} \equiv 1$ has no effect.
\end{rem}
\begin{cor} \label{cor:unique_KPZ}
For each $t>0$, $n\in\bN$, $z \in \bR$, set
\begin{equation} \label{eq:KPZ_line_ensemble_log}
\cH_{n}^{t}(z) = \ln\Big(  \frac{c_{n,t}^{-1} \cZ_n^{1}(t,z)}{c_{n-1,t}^{-1} \cZ_{n-1}^{1}(t,z)} \Big),
\end{equation}
where we take the convention $\cZ^1_0 \equiv 1$ and the constant $c_{n,t}$ is as in equation \eqref{eq:constant}. Then the line ensemble  $\left\{ \cH_n^t(z) \right\}_{n\in \bN, z\in \bR}$ satisfies the requirements of being a KPZ$_t$ line ensemble as defined in Theorem 2.15 in \cite{CorHam15}. 
\end{cor}
\begin{proof}
In \cite{CorHam15}, the KPZ line ensemble was constructed by showing tightness and then extracting a subsequential limit from rescaled versions of the process $\cZ_d^{1,M}$ (see Theorem 3.9 and Lemma 5.1 in \cite{CorHam15}). Corollary \ref{cor:KPZ_line_ensemble_corr} identifies the finite dimensional distributions of this process, thereby showing that all the subsequential limits are the same, and identifying this unique limit. Following the construction of the KPZ line ensemble in Section 5 of \cite{CorHam15} gives $\cH_n^t$ as in equation \eqref{eq:KPZ_line_ensemble_log}.
\end{proof}

\begin{rem}
The main result, Theorem 2.15, of \cite{CorHam15} was to show the \emph{existence} of a line ensemble which satisfies the requirements of being a KPZ line ensemble. Corollary \ref{cor:unique_KPZ} gives an explicit formula for the line ensemble constructed in \cite{CorHam15} in terms of the partition functions from \cite{OConnellWarren2015} by the definition in equation \eqref{eq:KPZ_line_ensemble_log}. It is reasonable to believe that this line ensemble is the \emph{unique} line ensemble which satisfies the required properties of being a KPZ line ensemble, but this is currently unproven.
\end{rem}

\begin{cor} \label{cor:strict_positivity}
For fixed $t>0$ and $d \in \bN$, the continuum partition function $\cZ^1_d(t,z)$ is almost surely positive and continuous as function of $z \in \bR$
\end{cor}
\begin{proof}
This follows from Corollary \ref{cor:unique_KPZ} since the KPZ line ensemble $\cH_n^t$ from \cite{CorHam15}  is continuous.
\end{proof}

\begin{rem}
The strict positivity and continuity of $\cZ_d^1$ was first proven in \cite{LunWarren15} by different methods. Note that the result in \cite{LunWarren15} is more powerful since it also proves continuity as $t$ varies.
\end{rem}

\begin{cor} \label{cor:stationary}
Let $\cH_d^t(z)$ be the ensemble from equation \eqref{eq:KPZ_line_ensemble_log}. (which satisfies the requirements of being a KPZ$_t$ line ensemble as in Theorem 2.15 of \cite{CorHam15}). Then, for fixed $d\in\bN$,$t>0$, the stochastic process  $\cH_d^t(z)+\frac{z^2}{2t}$ indexed by $z \in \bR$ is stationary. 
\end{cor}
\begin{proof}
This follows by combining the identification from Corollary \ref{cor:unique_KPZ} with the fact that for fixed $t$, $\rho^{-d}(t,z)\cZ^{t,M}_d(t,z)$ indexed by $z \in \bR$ is stationary jointly over $d \in \bN$. The latter process is stationary because of the invariance of the white noise environment under the affine shift of coordinates $\big(t,z\big)\to\big(t,z-\zf\frac{t}{\tf}\big)$ and because, for all $d \in \bN$, this shift maps the non-intersecting Brownian bridges with endpoint $(\tf,\zf)$ to the non-intersecting Brownian bridges with endpoint $(\tf,0)$ in a measure-preserving way.
\end{proof}
\begin{rem}
Conjecture 2.17 of \cite{CorHam15} is that the rescaled KPZ line ensemble plus a parabola converges as $t \to \infty$ to the Airy line ensemble. Corollary \ref{cor:stationary} supports this conjecture since the Airy line ensemble is known to be stationary. In fact, a possible avenue of proof of this conjecture goes by showing that the Airy line ensemble is the \emph{unique} line ensemble that is stationary and possesses a non-intersecting Brownian Gibbs property. The result Corollary \ref{cor:stationary} is a required first step for this method; see Section 2.3.3 of \cite{CorHam15} for a full outline of this argument.
\end{rem}

\begin{cor} \label{cor:contour_int}
For any $t>0$, $k \in \bN$ and a list of indices $r_\al \in \bN$, $1\leq \al \leq k$ and list of coordinates $x_1 < \ld <x_k$, the joint moments of the continuum random polymer are given by the following explicit contour integrals:
\begin{align}
\cE\left[ \prod_{\al = 1}^k c^{-1}_{r_\al,t} \cZ_{r_\al}^1(t,x_\al) \right] =& \prod_{\al = 1}^k \frac{1}{(2\pi \imath)^{r_\al} r_\al!}  \int\cdots \int \prod_{1\leq \al \leq \be \leq k} \Bigg( \prod_{i=1}^{r_\al} \prod_{j=1}^{r_\be} \frac{z_{\al,i} - z_{\be,j}}{z_{\al,i} - z_{\be,j} - 1} \Bigg) \nonumber \\ 
&\times \prod_{\al = 1}^{k} \Bigg( \bigg( \prod_{i \neq j}^{r_\al} (z_{\al,i} - z_{\al,j}) \bigg) \bigg(\prod_{j=1}^{r_\al} e^{\half t z^2_{\al,j} + x_\al z_{\al,j}} \dd z_{\al,j} \bigg) \Bigg),\label{eq:contour_int} 
\end{align}
where the constants $c_{r_\al,t}$ are as in equation \eqref{eq:constant}, and the $z_{\al,j}$-contour is along $C_\al + \imath \bR$ for any constants $C_1 > C_2 + 1 > C_3 + 2 >\ld >C_k + (k-1)$ for all $j \in \{1,\ld,r_\al\}$ and $\cE$ denotes expectation with respect to the random environment. (Note that because of the ordering of the contours, this formula only holds when $x_1 \leq \ld \leq x_k$ as in the hypothesis.)
\end{cor}
\begin{proof}
Proposition 5.4.6. in \cite{Borodin2014} explicitly calculates the contour integral on the RHS of \ref{eq:contour_int} as the $M \to \infty$ limit for the joint moments for the process $\cZ_d^{t,M}$ defined in Corollary \ref{cor:KPZ_line_ensemble_corr}. Since the convergence in Definition \ref{def:KPZ_line} holds for finite dimensional distributions and moments, this establishes equation \eqref{eq:contour_int}.
\end{proof}
\begin{rem}
The result of Corollary \ref{cor:contour_int} was originally conjectured in Remark 5.4.7 of \cite{Borodin2014}. Note that the constants $c^{-1}_{r_\al,t}$ are absent in the original formulas from Remark 5.4.7 of \cite{Borodin2014} because, just as in Remark \ref{rem:missing_constants}, these constants were not known to appear in the convergence at the time. Corollary \ref{cor:contour_int} also validates the use of these moment formulas in the physics literature, see \cite{DeLuca15}. (Only the $r_{\al}=1$ formulas were used here, for which the missing constant has no effect since $c_{1,t} \equiv 1$.)
\end{rem}

%

\subsection{Outline}
Subsections \ref{sec:21} and \ref{sec:22} contain the precise definitions of the stochastic processes used throughout the paper. Subsections \ref{sec:23} and \ref{sec:24} contain still more definitions and lemmas that reduce the proof of Theorem \ref{thm:main_thm} to the convergence of certain chaos series; this proof is given in Subsection \ref{sec:pf_of_main_thm}. Subsection \ref{sec:pf_of_L2} contains the proof of the main technical result, Theorem \ref{thm:L2_convergence_psiN_to_psi}, with important estimates, Propositions \ref{prop:D1},
\ref{prop:D2}, \ref{prop:D3} and \ref{prop:D4}, deferred to later sections. Subsection \ref{sec:cor_pf} contains the asymptotic analysis needed to prove Corollary \ref{cor:KPZ_line_ensemble_corr}. Propositions \ref{prop:D1} and \ref{prop:D2} are proven in Section \ref{sec:det} using methods involving orthogonal polynomials. Propositions \ref{prop:D3} and \ref{prop:D4} are proven in Section \ref{sec:L2b} using the machinery of overlap times and weak exponential moment control developed in Section \ref{sec:overlap}.

\subsection{Notation} \label{sec:notation}

Let $\bN=\{1,2,\ld\}$. We use the letters $t\in(0,\infty)$, $z\in\bR$ to denote space-time coordinates for Brownian motions and the letters $\ta \in (0,\infty)$, $x\in\bN$ to denote space-time for discrete processes in continuous time. 

We will use the superscript $\star$ to denote quantities related to the endpoints of polymers; for example $(\tf,\zf)$ denotes the endpoint of non-intersecting Brownian bridges, $\taf$ denotes the final time for non-intersecting up/right paths, and $\xf$ denotes the vertical displacement of each up/right path. 	


For convenience of notation, we will conflate $k$-tuples of space-time coordinates  with their list of time and space coordinates,  i.e.  $\left\{ (t_1,z_1),\ld,(t_k,z_k) \right\}$ with $(\vec{t},\vec{z})$. In the same spirit, we use the following shorthand for integrals:
\begin{equation*}
\iintt{\vec{t} \in \De_k(a,b)}{\vec{z} \in (c,d)^k} f(\vec{t},\vec{z}) \dd \vec{t} \; \dd \vec{z} \defequal \intop_{\De_k(a,b)} \intop_{(c,d)^k} f\Big((t_1,z_1),\ld,(t_k,z_k)\Big) \dd z_1 \cdots \dd z_k \dd t_1 \cdots \dd t_k.
\end{equation*}
We also use a similar shorthand for $k$-fold stochastic integrals against a $1+1$ dimensional white noise environment $\xi(t,z)$, namely
\begin{equation*}
\iintt{\vec{t} \in \De_k(a,b)}{\vec{z} \in (c,d)^k} f(\vec{t},\vec{x}) \xi^{\otimes k}\big( \dd \vec{t}, \dd \vec{z}\big)\defequal \intop_{\De_k(a,b)}\intop_{(c,d)^k} f \Big((t_{1},z_{1}),{\ld},(t_{k},z_{k})\Big)\xi(\dd t_{1},\dd z_{1}){\cdots}\xi(\dd t_{k},\dd z_{k}).
\end{equation*}
For the semi-discrete coordinates that appear (where time is continuous but space is discrete) we use the following notation:
\begin{equation*}
\sintt{\vec{\ta} \in \De_k(a,b)}{\vec{x} \in \{c,\ld,d\}^k} f(\vec{\ta},\vec{x}) \dd \vec{\ta} \defequal \intop_{\De_k(a,b)} \sum_{\vec{x}\in \{c,\ld,d\}^k} f\Big((\ta_1,x_1),\cdots,(\ta_k,x_k)\Big) \dd \ta_1 \cdots \dd \ta_k. 
\end{equation*}
For stochastic integrals with respect to i.i.d. Brownian motions $\left\{B_1(t),B_2(t),\ldots \right\}$ we use the notation
\begin{align*}
&\sintt{\vec{\ta} \in \De_k(a,b)}{\vec{x} \in \{c,\ld,d\}^k}\quad f(\vec{\ta},\vec{x}) \dd B_{x_{1}}(\ta_{1})\cdots \dd B_{x_{k}}(\ta_{k}) \\
\defequal &\intop_{\De_k(a,b)} \sum_{\vec{x}\in \{c,\ld,d\}^k} f\Big((\ta_1,x_1),\cdots,(\ta_k,x_k)\Big) \dd B_{x_1}(\tau_1) \cdots \dd B_{x_k}(\ta_k). 
\end{align*}

We use the notation $\p,\e$ to refer to the probability measure and its expectation on
non-intersecting random walks	 defined precisely in Definitions \ref{def:NIP}
and \ref{def:NIPb}. In contrast, we will use the probability space
$\left(\Om,\cF,\cP\right)$ for the disordered environment that our
random walks go through and $\cE$ for the expectation with
respect to this random environment. The $L^{2}\left(\cP\right)$
norm for mean-zero random variables on this probability space is
\[
\norm Z_{L^{2}(\cP)}\defequal\cE\left(Z^{2}\right)^{\half}.
\]

We use $d\in\bN$ to denote the number of Brownian motions or up-right paths in the non-intersecting ensembles. 

\section{Definitions and Proof of Main Results}
\label{sec:pf}

\subsection{Non-intersecting Brownian motions and bridges } \label{sec:21}

\begin{defn}[Non-intersecting Brownian motions]\label{def:NIB}
Define the  $d$-dimensional Weyl chamber $\bW^d=\left\{ \vec{z}\in\bR^{d}:\ z_{1}\leq \ld \leq z_{d}\right\} $.  Let $\vec{D}(t)\in\bW^d$, $t\in(0,\infty)$ denote
an ensemble of $d$ non-intersecting Brownian motions and let $\e_{\vec{z}^{\nogt}}\left[\cdot\right]$, denote the expectation of this process started from $\vec{D}(0)=\vec{z}^{\nogt}$. More specifically, $\vec{D}(t)$ is the Markov process which is obtained from
$d$ independent Brownian motions by a Doob $h$-transform using the Vandermonde determinant 
\begin{equation}
h_{d}(\vec{z})\defequal\prod_{1\leq i<j\leq d}(z_{j}-z_{i}). \nonumber
\end{equation}
(See Section 3 of \cite{Warren_Dyson_Brownian_Motions} for details on this $h$-transform.) We will use the following fact about this process: for any continuous function $f:\bW^d\to\bR$ and any $\vec{z}^{\nogt}\in(\bW^d)^\circ$ we have that
\begin{eqnarray*}
\e_{\vec{z}^{\nogt}}\left[f\big(\vec{D}(t)\big)\right] & \defequal & \frac{1}{h_{d}\left(\vec{z}^{\nogt}\right)}\e\left[f\big(\vec{B}(t)+\vec{z}^{\nogt}\big)h_{d}\big(\vec{B}(t)+\vec{z}^{\nogt}\big)\one\left\{ \ta_{\vec{z}^{\nogt}}>t\right\} \right]\\
\ta_{\vec{z}^{\nogt}} & \defequal & \inf\left\{ t>0:\vec{B}(t)+\vec{z}^{\nogt}\notin\bW^d\right\},
\end{eqnarray*}
where $\vec{B}(t)$ are $d$ independent standard Brownian motions. 


\end{defn}

\begin{defn}[Non-intersecting Brownian bridges] 
\label{def:NIBb} Fix $\tf>0$ and $\zf\in\bR$. Let $\vec{D}^{(\tf,\zf)}(t)\in\bW^d$, $t\in[0,\tf]$ denote an ensemble of $d$
non-intersecting Brownian bridges, where each bridge starts at $D_i^{(\tf,\zf)}(0)=0$
and ends at $D_i^{(\tf,\zf)}(\tf)=\zf$, $1\leq i \leq d$. The process $\D^{(\tf,\zf)}$ is constructed by starting with the process $\vec{D}(t)\in\bW^d$ from Definition \ref{def:NIB} and applying the Markovian construction
of a bridge process. (See Proposition 1 of \cite{MarkovianBridges} and Section 2 of \cite{OConnellWarren2015} for more details.)

Let $\ps_{k}^{(\tf,\zf)}:\Big((t_{1},z_{1}),\ld,(t_{k},z_{k})\Big)\in\Big((0,\tf)\times\bR\Big)^{k} \to \bR$ denote the $k$-point correlation functions for this process. This is defined by:
\[
\ps_{k}^{(\tf,\zf)}\big((t_{1},z_{1}),\ld,(t_{k},z_{k})\big)\defequal\sum_{\vec{j}\in\left\{ 1,\ld,d\right\} ^{k}}\rho_{j_{1},\ld,j_{k}}\Big((t_{1},z_{1}),\ld,(t_{k},z_{k})\Big),
\] 
where $\rho_{j_{1},j_{2},\ld,j_{k}}\big((t_{1},z_{1}),\ld,(t_{k},z_{k})\big)$
is the probability density of the random vector $D_{j_{1}}^{(\tf,\zf)}(t_{1}),..,D_{j_{k}}^{(\tf,\zf)}(t_{k})$ with respect to Lebesgue measure on $\bR^{k}$ evaluated at the point $(z_{1},\ld,z_{k}) \in \bR^{k}$. 
\end{defn}

\begin{prop}
\label{prop:psi_k_L2}(Lemma 4.1 and Proposition 4.2 of \cite{OConnellWarren2015})
For any $\zf\in\bR,\tf>0, k\in\bN$, the function $\ps_{k}^{(\tf,\zf)}\in {L^{2}\left(\De_k(0,\tf)\times\bR^k\right)}$. Moreover for any $\be>0$, the following series is absolutely
convergent 
\[
1+\sum_{k=1}^{\infty}\be^{k}\norm{\ps_{k}^{(\tf,\zf)}}_{L^{2}\left( \De_k(0,\tf)\times\bR^k \right)}^{2}<\infty.
\]
\end{prop}
%

\subsection{Non-intersecting Poisson processes and non-intersecting Poisson bridges} \label{sec:22}
\begin{defn}[Non-intersecting Poisson processes]
\label{def:NIP}  We denote by $\X(\ta)\in\bN^{d}$, $\ta\in(0,\infty)$
an ensemble of $d$ \emph{non-intersecting Poisson processes} and use $\e_{\vec{x}^{0}}\left[\cdot\right]$
to denote the expectation over this ensemble started from the initial
condition $\vec{X}(0)=\vec{x}^{0}$. This is the Markov
process obtained by conditioning $d$ independent rate one Poisson processes not to
intersect by applying a Doob $h$-transform with the Vandermonde determinant $h_{d}(\x)\defequal\prod_{1\leq i<j\leq d}\left(x_{i}-x_{j}\right)$.
The transition probabilities are therefore given by
\[
\p\left(\X(\ta^{\pr})=\x^{\prime}\given{\X(\ta)=\x}\right)\defequal q_{\ta^{\pr}-\ta}\left(\x,\x^{\prime}\right)\frac{h_{d}\left(\x^{\prime}\right)}{h_{d}\left(\x\right)},
\]
where $q_{\ta}(\x,\y)$ is the probability for $d$ iid Poisson processes
to go from $\x$ to $\y$ in time $\ta$ without intersections. By
the Karlin-MacGregor theorem, introduced in \cite{kmcg}, this is given by
\begin{equation*}
q_{\ta}(\x,\x^{\pr}) \defequal \det\left[\mu\left(\ta,x_{i}^{\pr}-x_{j}\right)\right]_{i,j=1}^{d}, \;\;\; 
\mu(\ta,x) \defequal e^{-\ta}\frac{\ta^{x}}{x!}\one_{x\geq0}.
\end{equation*}

\end{defn}

\begin{defn}[Non-intersecting Poisson bridges]\label{def:NIPb} Fix $\xf\in\bN$
and $\taf\in\bR$ . For $x\in\bN\cup\{0\}$, define $\vec{\de}_{d}(x)\defequal(x+1,x+2,\ld,x+d)\in\bN^{d}$.
We denote by $\X^{(\taf,\xf)}(\ta)\in\bN^{d}\cap \bW^d,\ta\in(0,\taf)$ the 
ensemble of $d$ non-intersecting Poisson bridges that start at $\X^{(\taf,\xf)}(0)=\vec{\de}_{d}(0)$
and end at $\vec{X}^{(\taf,\xf)}(\taf)=\vec{\de}_{d}(\xf)$. The measure on these processes is
the conditional measure one gets by starting $d$ independent Poisson processes
from $\vec{\de}_{d}\left(0\right)$ and then conditioning on the positive
probability event that there have been no intersections between them
for all $\ta\in(0,\taf)$ and that they end exactly at $\vec{\de}_{d}\left(\xf\right)$ at time $\taf$. By the Karlin-MacGregor theorem, the transition probabilities for this Markov process are given explicitly by
\[
\p\left(\X^{(\taf,\xf)}(\ta^{\prime})=\x^{\prime}\given{\X^{(\taf,\xf)}(\ta)=\x}\right)=\frac{q_{\ta^{\pr}-\ta}\left(\x,\x^{\prime}\right)q_{\taf-\ta^{\pr}}\left(\x^{\prime},\vec{\de}_d(\xf)\right)}{q_{\taf-\ta}\left(\x,\vec{\de}_d(\xf)\right)}\ \forall\ta<\ta^{\prime},
\]
where $q_{\ta}(\vec{x},\vec{y})$ is as in Definition \ref{def:NIP}.
Comparing this to Definition \ref{def:NIP}, we see that $\vec{X}^{(\taf,\xf)}$ is absolutely continuous with respect $\X(\ta)$ started from $\vec{X}(0)=\vec{\de}_d(0)$ with Radon-Nikodym derivative given
by
\[
\frac{\p\left(\X^{(\taf,\xf)}(\ta)=\x\right)}{\p\left(\X(\ta)=\x\right)}=\frac{q_{\taf-\ta}\left(\x,\vec{\de}_{d}(\xf)\right)}{q_{\taf}\left(\vec{\de}_{d}(0),\vec{\de}_{d}(\xf)\right)}\frac{h_{d}\left(\vec{\de}_{d}(0)\right)}{h_{d}(\x)}.
\]

\end{defn}

\begin{rem} \label{rem:Poisson_bridges}
Due to Poisson processes conditioned on their final position taking the uniform measure, the measure on Poisson bridges $\vec{X}^{(\taf,\xf)}(\cdot)$
in Definition \ref{def:NIPb} is exactly the same as the measure proportional to the Lebesgue measure over
non-intersecting up/right paths described in Definition \ref{def:semi_discrete_parition}.
 It is more convenient to think of this process as a
Poisson bridge because the relationship to the 
non-intersecting Poisson process $\vec{X}(\cdot)$ from Definition \ref{def:NIP} is used as an intermediate step in the proof of our results.
\end{rem}

\subsection{Iterated stochastic integrals} \label{sec:23}
In this section we will show how the partition function $Z_d^\be(\taf,\xf)$ can be identified as a chaos series of iterated stochastic integrals against	 Brownian motions. 
\begin{defn}
\label{def:iterated_integrals}  Consider an infinite family $\left\{ B_{i}(\cdot)\right\} _{i=1}^{\infty}$
of independent standard Brownian motions on the probability space
$\left(\Om,\cF,\cP\right)$. For $\taf>0$ and $\xf\in\bN$ and any
ensemble of up/right paths $\X(\ta)\in\left\{ 1,\ld,\xf+d\right\} ^{d},\ \ta\in[0,\taf]$
define the $k$-fold stochastic integral $I_{k}^{\vec{X}}$ by:
\[
I_{k}\left(\vec{X}\right)(\ta)\defequal \sintt{\vec{\ta} \in \De_k(0,\ta)}{\vec{x} \in \{1,\ld,\xf+d\}^k}\quad \prod_{i=1}^{k}\one\left\{ x_{i}\in\vec{X}(\ta_{i})\right\} \dd B_{x_{1}}(\ta_{1})\ld \dd B_{x_{k}}(\ta_{k}),
\]
where we recall the notation for semi-discrete sums from Section \ref{sec:notation}. Let $\p$ denote the probability w.r.t. non-intersecting Poisson
bridges $\vec{X}^{(\taf,\xf)}$ described in Definition \ref{def:NIPb}.
Define the $k$-fold stochastic integral $EI_{k}^{\left(\taf,\xf\right)}$ by:
\begin{align*}
EI_{k}^{(\taf,\xf)}(\ta) & \defequal \sintt{\vec{\ta} \in \De_k(0,\ta)}{\vec{x} \in \{1,\ld,\xf+d\}^k}\quad \prod_{i=1}^{k}\p\left(\bigcap_{i=1}^{k}\left\{ x_{i}\in\vec{X}^{(\taf,\xf)}(\ta_{i})\right\} \right)\dd B_{x_{1}}(\ta_{1})\ld\dd B_{x_{k}}(\ta_{k}).
\end{align*}
\end{defn}
\begin{rem}
Note that from the theory of iterated stochastic integrals (see e.g. Chapter 7 of \cite{janson97}) , we have the following It\^{o} isometry between $L^2(\cP)$ and ${L^{2}\left(\De_k(0,\tf)\times \bN^k\right)}$ for these stochastic integrals (see e.g. Theorem 7.6 in \cite{janson97}):
\begin{align} \label{eq:Ito_isometry}
\cE \left[ I_k(\X)(\ta) I_j(\X)(\ta) \right]	&= \de_{j,k} \sintt{\vec{\ta} \in \De_k(0,\ta)}{\vec{x} \in \{1,\ld,\xf+d\}^k}\quad \prod_{i=1}^{k}\one\left\{ x_{i}\in\vec{X}(\ta_{i})\right\}^2 \dd \ta_{1}\ld \dd \ta_{k}, \\
\cE \left[ EI^{(\taf,\xf)}_k(\ta) EI^{(\taf,\xf)}_j(\ta) \right]	&= \de_{j,k} \sintt{\vec{\ta}\in\De_k(0,\ta)}{\vec{x}\in\{1,\ld,\xf+d\}^k}\quad \p\left(\bigcap_{i=1}^{k}\left\{ x_{i}\in\vec{X}^{(\taf,\xf)}(\ta_{i})\right\} \right)^2 \dd \ta_{1}\ldots \dd \ta_{k}. \nonumber
\end{align}
\end{rem}

\begin{lem}
\label{lem:L2_Ik}For any $\taf>0$ and $\xf\in\bN$ we have that
\[
\norm{I_{k}\left(\vec{X}^{(\taf,\xf)}\right)(\ta)}_{L^{2}(\cP)}^{2}=\frac{\left(d\ta\right)^{k}}{k!}.
\]
\end{lem}
\begin{proof}
This holds since $\sum_{x=1}^{\xf+d}\one\left\{ x\in\vec{X}^{(\taf,\xf)}(\ta)\right\} =d\ a.s$
for every fixed $\ta$ and by an application of the It\^{o} isometry from equation \eqref{eq:Ito_isometry}.
\end{proof}

\begin{lem}
\label{lem:bounded_series}For any $\taf>0$, $\xf\in\bN$ and $\be > 0$ we have
that
\[
\sum_{k=0}^{\infty}\be^{k}\norm{EI_{k}^{(\taf,\xf)}(\ta)}_{L^{2}(\cP)}^{2} \leq \exp\left(d\be\ta\right).
\]
\end{lem}

\begin{proof}
By using the It\^{o} isometry from equation \eqref{eq:Ito_isometry}, and the inequality from Corollary \ref{cor:bound_on_sum_of_P2},
we can now bound the $L^{2}(\cP)$ norm by the $k$-th moment of the
overlap time random variable which is specified in Definition \ref{def:overlap_times_NIW}:
\begin{equation*}
\sum_{k=0}^{\infty}\be^{k}\norm{EI_{k}^{(\taf,\xf)}(\ta)}_{L^{2}(\cP)}^{2} \leq\sum_{k=0}^{\infty}\frac{\be^{k}}{k!}\e\left[\left(O^{(\taf,\xf)}[0,\ta]\right)^{k}\right] =\e\left[\exp\left(\be O^{(\taf,\xf)}[0,\ta]\right)\right].
\end{equation*}
The change of order of sum and expectation is justified by the monotone
convergence theorem since the overlap time is always non-negative.
The result then follows by the simple bound that $O^{(\taf,\xf)}[0,\ta]\leq d\ta$
which is immediate from Definition \ref{def:overlap_times_NIW}.
\end{proof}

\begin{lem}
\label{lem:interchange_E}Let $\e$ denote the expectation over non-intersecting
Poisson bridges  $\vec{X}^{(\taf,\xf)}$ described in Definition \ref{def:NIPb}. Recall the definition of $I_k$ and $EI_k$ from Definition \ref{def:iterated_integrals}. We have the following equality (as random variables in $L^{2}(\cP)$):
\[
\e\left[I_{k}\left(\vec{X}^{(\taf,\xf)}\right)(\ta)\right]=EI_{k}^{\left(\taf,\xf\right)}(\ta).
\]
\end{lem}
\begin{proof}
The result amounts to an interchange of the order of the stochastic integration and the expectation $\e$. This is justified by a stochastic Fubini theorem for multiple stochastic integrals: see Theorem 5.13.1 in \cite{Peccati2011}. The required integrability condition is clear in this case since the integrand, $\prod_{i=1}^{k}\one\left\{ x_{i}\in\vec{X}(\ta_{i})\right\}$, is non-negative and bounded above by $1$.
\end{proof}

\begin{lem}
\label{lem:interchange_E_chaos}We have the following equality (as
random variables in $L^{2}(\cP)$):
\begin{equation}
\e\left[\sum_{k=0}^{\infty}\be^{k}I_{k}\left(\vec{X}^{(\taf,\xf)}\right)(\ta)\right]=\sum_{k=0}^{\infty}\be^{k}EI_{k}^{(\taf,\xf)}(\ta).\label{eq:sum_interchange}
\end{equation}
\end{lem}
\begin{proof}
First notice that the infinite series from \eqref{eq:sum_interchange}  is guaranteed to converge by the estimate from Lemma \ref{lem:bounded_series}. To see the equality, we will show that given any $\ep>0$, the difference between the LHS
and the RHS of equation (\ref{eq:sum_interchange}) has an $L^{2}(\cP)$
norm less than $\ep$. Given such an $\ep>0$, we first find an $M\in\bN$
so that $\norm{\e\left[\sum_{k=M}^{\infty}\be^{k}I_{k}\left(\vec{X}^{(\taf,\xf)}\right)(\ta)\right]}_{L^{2}(\cP)}<\ep$
and $\norm{\sum_{k=M}^{\infty}\be^{k}EI_{k}^{(\taf,\xf)}(\ta)}_{L^{2}(\cP)}<\ep$.
This can be achieved since we have
\begin{equation*}
\norm{\e\left[\sum_{k=M}^{\infty}\be^{k}I_{k}\left(\vec{X}^{(\taf,\xf)}\right)(\ta)\right]}_{L^{2}(\cP)}^{2} \leq \sum_{k=M}^{\infty}\e\left[\cE\left[\left(\be^{k}I_{k}\left(\vec{X}^{(\taf,\xf)}\right)(\ta)\right)^{2}\right]\right],
\end{equation*}
by an application of Jensen's inequality, Tonelli's theorem, and the fact that the individual terms $I_k$ are orthogonal in $L^2(\cE)$. 
Thus we can find such an $M\in\bN$ to bound this above by $\ep$, since we recognize this is as the tail of an absolutely convergent series by an application of Lemma \ref{lem:L2_Ik}. A similar result holds for $\norm{\sum_{k=M}^{\infty}\be^{k}EI_{k}^{(\taf,\xf)}(\ta)}_{L^{2}(\cP)}$
since the stochastic integrals $\left\{ EI_{k}^{(\taf,\xf)}\right\} _{k\in\bN}$are
orthogonal in $L^{2}(\cP)$, and since the sum $\sum_{k=0}^{\infty}\be^{2k}\norm{EI_{k}^{(\taf,\xf)}}_{L^{2}(\cP)}^{2}$
also is convergent by Lemma \ref{lem:bounded_series}. Once such an
$M$ is chosen, we have by the triangle inequality and Lemma \ref{lem:interchange_E} applied to the first $M$ terms that
\begin{equation*}
\norm{\e\left[\sum_{k=0}^{\infty}\be^{k}I_{k}\left(\vec{X}^{(\taf,\xf)}\right)(\ta)\right]-\sum_{k=0}^{\infty}\be^{k}EI_{k}^{(\taf,\xf)}(\ta)}_{L^{2}(\cP)} \leq0+\ep+\ep.
\end{equation*}
Since this holds for any $\ep>0$, this completes the proof.
\end{proof}

\begin{lem}
\label{lem:chaos_series_for_Z}Recall the semi-discrete polymer partition function $Z_{d}^{\be}$ from Definition \ref{def:semi_discrete_parition}.
It is possible to realize the partition functions $Z_{d}^{\be}$ as
the following infinite chaos series:
\[
Z_{d}^{\be}\left(\taf,\xf\right)\exp\left(-\frac{d}{2}\taf\be^{2}\right)=\sum_{k=0}^{\infty}\be^{k}EI_{k}^{(\taf,\xf)}.
\]
\end{lem}
\begin{proof}
Recall the definition of the energy of the $i$-th line $H\left(X_{i}^{(\taf,\xf)}\right)$ from equation \eqref{eq:energy_def}. For any fixed path $X_{i}$, we notice by the It\^{o} isometry that $H\left(X_{i}^{(\taf,\xf)}\right)$ is a Gaussian random variable of mean $0$ and variance $\taf$. Moreover,
by the non-intersecting condition, the energies $H\left(X_{i}^{(\taf,\xf)}\right)$
and $H\left(X_{j}^{(\taf,\xf)}\right)$ are independent when $i\neq j$,
so the sum $\be\sum_{i=1}^{d}H\left(X_{i}^{(\taf,\xf}\right)$
is a Gaussian with mean 0 and variance $d\be^{2}\taf$. We now apply
the relationship between exponential of Gaussians and the Wick exponential,
$:e^{\xi}:=e^{\xi}e^{-\e[\xi^{2}/2]}$ (see Theorem 3.3 and the definition of the Wick exponential in \cite{janson97}). This gives

\begin{align*}
\e\left[\exp\left(\be\sum_{i=1}^{d}H\left(X_{i}^{(\taf,\xf)}\right)\right)\right]\exp\left(-\frac{d}{2}\be^{2}\taf\right) & =\e\left[:\exp:\left(\be\sum_{i=1}^{d}H\left(\vec{X}_i^{(\taf,\xf)}\right)\right)\right].
\end{align*}
Since each $H\left(X_{i}^{(\taf,\xf)}\right)$ is a single stochastic
integral, the Wick exponential is given by the chaos series (see Theorem 7.3 from \cite{janson97}):
\[
:\exp:\left(\be\sum_{i=1}^{d}H\left(\vec{X}_i^{(\taf,\xf)}\right)\right)=\sum_{k=0}^{\infty}\be^{k}I_{k}\left(\vec{X}^{(\taf,\xf)}\right)(\ta).
\]
The desired result then follows by application of the interchange of infinite sum and expectation from Lemma \ref{lem:interchange_E_chaos}.
\end{proof}

\subsection{1+1 dimensional white noise} \label{sec:24}
In this section we will couple the semi-discrete partition function $Z_d^\beta$ to the continuum polymer $\cZ_d^\beta$. This is achieved by constructing the Brownian motions that define $Z_d^\beta$ from integrals of the white noise environment (see \cite{janson97} for the background on these integrals). This coupling approach is also used in \cite{MSRInotestwo} in the proof of convergence of the single-line (i.e. $d=1$) semi-discrete polymer to the continuum random polymer. 


 \begin{figure} 
\begin{center}
\begin{tikzpicture}[yscale = 0.8,
declare function={ my(\x,\y)                 = \y*\scaley+\x*\slopex+\offy; 
                  },xscale = 4,yscale=0.5
]			
			
			\node[anchor = west] at (1,5) (Left) {};
			\node[anchor = east] at (0+\offx,5) (Right) {};
			\draw (Left) edge[out=60,in=120,->] node[draw=none,midway,above] {$\vp^{(N)}$} (Right) ;
			
			\node[anchor = north, gray] at (0,1) (noname) {$0$};
			\node[anchor = north, gray] at (1,1) (noname) {$\taf$};

			\foreach \y in {1,2,3,4}
   			    \draw[thin,gray](1,\y) -- (0,\y)
   			     node[anchor=east] {$\y$};
   			     
   			\foreach \y in {5,6}
   			    \draw[thin,gray](1,\y) -- (0,\y)
   			     node[anchor=east] {$\vdots$};

   			\foreach \y in {7}
   			    \draw[thin,gray](1,\y) -- (0,\y)
   			     node[anchor=east] {$\xf$};   			
   			
   			\foreach \y in {1,2,3}
   			    \draw[thin,gray](1,7+\y) -- (0,7+\y)
   			     node[anchor=east] {$\xf + \y$};

   			
   			    
   			     
   			
   			\draw [black,thick] (0,1) -- (\xAa,1);
   			\draw [black,dotted] (\xAa,1) -- (\xAa,2);
   			\draw [black,thick] (\xAa,2) -- (\xAb,2);
   			\draw [black,dotted] (\xAb,2) -- (\xAb,3);
   			\draw [black,thick] (\xAb,3) -- (\xAc,3);
   			\draw [black,dotted] (\xAc,3) -- (\xAc,4);
   			\draw [black,thick] (\xAc,4) -- (\xAd,4);
   			\draw [black,dotted] (\xAd,4) -- (\xAd,5);
   			\draw [black,thick] (\xAd,5) -- (\xAe,5);
   			\draw [black,dotted] (\xAe,5) -- (\xAe,6);
   			\draw [black,thick] (\xAe,6) -- (\xAf,6);
   			\draw [black,dotted] (\xAf,6) -- (\xAf,7);
   			\draw [black,thick] (\xAf,7) -- (\xAg,7);
   			\draw [black,dotted] (\xAg,7) -- (\xAg,8);
   			\draw [black,thick] (\xAg,8) -- (1,8);
   			
   			\draw [black,thick] (0,2) -- (\xBa,2);
   			\draw [black,dotted] (\xBa,2) -- (\xBa,3);
   			\draw [black,thick] (\xBa,3) -- (\xBb,3);
   			\draw [black,dotted] (\xBb,3) -- (\xBb,4);
   			\draw [black,thick] (\xBb,4) -- (\xBc,4);
   			\draw [black,dotted] (\xBc,4) -- (\xBc,5);
   			\draw [black,thick] (\xBc,5) -- (\xBd,5);
   			\draw [black,dotted] (\xBd,5) -- (\xBd,6);
   			\draw [black,thick] (\xBd,6) -- (\xBe,6);
   			\draw [black,dotted] (\xBe,6) -- (\xBe,7);
   			\draw [black,thick] (\xBe,7) -- (\xBf,7);
   			\draw [black,dotted] (\xBf,7) -- (\xBf,8);
   			\draw [black,thick] (\xBf,8) -- (\xBg,8);
   			\draw [black,dotted] (\xBg,8) -- (\xBg,9);
   			\draw [black,thick] (\xBg,9) -- (1,9);
   			
   			\draw [black,thick] (0,3) -- (\xCa,3);
   			\draw [black,dotted] (\xCa,3) -- (\xCa,4);
   			\draw [black,thick] (\xCa,4) -- (\xCb,4);
   			\draw [black,dotted] (\xCb,4) -- (\xCb,5);
   			\draw [black,thick] (\xCb,5) -- (\xCc,5);
   			\draw [black,dotted] (\xCc,5) -- (\xCc,6);
   			\draw [black,thick] (\xCc,6) -- (\xCd,6);
   			\draw [black,dotted] (\xCd,6) -- (\xCd,7);
   			\draw [black,thick] (\xCd,7) -- (\xCe,7);
   			\draw [black,dotted] (\xCe,7) -- (\xCe,8);
   			\draw [black,thick] (\xCe,8) -- (\xCf,8);
   			\draw [black,dotted] (\xCf,8) -- (\xCf,9);
   			\draw [black,thick] (\xCf,9) -- (\xCg,9);
   			\draw [black,dotted] (\xCg,9) -- (\xCg,10);
   			\draw [black,thick] (\xCg,10) -- (1,10);
   			   
%
			\node[anchor = north, gray] at ({0+\offx},{\scaley*1+\offy+1*\slopex}) (noname_c) {$0$};
			\node[anchor = north, gray] at ({1+\offx},{\scaley*1+\offy+1*\slopex}) (noname_d) {$\tf$};
			\node[anchor = west, gray] at ({1+\offx},{\scaley*8+\offy+1*\slopex}) (noname_c) {$\zf$};

			\foreach \y in {1,2,3,4,5,6,7,8,9,10}
				\draw[thin,gray](1+\offx,\scaley*\y+\offy+1*\slopex) -- (0+\offx,\scaley*\y+\offy+0*\slopex)
 			    node[anchor=east] {};
 			    
 			\draw [black,thick] (0+\offx,\scaley*1+\offy+0*\slopex) -- (\xAa+\offx,\scaley*1+\offy+\xAa*\slopex);
   			\draw [black,dotted] (\xAa+\offx,\scaley*1+\offy+\xAa*\slopex) -- (\xAa+\offx,\scaley*2+\offy+\xAa*\slopex);
   			\draw [black,thick] (\xAa+\offx,\scaley*2+\offy+\xAa*\slopex) -- (\xAb+\offx,\scaley*2+\offy+\xAb*\slopex);
            \draw [black,dotted] (\xAb+\offx,\scaley*2+\offy+\xAb*\slopex) -- (\xAb+\offx,\scaley*3+\offy+\xAb*\slopex);
            \draw [black,thick] (\xAb+\offx,\scaley*3+\offy+\xAb*\slopex) -- (\xAc+\offx,\scaley*3+\offy+\xAc*\slopex);
            \draw [black,dotted] (\xAc+\offx,\scaley*3+\offy+\xAc*\slopex) -- (\xAc+\offx,\scaley*4+\offy+\xAc*\slopex);
            \draw [black,thick] (\xAc+\offx,\scaley*4+\offy+\xAc*\slopex) -- (\xAd+\offx,\scaley*4+\offy+\xAd*\slopex);
            \draw [black,dotted] (\xAd+\offx,\scaley*4+\offy+\xAd*\slopex) -- (\xAd+\offx,\scaley*5+\offy+\xAd*\slopex);
            \draw [black,thick] (\xAd+\offx,\scaley*5+\offy+\xAd*\slopex) -- (\xAe+\offx,\scaley*5+\offy+\xAe*\slopex);
            \draw [black,dotted] (\xAe+\offx,\scaley*5+\offy+\xAe*\slopex) -- (\xAe+\offx,\scaley*6+\offy+\xAe*\slopex);
            \draw [black,thick] (\xAe+\offx,\scaley*6+\offy+\xAe*\slopex) -- (\xAf+\offx,\scaley*6+\offy+\xAf*\slopex);
            \draw [black,dotted] (\xAf+\offx,\scaley*6+\offy+\xAf*\slopex) -- (\xAf+\offx,\scaley*7+\offy+\xAf*\slopex);
            \draw [black,thick] (\xAf+\offx,\scaley*7+\offy+\xAf*\slopex) -- (\xAg+\offx,\scaley*7+\offy+\xAg*\slopex);
            \draw [black,dotted] (\xAg+\offx,\scaley*7+\offy+\xAg*\slopex) -- (\xAg+\offx,\scaley*8+\offy+\xAg*\slopex);
            \draw [black,thick] (\xAg+\offx,\scaley*8+\offy+\xAg*\slopex) -- (1+\offx,\scaley*8+\offy+1*\slopex);
 			     
 			 \draw [black,thick] (0+\offx,\scaley*2+\offy+0*\slopex) -- (\xBa+\offx,\scaley*2+\offy+\xBa*\slopex);
   			\draw [black,dotted] (\xBa+\offx,\scaley*2+\offy+\xBa*\slopex) -- (\xBa+\offx,\scaley*3+\offy+\xBa*\slopex);
   			\draw [black,thick] (\xBa+\offx,\scaley*3+\offy+\xBa*\slopex) -- (\xBb+\offx,\scaley*3+\offy+\xBb*\slopex);
            \draw [black,dotted] (\xBb+\offx,\scaley*3+\offy+\xBb*\slopex) -- (\xBb+\offx,\scaley*4+\offy+\xBb*\slopex);
            \draw [black,thick] (\xBb+\offx,\scaley*4+\offy+\xBb*\slopex) -- (\xBc+\offx,\scaley*4+\offy+\xBc*\slopex);
            \draw [black,dotted] (\xBc+\offx,\scaley*4+\offy+\xBc*\slopex) -- (\xBc+\offx,\scaley*5+\offy+\xBc*\slopex);
            \draw [black,thick] (\xBc+\offx,\scaley*5+\offy+\xBc*\slopex) -- (\xBd+\offx,\scaley*5+\offy+\xBd*\slopex);
            \draw [black,dotted] (\xBd+\offx,\scaley*5+\offy+\xBd*\slopex) -- (\xBd+\offx,\scaley*6+\offy+\xBd*\slopex);
            \draw [black,thick] (\xBd+\offx,\scaley*6+\offy+\xBd*\slopex) -- (\xBe+\offx,\scaley*6+\offy+\xBe*\slopex);
            \draw [black,dotted] (\xBe+\offx,\scaley*6+\offy+\xBe*\slopex) -- (\xBe+\offx,\scaley*7+\offy+\xBe*\slopex);
            \draw [black,thick] (\xBe+\offx,\scaley*7+\offy+\xBe*\slopex) -- (\xBf+\offx,\scaley*7+\offy+\xBf*\slopex);
            \draw [black,dotted] (\xBf+\offx,\scaley*7+\offy+\xBf*\slopex) -- (\xBf+\offx,\scaley*8+\offy+\xBf*\slopex);
            \draw [black,thick] (\xBf+\offx,\scaley*8+\offy+\xBf*\slopex) -- (\xBg+\offx,\scaley*8+\offy+\xBg*\slopex);
            \draw [black,dotted] (\xBg+\offx,\scaley*8+\offy+\xBg*\slopex) -- (\xBg+\offx,\scaley*9+\offy+\xBg*\slopex);
            \draw [black,thick] (\xBg+\offx,\scaley*9+\offy+\xBg*\slopex) -- (1+\offx,\scaley*9+\offy+1*\slopex);
            
             \draw [black,thick] (0+\offx,\scaley*3+\offy+0*\slopex) -- (\xCa+\offx,\scaley*3+\offy+\xCa*\slopex);
   			\draw [black,dotted] (\xCa+\offx,\scaley*3+\offy+\xCa*\slopex) -- (\xCa+\offx,\scaley*4+\offy+\xCa*\slopex);
   			\draw [black,thick] (\xCa+\offx,\scaley*4+\offy+\xCa*\slopex) -- (\xCb+\offx,\scaley*4+\offy+\xCb*\slopex);
            \draw [black,dotted] (\xCb+\offx,\scaley*4+\offy+\xCb*\slopex) -- (\xCb+\offx,\scaley*5+\offy+\xCb*\slopex);
            \draw [black,thick] (\xCb+\offx,\scaley*5+\offy+\xCb*\slopex) -- (\xCc+\offx,\scaley*5+\offy+\xCc*\slopex);
            \draw [black,dotted] (\xCc+\offx,\scaley*5+\offy+\xCc*\slopex) -- (\xCc+\offx,\scaley*6+\offy+\xCc*\slopex);
            \draw [black,thick] (\xCc+\offx,\scaley*6+\offy+\xCc*\slopex) -- (\xCd+\offx,\scaley*6+\offy+\xCd*\slopex);
            \draw [black,dotted] (\xCd+\offx,\scaley*6+\offy+\xCd*\slopex) -- (\xCd+\offx,\scaley*7+\offy+\xCd*\slopex);
            \draw [black,thick] (\xCd+\offx,\scaley*7+\offy+\xCd*\slopex) -- (\xCe+\offx,\scaley*7+\offy+\xCe*\slopex);
            \draw [black,dotted] (\xCe+\offx,\scaley*7+\offy+\xCe*\slopex) -- (\xCe+\offx,\scaley*8+\offy+\xCe*\slopex);
            \draw [black,thick] (\xCe+\offx,\scaley*8+\offy+\xCe*\slopex) -- (\xCf+\offx,\scaley*8+\offy+\xCf*\slopex);
            \draw [black,dotted] (\xCf+\offx,\scaley*8+\offy+\xCf*\slopex) -- (\xCf+\offx,\scaley*9+\offy+\xCf*\slopex);
            \draw [black,thick] (\xCf+\offx,\scaley*9+\offy+\xCf*\slopex) -- (\xCg+\offx,\scaley*9+\offy+\xCg*\slopex);
            \draw [black,dotted] (\xCg+\offx,\scaley*9+\offy+\xCg*\slopex) -- (\xCg+\offx,\scaley*10+\offy+\xCg*\slopex);
            \draw [black,thick] (\xCg+\offx,\scaley*10+\offy+\xCg*\slopex) -- (1+\offx,\scaley*10+\offy+1*\slopex);

		\end{tikzpicture}
		\end{center}
		\caption{The map $\vp^{(N)}$ sends non-intersecting paths $X^{(\taf,\xf)}$ to their rescaled and compensated version $X^{(N),(\tf,\zf)}$ when $\taf = N\tf$ and $\xf = \floor{N\tf + \sqrt{N}\zf}$. The vertical spacing between lines in the image is $N^{-\half}$.} \label{fig:vp}
		\end{figure}

\begin{defn}
\label{def:semi_discrete_tesselation}Define the map $\vp^{(N)}:(0,\infty)\times\bN\to(0,\infty)\times\bR$
by: 
\[
\vp^{(N)}(\ta,x)\defequal\left(\frac{\ta}{N},\frac{x-\ta}{\sqrt{N}}\right).
\]
and denote by $\bS^{(N)}$ the image of $(0,\infty)\times\bN$ through
this map:
\[
\bS^{(N)}\defequal\vp^{(N)}\left((0,\infty)\times\bN\right)\subset(0,\infty)\times\bR.
\]
See Figure \ref{fig:vp} for an illustration of this map. Also define the intervals
\[
I^{(N)}(t,z)\defequal t\times\left[z,z+\frac{1}{\sqrt{N}}\right).
\]

\end{defn}
Any function $f:\left(\bS^{(N)}\right)^{k}\to\bR$ can be extended
to a function $f:\left((0,\infty)\times\bR\right)^{k}\to\bR$ by declaring
that $f$ is constant on cells of the form $I^{(N)}(t_{1},z_{1})\times\ld\times I^{(N)}(t_{k},z_{k})$
for every $\Big((t_{i},z_{i}),\ld(t_{k},z_{k})\Big)\in\bS^{(N)}$.
Note that, since $f$ is constant on these cells, we have
\begin{equation}
\iintt{\vec{t}\in\De_{k}(a,b)}{\vec{z}\in \bR^k} f(\vec{t},\vec{z})\dd \vec{t} \dd \vec{z} = N^{-\frac{3k}{2}} \sintt{\vec{\ta}\in\De_{k}(Na,Nb)}{\vec{x}\in\bN^k}f\left(\Big(\frac{\ta_{1}}{N},\frac{x_{1}-\ta_{1}}{\sqrt{N}}\Big),\ld\Big(\frac{\ta_{k}}{N},\frac{x_{k}-\ta_{k}}{\sqrt{N}}\Big)\right) \dd \vec{\ta}.
\label{eq:integral_is_semidiscrete_sum}
\end{equation}


\begin{lem}
\label{lem:coupling} Let $\xi(t,z)$ be a 1+1 dimensional white noise
field. For each
$N\in\bN$, we may couple an infinite family of iid standard Brownian
motions $\left\{ B_{x}^{(N)}\right\} _{x\in\bN}$ to the white noise
field $\xi$ by the prescription that:
\begin{equation}
B_{x}^{(N)}(\ta)\defequal N^{3/4}\cdot\iint_{\vp^{(N)}\big((0,\ta)\times [x,x+1)\big)}\xi(\dd t,\dd z).\label{eq:coupling}
\end{equation}

\end{lem}

\begin{proof}
By the It\^{o} isometry, since the area of the region $\vp^{(N)}\big((\ta,\ta^\prime)\times [x,x+1)\big)$ is $N^{-3/2}(\ta^\prime - \ta)$, we can make the following variance computation:
\[
\var\left(B_{x}^{(N)}(\ta^{\prime})-B_{x}^{(N)}(\ta)\right)=\left(N^{3/4}\right)^{2} N^{-3/2} (\ta^{\prime} - \ta) =\ta^{\prime}-\ta.
\]
By properties of the 1+1 dimensional white noise, we also observe that the integral on the RHS of equation (\ref{eq:coupling}) defines
a Gaussian random variable, that the increments for disjoint time intervals are independent, and that the process $B_{x}^{(N)}(\cdot)$
admits a modification which has almost surely continuous sample paths. Hence it must be that $B_{x}^{(N)}$ is a Brownian
motion, as desired. The fact that $B_{x}^{(N)}$ is independent of $B_{x^{\prime}}^{(N)}$ for
$x\neq x^{\prime}$ is clear because the regions $\vp^{(N)}\big((0,\ta)\times [x,x+1)\big)$ and $\vp^{(N)}\big((0,\ta)\times [x^\prime,x^\prime+1)\big)$, define disjoint
regions in the integration.\end{proof}
\begin{defn}
\label{def:rescaled_NIPb}For $\tf>0$ and $\zf\in\bR$, we will define
the rescaled (and compensated) non-intersecting Poisson processes
by:
\[
\X^{(N),(\tf,\zf)}(t)\defequal\frac{1}{\sqrt{N}}\left( \X^{(N\tf,\floor{N\tf+\sqrt{N}\zf})}\left(Nt\right) - Nt \right).
\]
See Figure \ref{fig:vp} for an illustration of these processes. Define the rescaled $k$-point correlation function for $\ps_{k}^{(N),(\tf,\zf)}:\left(\bS^{(N)}\right)^{k}\to\bR$
by defining for $k$-tuples $\Big((t_{1},z_{1}),\ld,(t_{k},z_{k})\Big)\in\left(\bS^{(N)}\right)^{k}$
where all the entries $(t_{i},z_{i})$ are distinct: 
\begin{align}
\ps_{k}^{(N),(\tf,\zf)}\left(\vec{t},\vec{z}\right) & \defequal\sqrt{N}^{k}\sum_{\vec{j}\in\left\{ 1,\ld,d\right\} ^{k}}\p\left(\bigcap_{i=1}^{k}\left\{ z_{i}=X_{j_{i}}^{(N),(\tf,\zf)}(t_{i})\right\} \right)\label{eq:def_ps_k_N}\\
 & =\sqrt{N}^{k}\p\left(\bigcap_{i=1}^{k}\left\{ z_{i}\in\X^{(N),(\tf,\zf)}(t_{i})\right\} \right),\nonumber 
\end{align}
and declaring that $\ps_{k}^{(N),(\tf,\zf)}\defequal0$ if any of
the space time coordinates are duplicated $(t_{i},z_{i})=(t_{j},z_{j})$
for $i\neq j$. We extend the domain of $\ps_{k}^{(N),(\tf,\zf)}$
to all of $\left((0,\tf)\times\bR\right)^{k}$ as in Definition \ref{def:semi_discrete_tesselation}
by declaring it to be constant on the cells $I^{(N)}(t_{1},z_{1})\times\ld\times I^{(N)}(t_{k},z_{k})$
for every $\left((t_{i},z_{i}),\ld,(t_{k},z_{k})\right)\in\bS^{(N)}$.
Notice that because $\ps_{k}^{(N),(\tf,\zf)}$ is constant on these
cells, any integral of $\ps_{k}^{(N),(\tf,\zf)}$ can be decomposed
into a semi-discrete sum as in equation (\ref{eq:integral_is_semidiscrete_sum}).
\end{defn}

\begin{lem}
\label{lem:integral_for_EI}Fix $\tf>0$ and $\zf>0$. Let $\xi$
be a 1+1 dimensional white noise environment. Recall that for each
$N\in\bN,$ we may couple an infinite collection of iid Brownian
motions $B_{x}^{(N)}$ to this probability space by the prescription from Lemma \ref{lem:coupling}.
If we use the Brownian motions $B_{x}^{(N)}$ to define the iterated
stochastic integral $EI_{k}^{(N\tf,\floor{N\tf+\sqrt{N}\zf})}$from
Definition \ref{def:iterated_integrals} , then we have
\[
EI_{k}^{(N\tf,\floor{N\tf+\sqrt{N}\xf})}(Nt)=N^{\frac{1}{4}k}\iintt{\vec{t}\in\De_k(0,t)}{\vec{z}\in\bR^{k}}\ps_{k}^{(N),(\tf,\zf)}\left(\vec{t},\vec{z}\right)\xi^{\otimes k}\left(\dd \vec{t},\dd \vec{z}\right).
\]
\end{lem}
\begin{proof}
The identity is immediate from Definition \ref{def:iterated_integrals}
and Definition \ref{def:rescaled_NIPb} using the fact that $\ps_{k}^{(N),(\tf,\zf)}$
is constant on intervals of the form $I^{(N)}\left(t,z\right).$\end{proof}
\begin{rem}
Note that the power of $N^{\oo 4k}$ arises from multiplying the $N^{\frac{3}{4}}$
from the definition of the Brownian motion in the coupling Lemma \ref{lem:coupling},
and the $N^{-\oo 2}$ in the rescaling of $\ps_{k}^{(\taf,\xf)}$
in Definition \ref{def:rescaled_NIPb}. 
\end{rem}

\subsection{Convergence of chaos series -- Proof of Theorem \ref{thm:main_thm}} \label{sec:pf_of_main_thm}

In addition to the $L^2$ convergence result of  Theorem \ref{thm:L2_convergence_psiN_to_psi}, and the set up of the coupling from the previous subsection, we will need the following proposition:

\begin{prop}
\label{prop:uniform_exp_moms} For any $\ga>0$, we have
\[
\sup_{N\in \bN}\sum_{k=1}^{\infty}\ga^{k} \norm{\ps_k^{(N),(\tf,\zf)}}^2_{L^2(\De_k(0,\tf)\times\bR^k)}<\infty.
\]
Moreover, for any $\ga>0$ and any $\ep>0$, there exists $N_{\ga,\ep}$
so that we have that
\[
\limsup_{\ell\to\infty}\sup_{N>N_{\ga,\ep}}\sum_{k=\ell}^{\infty}\ga^{k} \norm{\ps_k^{(N),(\tf,\zf)}}^2_{L^2(\De_k(0,\tf)\times\bR^k)}<\ep.
\]
\end{prop}

The proof of Proposition \ref{prop:uniform_exp_moms} is deferred to Section \ref{sec:L2b}, where it is proven using tools developed in Section \ref{sec:overlap}.

\begin{proof}
(Of Theorem \ref{thm:main_thm}) We will explicitly construct the
coupling for which the convergence happens in $L^{2}\left(\cP\right)$;
the convergence in distribution is an immediate consequence, and we will separately argue the convergence in $L^p(\cP)$ for $p \neq 2$ afterwards. We present the proof only for fixed $d \in \bN$, $\tf \in (0,\infty)$, $\zf \in \bR$, but the method of proof easily extends to finite dimensional distributions of the process by considering finite linear combinations  and using the Cramer-Wold device. 

Couple the random variables $\cZ_{d}^{\be}$ and $Z_{d}^{\be_{N}}$ by taking
the Brownian motions $\left\{ B_{x}^{(N)}\right\} _{x\in\bZ}$ (which
define $Z_{d}^{\be_{N}}$) to be as defined as in the coupling from Lemma \ref{lem:coupling}, and define for each $k\in\bN$ the $k$-fold stochastic integrals:
\begin{equation*}
J_{k}^{(N)} \defequal \iintt{\vec{t}\in \De_k(0,\tf)}{\vec{z}\in\bR^k}\ps_{k}^{(N),(\tf,\zf)}\left(\vec{t},\vec{z}\right)\xi^{\otimes k}\left(\dd \vec{t},\dd \vec{z}\right), \quad J_{k} \defequal \iintt{\vec{t}\in\De_k(0,\tf)}{\vec{z}\in\bR^{k}}\ps_{k}^{(\tf,\zf)}\left(\vec{t},\vec{z}\right)\xi^{\otimes k}\left(\dd \vec{t},\dd \vec{z}\right)
\end{equation*}
With this setup, by combining Lemma \ref{lem:chaos_series_for_Z}
and Lemma \ref{lem:integral_for_EI}, and by equation \eqref{eq:cZ_def}, we recognize the quantities of interest as the following infinite series:
\begin{equation}
\frac{Z_{d}^{\be_{N}}\left(N\tf,\floor{N\tf+\sqrt{N}\zf}\right)}{\exp\left(\half dN\tf\be_N^{2}\right)} =\sum_{k=0}^{\infty}\big(N^{\oo{4}} \be_N\big)^{k}J_{k}^{(N)}, \quad \frac{\cZ_{d}^{\be}(\tf,\zf)}{\rho(\tf,\zf)^d} =\sum_{k=0}^{\infty}\be^{k}J_{k}. \label{eq:target}
\end{equation}
The desired result is hence reduced to the convergence as $N \to \infty$ of the chaos series in equation \eqref{eq:target}. It suffices to show the convergence in the simpler case when $N^{\oo{4}}\be_N = \be$, since the hypothesis $\be_N N^{\oo{4}} \to \be$ and Proposition \ref{prop:uniform_exp_moms} guarantee the error made by this replacement can be made arbitrarily small.

Notice by the It\^{o} isometry for 1+1 dimensional white noise that these
stochastic integrals are naturally related to the $L^{2}\left(\De_k(0,\tf)\times \bR^k\right)$
norm of the functions $\ps_{k}^{(\tf,\zf)}$ and $\ps_{k}^{(N),(\tf,\zf)}$, namely:
\begin{equation}
\norm{J_{k}^{(N)}-J_{k}}_{L^{2}(\cP)}=\norm{\ps_{k}^{(N),(\tf,\zf)}-\ps_{k}^{(\tf,\zf)}}_{L^{2}\left(\De_k(0,\tf)\times \bR^k\right)}. \label{eq:termwise_conv}
\end{equation}
It is clear then for each fixed $k\in\bN$, that $J_k^{(N)} \to J_k$ in $L^2(\cP)$ by the convergence result from Theorem \ref{thm:L2_convergence_psiN_to_psi}. It remains only to justify why the convergence of each individual term in the series in equation (\ref{eq:target}) gives convergence of the full series. 

To see this, take any $\ep>0$, and use the convergence results of Propositions \ref{prop:uniform_exp_moms}
and Proposition \ref{prop:psi_k_L2} along with the It\^{o} isometry to find $N_{\be,\ep}\in\bN$
and $\ell_{\be,\ep}\in\bN$ so large so that
\begin{align*}
\sup_{N>N_{\be,\ep}}\sum_{k=\ell_{\be,\ep}}^{\infty}\norm{\be^{k}J_{k}^{(N)}}_{L^{2}(\cP)}^{2}=\sup_{N>N_{\be,\ep}}\sum_{k=\ell_{\be,\ep}}^{\infty}\be^{2k}\norm{\ps_{k}^{(N),(\tf,\zf)}}_{L^{2}\left(\De_k(0,\tf)\times\bR^k\right)}^{2} & <\ep,\\
\sum_{k=\ell_{\be,\ep}}^{\infty}\norm{\be^{k}J_{k}}_{L^{2}(\cP)}^{2}=	\sum_{k=\ell_{\be,\ep}}^{\infty}\be^{2k}\norm{\ps_{k}^{(\tf,\zf)}}_{L^{2}\left(\De_k(0,\tf)\times\bR^k\right)}^{2} & <\ep.
\end{align*}
With this choice, we have finally by the triangle inequality in $L^{2}(\cP)$,
and the termwise convergence observed in equation \eqref{eq:termwise_conv},
that
\begin{align*}
\limsup_{N\to\infty}\norm{ \sum_{k=0}^{\infty} \be^k J^{(N)}_k -\sum_{k=0}^{\infty}\be^k J_k}_{L^{2}(\cP)} \leq& \limsup_{N\to\infty}\norm{ \sum_{k=\ell_{\be,\ep}}^{\infty} \be^k J^{(N)}_k -\sum_{k=\ell_{\be,\ep}}^{\infty}\be^k J_k}_{L^{2}(\cP)} \\
  \leq& \sup_{N>N_{\be,\ep}} \norm{\sum_{k=\ell_{\be,\ep}}^{\infty}\be^{k}J_{k}^{(N)}}_{L^{2}(\cP)}+\norm{\sum_{k=\ell_{\be,\ep}}^{\infty}\be^{k}J_{k}}_{L^{2}(\cP)} \\
  <& 2\ep.
\end{align*}
Since this holds for any $\ep>0$, we have the desired convergence in $L^2(\cP)$.

The $L^2(\cP)$ convergence proven directly implies $L^p(\cP)$ convergence for $1\leq p \leq 2$. To see the convergence in $L^p(\cP)$ for $p > 2$, we first use the hypercontractive property of a fixed Wiener chaos to see that there is a constant $c_p$ so that stochastic integrals $J_k^{(N)}$ and $J_k$ have\footnote{The same idea is used in a similar setting in Lemma 2.3 in \cite{LunWarren15} and can be proven by applying the Burkholder-Davis-Gundy inequality $k$ times. See also Theorem 5.10 in \cite{janson97}.}: 
\begin{equation}
\norm{J_{k}^{(N)}}^2_{L^{p}(\cP)}\leq c_p^k \norm{\ps_{k}^{(N),(\tf,\zf)}}_{L^{2}\left(\De_k(0,\tf)\times\bR^k\right)}^{2}, \quad  \norm{J_{k}}^2_{L^{p}(\cP)}\leq c_p^k \norm{\ps_{k}^{(\tf,\zf)}}_{L^{2}\left(\De_k(0,\tf)\times\bR^k\right)}^{2}.
\end{equation}
Hence, the infinite series $\sum_{j=1}^\infty \be^k J_k$ is seen to have finite $L^p(\cP)$ norm by Proposition \ref{prop:psi_k_L2}, and by Proposition \ref{prop:uniform_exp_moms}, we can find $N_{p,\be} \in \bN$ large enough so that the infinite series from equation \eqref{eq:target} has $L^p$ norm which is uniformly bounded:
\begin{equation} \label{eq:Lp_bound}
\sup_{N > N_{p,\be}} \norm{\sum_{k=0}^\infty \be^k J_{k}^{(N)}}_{L^{p}(\cP)}\leq \sup_{N > N_{p,\be}} \sum_{k=0}^\infty (\be c_p)^k \norm{\ps_{k}^{(N),(\tf,\zf)}}_{L^{2}\left(\De_k(0,\tf)\times\bR^k\right)}^{2} < \infty.
\end{equation} 
Since these norms are finite for any $p$, we now apply the Holder inequality in the form
\begin{equation*}
\norm{ \sum_{k=0}^{\infty} \be^k J^{(N)}_k -\sum_{k=0}^{\infty}\be^k J_k}_{L^{p}} \leq \norm{ \sum_{k=0}^{\infty} \be^k J^{(N)}_k -\sum_{k=0}^{\infty}\be^k J_k}_{L^{2}}^{1/p} \norm{ \sum_{k=0}^{\infty} \be^k J^{(N)}_k -\sum_{k=0}^{\infty}\be^k J_k}_{L^{2(p-1)}}^{(p-1)/p},
\end{equation*}
from which the $L^{p}(\cP)$ convergence follows from the $L^2(\cP)$ convergence and the uniform bound on the $L^{2(p-1)}$ norm in equation \eqref{eq:Lp_bound}. 
\end{proof}

\subsection{Proof of Corollary \ref{cor:KPZ_line_ensemble_corr}}
\label{sec:cor_pf}

\begin{lem} \label{lem:Leb_measure}
Recall from Definition \ref{def:KPZ_line} that $D^{(M)}_d(t) \subset \bR^{(M-d)d}$ denotes the set of jump times for $d$ non-intersecting up/right paths that start at $X_i(0) = i$ and end at $X_i(t) = M - d + i$ for $1\leq i\leq d$. The Lebesgue measure of this set is:
\begin{equation} \label{eq:Leb_measure}
\cL\left( D_d^{(M)}(t) \right) = t^{(M-d)d}\prod_{i=0}^{d-1}\frac{i!}{\left(M-d+i\right)!}.
\end{equation}
\end{lem}

\begin{proof}
The jump times of such a non-intersecting ensemble are in bijection with pairs $(S,\vec{t})$, where $\vec{t} \in \De_{(M-d)d}(0,t)$  is the ordered list of \emph{all} the jump times for the ensemble, and $S$ is a standard Young tableau of rectangular shape $d$ by $M-d$ that indicates which jump times correspond to which path (i.e. the first row of $S$ indicates the  jump times of the top most up/right path and so on; see Section 2 of \cite{Nica2016} for more details on this bijection). The Lebesgue measure of $\De_{(M-d)d}(0,t)$ is $t^{(M-d)d}(d(M-d)!)^{-1}$ and the number of standard tableaux of this shape is $(d(M-d))!\prod_{i=0}^{d-1}\frac{i!}{(M-d+i)!}$ (this is a direct application of the hook-length formula: see Corollary 7.2.14 in \cite{stanley2001} or \cite{A060854}). Combining these gives the desired result. 
\end{proof}
\begin{lem} \label{lem:KPZ_exp_convergence}
Fix any $d \in \bN$, $t>0$, and $z\in \bR$. Let $C(M,t,z)$, $c_{d,t}^{-1}$ and $\cZ_d^{t,M}$ be as in Corollary \ref{cor:KPZ_line_ensemble_corr}. Then:
\begin{equation} \label{eq:moment_convergence}
\lim_{M\to \infty}\cE\left( \cZ_d^{t,M} \right) = c_{d,t}^{-1} \rho(t,z)^d.
\end{equation}
\end{lem}
\begin{proof}
Let $\sim$ denote asymptotic equality as $M\to \infty$, i.e. $f(M) \sim g(M)$ if and only if $f(M)=g(M)(1+o(1))$ as $M \to \infty$. By using the connection from equation \eqref{eq:oldZ_newZ_relation} between $\cZ_d^{t,M}$ and $Z^\be_d(\taf,\xf)$ when $\be = 1$, $\taf = \sqrt{tM} +  z$ and $\xf = M -d$ we have then:


\begin{align*}
\cE\left( \cZ_d^{t,M} \right) & =\cL\left(D_{d}^{M}(\sqrt{tM}+z)\right)\cdot\exp\left(\frac{d}{2}(\sqrt{tM}+z)\right)C(M,t,z)^{-d} \\
 & =\left(\sqrt{tM}+z\right)^{(M-d)d}\prod_{i=0}^{d-1}\frac{i!}{\left(M-d+i\right)!}\exp\left(-dM-dz\frac{\sqrt{M}}{\sqrt{t}}\right)\left(\frac{\sqrt{M}}{\sqrt{t}}\right)^{Md},
\end{align*}
where the factor of $\exp(\frac{d}{2}(\sqrt{tM}+z) )$ is canceled due to contributions from $C(M,t,z)^{-d}$. We now use $\left(\sqrt{tM}+z\right)^{(M-d)d}\sim\sqrt{tM}^{(M-d)d}\exp\left(d z\frac{\sqrt{M}}{\sqrt{t}}\right)\exp\left(-\frac{z^{2}}{2t}\right)^{d}$ to further simplify:
\begin{align*}
\cE\left( \cZ_d^{t,M} \right)
 & \sim\sqrt{tM}^{(M-d)d}\exp\left(-\frac{z^{2}}{2t}\right)^{d}\prod_{i=0}^{d-1}\frac{i!}{(M-d+i)!}\exp\left(-dM\right)\left(\frac{\sqrt{M}}{\sqrt{t}}\right)^{Md} \\
& = \frac{1}{t^{\half d^2} M^{\half d(d+1)}\sqrt{2\pi}^{d}} \left(\sqrt{2\pi M}M^{M}\exp\left(-M\right)\right)^{d}\exp\left(-\frac{z^{2}}{2t}\right)^{d}\prod_{i=0}^{d-1}\frac{i!}{(M-d+i)!}.
\end{align*}
We now use Stirling's formula $\sqrt{2\pi M}M^{M}\exp\left(-M\right)\sim M!$ to get:
\begin{equation*}
\cE\left( \cZ_d^{t,M} \right) \sim \frac{1}{t^{\half d^2} M^{\half d(d+1)}\sqrt{2\pi}^{d}}\exp\left(-\frac{z^{2}}{2t}\right)^{d}\prod_{i=0}^{d-1}\frac{M!i!}{(M-d+i)!}.	
\end{equation*}
Since $M!((M-j)!)^{-1} \sim M^{j}$, we see that the powers of $M$ vanish, and we remain with the desired result of equation \eqref{eq:moment_convergence}.
\end{proof}
\begin{proof} (Of Corollary \ref{cor:KPZ_line_ensemble_corr}) Define the scaling parameter $N$, and the parameter $\beta_N$ by
\begin{align*}
N &\defequal \frac{M}{t} - \frac{z}{t^{3/2}}\sqrt{M}\\
\beta^2_N &\defequal \sqrt{\frac{t}{M}}
\end{align*}
With this definition, the limit $N \to \infty$ is the same as the limit $M \to \infty$. Notice also that $\lim_{N\to \infty}N^{\oo{4}} \be_N = 1$ by this definition.  Define also the shorthands $\taf \defequal \sqrt{tM} + z$, $\xf \defequal M-d$ and $y \defequal \frac{\xf - tN}{\sqrt{N}}$. Observe the following limit as $M\to\infty$:
\begin{equation}
\label{eq:y_limit}
\lim_{M \to \infty} y = \lim_{M \to \infty} \frac{M - d -tN}{\sqrt{N}} = \lim_{M \to \infty} \frac{-d + \frac{z}{\sqrt{t}} \sqrt{M}}{\sqrt{\frac{M}{t} + \frac{z}{t^{3/2}}\sqrt{M}}} = z.  
\end{equation}
Now, by the relation from equation \eqref{eq:oldZ_newZ_relation}, we have that:	
\begin{align*}
\cZ_d^{t,M}(z) &= C(M,t,z)^{-d} \cL \left( D^{(M)}_d(\taf) \right) Z^1_d(\taf,\xf) \\
&= \cE\left( \cZ^{t,M}_d \right) Z^1_d(\taf,\xf)\exp\Big(-\frac{d}{2}\taf\Big).
\end{align*}
We now use the following general Gaussian scaling relation for $Z_d^{\be}$, that
$Z_d^{1}(s,y) \dequal Z_d^{\be}\big(\be^{-2}s,y\big)$. This holds because rescaling both time and the inverse temperature parameter leaves the chaos series invariant. We can hence write
\begin{align*}
\cZ_d^{t,M}(z) &\dequal \cE\left( \cZ^{t,M}_d \right) Z^{\be_N}_d(\be^{-2}_{N}\taf,\xf)\exp\Big(-\frac{d}{2}\frac{\taf}{tN}tN \Big) \\
&= \cE\left( \cZ^{t,M}_d \right) Z^{\be_N}_d(tN,\lfloor tN+y\sqrt{N} \rfloor)\exp\Big(-\frac{d}{2}\be^2_NtN \Big), 
\end{align*}
since $\beta_N^{-2} \taf = tN$. The desired result then follows by the convergence of equation \eqref{eq:y_limit}, Theorem \ref{thm:main_thm} and Lemma \ref{lem:KPZ_exp_convergence}.
\end{proof}

\subsection{\label{sub:L2} \texorpdfstring{$L^{2}$}{L2} convergence -- Proof of Theorem \ref{thm:L2_convergence_psiN_to_psi}}
\label{sec:pf_of_L2}

The main technical result that was needed in the proof of Theorem \ref{thm:main_thm} was the $L^{2}$ convergence
from Theorem \ref{thm:L2_convergence_psiN_to_psi}. This is proven by a similar general strategy as the proof of Theorem 1.13 from \cite{CorwinNica16}, which was an analogous convergence result for non-intersecting simple random walks rather than non-intersecting Poisson processes. There are several additional complications in this case due to the fact that the processes here evolve in continuous time. The proof goes by dividing the set of space-time coordinates $\De_k(0,\tf)\times\bR^k$ into four parts and analyzing contribution to the $L^2$ norm on each one separately.

\begin{defn} \label{def:space-time} Fix any $\tf>0$ and $k \in \bN$. For any $\et>0$, define the set $S_\et \subset \De_k(0,\tf)$ by:
\[
S_\et \defequal \Big\{ \vec{t} \in \De_k(0,\tf) : t_{i+1} - t_{i} > \et \: \forall \: 1\leq i\leq k-1 \Big\}
\]
For any parameters $\de,\et,M>0$, we subdivide $\De_k(0,\tf)\times \bR^k$ into the following four sets:
\begin{eqnarray*}
D_{1}(\de,\et,M) & \defequal & \Big\{ (\vec{t},\vec{z}) : \vec{t} \in \De_k(\de,\tf-\de)\cap S_\et , \vec{z} \in [-M,M]^k \Big\} \\
D_{2}(\de,\et,M) & \defequal & \Big\{ (\vec{t},\vec{z}) : \vec{t} \in \De_k(\de,\tf-\de)\cap S^c_\et , \vec{z} \in [-M,M]^k \Big\} ,\\
D_{3}(\de) & \defequal & \Big\{ (\vec{t},\vec{z}) : \vec{t} \in \De_k(0,\tf) \setminus \De_k(\de,\tf-\de), \vec{z} \in \bR^k \Big\} \\
D_{4}(M) & \defequal & \Big\{ (\vec{t},\vec{z}) : \vec{t} \in \De_k(0,\tf) , \vec{z} \in \bR^k \setminus [-M,M]^k \Big\} .
\end{eqnarray*}
The set $D_{1}(\de,\et,M)$ can be thought of as the ``typical'' part of the space $\De_k(0,\tf) \times \bR^k$
while the sets $D_{2}(\de,\et,M)$, $D_{3}(\de)$ and $D_{4}(M)$ can
be thought of as ``exceptional'' sets. This subdivision is chosen in order to make
$D_{1}(\de,\et,M)$ a compact set on which the function $\ps_{k}^{(N),(\tf,\zf)}$
has no singularities. This essentially reduces $L^{2}$ convergence on $D_{1}(\de,\et,M)$ to proving pointwise convergence on $D_{1}(\de,\et,M)$. All
of the singularities/non-compactness issues occur on the exceptional
sets $D_{2}(\de,\et,M)$, $D_{3}(\de)$ and $D_{4}(M)$ where we will
separately argue that they have a negligible contribution to the $L^2$ norm in (\ref{eq:L2_thm}). With this strategy in mind, the core of the proof of Theorem \ref{thm:L2_convergence_psiN_to_psi} is divided
into Propositions \ref{prop:D1}, \ref{prop:D2}, \ref{prop:D3}
and \ref{prop:D4}; each Proposition handles one of these four sets. \end{defn}

\begin{rem}
These propositions are very similar to Propositions 2.19, 2.20, 2.21 and 2.22 from \cite{CorwinNica16}. Here $\ps_k^{(N),(\tf,\zf)}$ is the $k$-point correlation function for Poisson processes, while in \cite{CorwinNica16} simple symmetric random walks were studied. Propositions \ref{prop:D1} and \ref{prop:D2} are proven in Section \ref{sec:det} using methods involving orthogonal polynomials. Propositions \ref{prop:D3} and \ref{prop:D4} are proven in Section \ref{sec:L2b} using the machinery of overlap times and weak exponential moment control developed in Section \ref{sec:overlap}.
\end{rem}

\begin{prop}
\label{prop:D1} Fix $\tf > 0, \zf \in \bR$. For all $\de,\et,M>0$,
we have pointwise convergence
\[
\lim_{N\to\infty}\ps_{k}^{(N),(\tf,\zf)}\big(\vec{t},\vec{z}\big)=\ps_{k}^{(\tf,\zf)}\big(\vec{t},\vec{z}\big),
\]
and the convergence is uniform over all $\big(\vec{t},\vec{z}\big)\in D_{1}\left(\de,\et,M\right)$. Moreover, there is a constant $C_{D_{1}}=C_{D_{1}}(\de,\et,M)$
so that for all $\big(\vec{t},\vec{z}\big)\in D_{1}\left(\de,\et,M\right)$ we
have:
\[
\sup_{N}\ps_{k}^{(N),(\tf,\zf)}\big(\vec{t},\vec{z}\big)\leq C_{D_{1}},\ \ \ps_{k}^{(\tf,\zf)}\big(\vec{t},\vec{z}\big)\leq C_{D_{1}}.
\]
\end{prop}
\begin{prop}
\label{prop:D2} Fix $\tf > 0, \zf \in \bR$. For all $\de,\ep,M>0$, there exists $\et>0$ small enough so that:
\[
\limsup_{N\to\infty}\iintop_{D_{2}(\de,\et,M)}\abs{\ps_{k}^{(N),(\tf,\zf)}\big(\vec{t},\vec{z}\big)}^{2}\dd\vec{t}\dd\vec{z}<\ep.
\]
\end{prop}
\begin{prop}
\label{prop:D3} Fix $\tf > 0, \zf \in \bR$. For all $\ep>0$, there exists $\de>0$ small enough so that:
\begin{equation*}
\limsup_{N\to\infty}\iintop_{D_{3}(\de)}\abs{\ps^{(N),(\tf,\zf)}\big(\vec{t},\vec{z}\big)}^{2}\dd\vec{t}\dd\vec{z}< \ep
\end{equation*}
\end{prop}

\begin{prop}
\label{prop:D4} Fix $\tf > 0, \zf \in \bR$. For all $\ep>0$, there exists $M>0$ large enough so that:
\begin{equation}
\limsup_{N\to\infty}\iintop_{D_{4}(M)}\abs{\ps_{k}^{(N),(\tf,\zf)}\big(\vec{t},\vec{z}\big)}^{2}\dd\vec{t}\dd\vec{z}<\ep\label{eq:D4_target}
\end{equation}

\end{prop}
\begin{proof}
(Of Theorem \ref{thm:L2_convergence_psiN_to_psi}) Once Propositions \ref{prop:D1}, \ref{prop:D2}, \ref{prop:D3} and \ref{prop:D4} are established, the proof is identical to the proof of Theorem 1.13 in \cite{CorwinNica16}. The strategy of the proof goes by first choosing $\de,\et,M>0$ so that the contribution on the exceptional
sets $D_{2}(\de,\et,M)$, $D_{3}(\de)$ and $D_{4}(M)$ are less than
$\ep$, and once $\de,\et,M$ are fixed, using Proposition \ref{prop:D1} to establish $L^{2}$ convergence on the typical set $D_{1}(\de,\et,M)$.
\end{proof}

\section{Determinantal Kernels and Orthogonal Polynomials} 

\label{sec:det}

In this section we prove Proposition \ref{prop:D1} and Proposition \ref{prop:D2} by using the determinantal structure of the non-intersecting processes. The methods used here are similar to those from Section 3 of \cite{CorwinNica16}. In \cite{CorwinNica16}, non-intersecting simple symmetric random walk bridges, for which the Hahn orthogonal polynomials arise, were studied. Here we study non-intersecting Poisson bridges, for which the Krawtchouk orthogonal polynomials arise. The limiting object for both are non-intersecting Brownian bridges, which are related to the Hermite polynomials.

 
\subsection{Determinantal kernel for non-intersecting Brownian bridges}

We recall some useful definitions and facts about non-intersecting Brownian bridges which were given in detail in Section 3.1. from \cite{CorwinNica16}.

\begin{defn}
\label{def:K}(Definition 3.2 from \cite{CorwinNica16}) Fix any $\tf>0$. For $t\in(0,\tf)$, define
$\dfac_{t}\defequal\sqrt{\frac{\tf}{2t(\tf-t)}}$. For $z,z^{\prime}\in\bR$
and $t,t^{\prime}\in(0,\tf)$, define the kernel $K^{(\tf,0)}\Big((t,z);(t^{\prime},z^{\prime})\Big)$
by:

\begin{align}
\label{eq:NIBb_kernel}
& K^{(\tf,0)}\Big((t,z);(t^{\prime},z^{\prime})\Big) \nonumber \\ 
 \defequal & -\frac{1}{\sqrt{2\pi\left(t^{\prime}-t\right)}}\exp\left(-\frac{(z^{\prime}-z)^{2}}{2\left(t^{\prime}-t\right)}\right)\one\left\{ t<t^{\prime}\right\} \nonumber \\
 & +\left(\frac{\tf}{2t^{\prime}(\tf-t)}\right)^{\half}\;\sum_{j=0}^{d-1}\left(\frac{t(\tf-t^{\prime})}{(\tf-t)t^{\prime}}\right)^{j/2}p_{j}(z\dfac_{t})\exp\left(-\frac{z^{2}}{2(\tf-t)}\right)p_{j}(z^{\prime}\dfac_{t^{\prime}})\exp\left(-\frac{z^{\prime2}}{2t^{\prime}}\right).
\end{align}
where $p_{j}(y)$, $j\in\bN$, $y\in\bR$ are the normalized Hermite polynomials:
\begin{eqnarray}
p_{j}(y) & \defequal & \frac{H{}_{j}(y)}{\sqrt{\sqrt{\pi}\cdot j!\cdot2^{j}}},\label{eq:Hermite_normal}\\
H{}_{j}(y) & \defequal & (-1)^{j}e^{y^{2}}\frac{\dd^{j}}{\dd y^{j}}e^{-y^{2}}.\nonumber 
\end{eqnarray}
Finally, for any $\zf\in\bR$, we will define:
\begin{equation}
K^{(\tf,\zf)}\Big((t,z);(t^{\prime},z^{\prime})\Big)\defequal K^{(\tf,0)}\left(\Big(t,z-\zf\frac{t}{\tf}\Big);\Big(t^{\prime},z^{\prime}-\zf\frac{t^{\prime}}{\tf}\Big)\right)\frac{\exp\Big(\frac{\zf}{\tf}\big(z^{\prime}-\zf\frac{t^{\prime}}{\tf}\big)\Big)}{\exp\Big(\frac{\zf}{\tf}\big(z-\zf\frac{t}{\tf}\big)\Big)}.\label{eq:K_zf_tf}
\end{equation}
\end{defn}
\begin{lem}[Lemma 3.3 from \cite{CorwinNica16}]
\label{lem:psi_is_det_K}  Fix any $\tf>0$ and $\zf\in\bR$. Recall
from Definition \ref{def:NIBb} the $k$-point correlation functions
$\ps_{k}^{(\tf,\zf)}$ for non-intersecting Brownian bridges $\vec{D}^{(\tf,\zf)}$.
We have that $\ps_{k}^{(\tf,\zf)}$ is determinantal with kernel $K^{(\tf,\zf)}$:
\begin{equation}
\ps_{k}^{(\tf,\zf)}\Big((t_{1},z_{1}),\ld,(t_{k},z_{k})\Big)=\det\left[K^{(\tf,\zf)}\Big((t_{i},z_{i});(t_{j},z_{j})\Big)\right]_{i,j=1}^{k}.\label{eq:ps_is_detK}
\end{equation}
\end{lem}


\subsection{Determinantal kernel for non-intersecting Poisson bridges}
\begin{defn}
\label{def:Krawtchouk-polynomials}The Krawtchouk polynomials are
a family of orthogonal polynomials parametrized by the two parameters
$N\in\bN$ and $p\in(0,1)$ and given explicitly in terms of the hypergeometric
function $_{2}F_{1}$ by
\[
K_{j}\left(x;p,N\right)={}_{2}F_{1}\binom{-j,-x}{-N}\left(\frac{1}{p}\right).
\]
The first few polynomials are:
\begin{align*}
K_0(x;p,N) &= 1, \\
K_1(x;p,N) &= -Np + x,\\
K_2(x;p,N) &= \half \left( N^2 p^2 + x(2p + x - 1) - N p (p+2x) \right). 
\end{align*}
See \cite{ks98} for more details on the Krawtchouk polynomials. Fix
$\taf>0$ and $\xf\in\bN$. For any $\ta>0,x\in\bN$, and any $0\leq j\leq d-1$, define
$R_{j}^{(\taf,\xf)}(\ta,x)$ and $\tilde{R}_{j}^{(\taf,\xf)}(\ta,x)$
which are defined in terms of the Krawtchouk with parameters depending
on $\ta,x,\taf,\xf$ by: 
\begin{align*}
R_{j}^{(\taf,\xf)}(\ta,x) & \defequal  K_{j}\left(x,\frac{\ta}{\taf},\xf+d-1\right).\\
\tilde{R}_{j}^{(\taf,\xf)}(\ta,x) & \defequal K_{j}\left(\xf-x+d-1,1-\frac{\ta}{\taf},\xf+d-1\right).
\end{align*}
(We will refer to $R_j$ and $\tilde{R}_j$ without the superscripts for ease of notation whenever there is no ambiguity in this). Finally define the kernel $K_P^{(\tf,\xf)}$ for pairs of space-time coordinates $(\ta,x)\in (0,\tf)\times \bN$, $(\ta^\prime,x^\prime)\in (0,\tf)\times \bN$ as follows. (Recall that $\mu(\ta,x) \defequal e^{-\ta}{\ta^{x}}{x!}^{-1} \one_{x\geq0}$ is the Poisson probability mass function.)
\begin{align}
& K_{P}^{(\taf,\xf)}\Big((\ta,x);(\ta^{\prime},x^{\prime})\Big) \nonumber \\
\defequal& \mu\left(\ta^{\pr}-\ta,x^{\prime}-x\right)\one\left\{ \ta<\ta^{\prime}\right\} \nonumber \\ 
 &+\sum_{j=0}^{d-1}\frac{(-1)^{j}\binom{\xf+d-1}{j}}{\mu\left(\taf,\xf+d-1\right)}\tilde{R}_{j}(\ta,x)\mu(\taf-\ta,\xf-x+d-1)R_{j}(\ta^{\prime},x^{\prime})\mu(\ta^{\prime},x^{\prime}), \label{eq:NIPb_kernel}
\end{align}
and define the rescaled version of this, for a pair of space-time coordinates $(t,z)\in (0,\tf)\times \bR$, $(t^\prime,z^\prime)\in (0,\tf)\times \bR$:	
\begin{align*}
& K^{(N),(\tf,\zf)}\Big((t,z);(t^{\prime},z^{\pr})\Big) \\
\defequal & \sqrt{N}K_{P}^{(N\tf,\floor{N\tf+\sqrt{N}\zf})}\left(\big(Nt,\floor{Nt+\sqrt{N}z}\big);\big(Nt^{\prime},\floor{Nt^{\prime}+\sqrt{N}z^{\prime}}\big)\right).
\end{align*}
\end{defn}
\begin{lem}
\label{lem:psiN_is_det_KN}
Fix any $\tf>0$ and $\zf\in\bR$. Recall
from Definition \ref{def:rescaled_NIPb} the $k$-point correlation functions $\ps_{k}^{(N),(\tf,\zf)}$ for rescaled non-intersecting Poisson bridges $\vec{X}^{(N),(\tf,\zf)}$. We have that $\ps_{k}^{(N),(\tf,\zf)}$ is determinantal with kernel $K^{(N),(\tf,\zf)}$:

\begin{align}
\ps_{k}^{(N),(\tf,\zf)}\Big((t_{1},z),\ld,(t_{k},z_{k})\Big) &= \det\left[K^{(N),(\tf,\zf)}\Big(\left(t_{i},z_{i}\right);\left(t_{j},z_{j}\right)\Big)\right]_{i,j=1}^{k}.\label{eq:psN_is_det_KN}
\end{align}
\end{lem}
\begin{proof}
It suffices to show that $K_{P}^{(\taf,\xf)}$ is the determinantal
kernel for non-intersecting Poisson bridges that start at $\vec{\de}_d(0)$
and end at $\vec{\de}_d(\xf)$, and the result for $K^{(N),(\tf,\zf)}$
follows from the rescaling in the definitions of $\ps_{k}^{(N),(\tf,\xf)}$
from Definition \ref{def:rescaled_NIPb}. The proof that
$K_{P}^{(\taf,\xf)}$ is this determinantal kernel is deferred to the
appendix, Proposition \ref{prop:KP_is_the_kernel}.
\end{proof}

\subsection{Pointwise convergence on \texorpdfstring{$D_{1}(\protect\de,\protect\et,M)$}{D1(delta,eta,M)} -- Proof of Proposition \ref{prop:D1} }
\begin{lem}
\label{lem:kraw_to_hermite}Fix any $0<\de<1$, and $L>0$. Let $p\in\left[\de,1-\de\right]$,
and $y\in[-L,L]$. Suppose that $\left\{ y_{M}\right\} _{M=1}^{\infty}$
is a sequence so that $y_{M}=pM+y\sqrt{2p(1-p)M}+o(\sqrt{M})$ as
$M\to\infty$ and moreover suppose that the $O(\sqrt{M})$ error is
uniform over all choices of $p,y$ with $p\in\left[\de,1-\de\right]$
and $y\in[-L,L]$. Define for any $j\in\bN$:
\[
G_{j}^{(M)}(y)=j!\binom{M}{j}\left(\frac{2p}{M(1-p)}\right)^{j/2}K_{j}\left(y_{M},p,M\right).
\]
Then, for each $j\in\bN$, uniformly over $p\in\left[\de,1-\de\right]$ and $y\in[-L,L]$,
we have convergence to the Hermite polynomials:
\[
G_{j}^{(M)}(y)=(-1)^{j}H_{j}(y)+O\left(M^{-\half}\right).
\]
where $H_j(y) = (-1)^j e^{y^2} \frac{d}{dy^j} e^{-y^2}$ are the standard Hermite polynomials.\end{lem}
\begin{proof}
The proof is by induction on $j$. The base cases $j=0$ and
$j=1$ are clear since one can verify $G_{0}^{(M)}\equiv1\equiv H_{0}$ and $G_{1}^{(M)}(y)=-2y+O(M^{-\half})=-H_{1}(y)+O(M^{-\half})$. 

Assume that the result holds for $j$ now. To prove the induction
step, we compare the three term recurrence for the Krawtchouk polynomials
to the three term recurrence for the Hermite polynomials. These are
(see \cite{ks98}):
\begin{align*}
H_{n+1}(x) & =2xH_{n}(x)-2nH_{n-1}(x),\\
K_{n+1}(x,p,N) & =\frac{pN+n-2pn-x}{p(N-n)}K_{n}(x,p,N)-\frac{n(1-p)}{p(N-n)}K_{n-1}(x,p,N).
\end{align*}
This gives the following three term recurrence for $G_{j+1}^{(M)}(y)$
in terms of $G_{j}^{(M)}(y)$ and $G_{j-1}^{(M)}(y)$:
\begin{equation}
G_{j+1}^{(M)}(y)=-2yG_{j}^{(M)}(y)\left(1-\frac{j-2p}{\sqrt{2Mp(1-p)}}\right)-2jG_{j-1}^{(M)}(y)\left(\frac{M-j+1}{M}\right).\label{eq:3term_rec}
\end{equation}
By the inductive hypothesis, the RHS of equation \eqref{eq:3term_rec} is
equal to 
\begin{equation*}
(-1)^{j+1}\left(2yH_{j}(y)-2jH_{j-1}(y)\right)+O(M^{-\half}),
\end{equation*}
and the $O(M^{-\half})$ error is uniform over $p\in\left[\de,1-\de\right]$
and $y\in[-L,L]$. By the three term recurrence for the Hermite polynomials,
this is $(-1)^{j+1}H_{j+1}(y)+O\left(M^{-\half}\right)$ as desired.\end{proof}
\begin{cor}
\label{cor:Rj_to_Pj}
Fix $\tf>0$ and $\zf\in\bR$. Let $\taf = N\tf$ and $\xf = \floor{N\tf+\sqrt{N}\zf}$. 
Recall the definition of $\dfac_{t}$ from Definition \ref{def:K}. For any choice of $\de,L>0$, the polynomials $R_{j}$ and $\tilde{R}_{j}$ from Definition \ref{def:Krawtchouk-polynomials}
have the following limit as $N\to\infty$, uniformly over the set
$(t,z)\in(\de,\tf-\de)\times(-L,L)$:
\begin{align*}
{N}^{\frac{j}{2}}R_{j}^{(\taf,\xf)}\left(Nt,\floor{Nt+\sqrt{N}z}\right) & = \left(\frac{-1}{\sqrt{2}}\right)^{j}\left(\frac{\tf-t}{\tf t}\right)^{j/2}H{}_{j}\left(\Big(z-\zf\frac{t}{\tf}\Big	)\al_{t}\right)+O\left({N}^{-\half}\right),\\
{N}^{\frac{j}{2}}	\tilde{R}_{j}^{(\taf,\xf)}\left(Nt,\floor{Nt+\sqrt{N}z}\right) & = \frac{1}{\sqrt{2}^j}\left(\frac{t}{\tf(\tf-t)}\right)^{j/2}H{}_{j}\left(\Big(z-\zf\frac{t}{\tf}\Big)\al_{t}\right)+O\left({N}^{-\half}\right).
\end{align*}
\end{cor}
\begin{proof}
This follows by the definitions $R_{j}$ and $\tilde{R}_{j}$ in terms
of Krawtchouk polynomials from Definition \ref{def:Krawtchouk-polynomials}
and the asymptotics from Lemma \ref{lem:kraw_to_hermite}. For $R_j$, the parameters
from Lemma \ref{lem:kraw_to_hermite} are to be fixed as $M=N\tf+\sqrt{N}\zf+d-1$,
$p=\frac{t}{\tf}$, $y=\left(z-\zf\frac{t}{\tf}\right)\al_{t}$, $y_{M}=Nt+\sqrt{N}z$. $\tilde{R}_j$ can be done analogously, but it is easier to note that the transformation $(t,z) \to (\tf - t, \zf - z)$ takes $\tilde{R}_j$ to $R_j$ in this scaling limit. (The extra factor of $(-1)^j$ that appears in $\tilde{R}$ comes out by simplifying using $H_j(-y) = (-1)^j H_j(y)$ and we have also used $\al_{\tf - t} = \al_t$) \end{proof}
\begin{lem}
\label{lem:KN_to_K}Fix $\tf>0$ and $\zf\in\bR$. For all
$\de,\et,M>0$, we have the following pointwise convergence uniformly
over all pairs $\left(t,z\right),\left(t^{\pr},z^{\prime}\right)\in (0,\tf)\times \bR$
that satisfy $z,z^{\prime}\in\left(-M,M\right)$, $t,t^{\prime}\in(\de,\tf-\de)$
and $\abs{t-t^{\prime}}>\et$:
\[
\lim_{N\to\infty}K^{(N),(\tf,\zf)}\Big((t,z);(t^{\prime},z^{\pr})\Big)=K^{(\tf,\zf)}\Big((t,z);(t^{\prime},z^{\pr})\Big).
\]
\end{lem}

\begin{proof}
Define the variables (which depend on $N$), $\ta ,\ta^{\prime},\taf>0$  and $x,x^{\prime},\xf\in\bZ$ by $\ta^{\prime},x^{\prime}\defequal Nt^{\prime},\floor{Nt^{\prime} + \sqrt{N}z^{\prime}}$,
$\ta,x\defequal Nt,\floor{Nt + \sqrt{N}z}$, $\taf,\xf\defequal N\tf,\floor{N\tf + \sqrt{N}\zf}.$
By comparing the kernel for non-intersecting Brownian bridges in Definition \ref{def:K} to the kernel for non-intersecting Poisson bridges in equation (\ref{eq:NIPb_kernel}), we see that both kernels consist of a sum of $d+1$ terms. We will show the convergence of each term individually with the help of the Poisson CLT: $\lim_{M\to\infty} \sqrt{M} \mu\left(M,M+\floor{\sqrt{M}z}\right) = \frac{1}{\sqrt{2\pi}}\exp\left(-\half z^{2}\right)$. The convergence holds uniformly and is stated precisely in Proposition \ref{prop:local_clt}.

The convergence of the first term of equation (\ref{eq:NIPb_kernel}) is a direct application of this Poisson CLT. Notice that uniformly over all $t^{\prime},t$ with $\abs{t^{\prime}-t}>\et$
that we have $n^{\prime}-n>N\et$. By application of Proposition \ref{prop:local_clt} and the definition of $\ta^\prime,\ta,x^\prime,x$, we have that uniformly over all such $t^{\prime},t$ and all $z,z^{\prime}$
\[
\lim_{N\to\infty}\sqrt{N}\mu\big(\ta^\prime - \ta, x^\prime - x\big) =\frac{1}{\sqrt{2\pi\left(t^{\prime}-t\right)}}\exp\left(-\frac{(z^{\prime}-z)^{2}}{2\left(t^{\prime}-t\right)}\right),
\]
and it is clear that $\one\left\{ \ta<\ta^{\prime}\right\}=\one\left\{ t<t^{\prime}\right\}$. To see the convergence of the remaining $d$ terms consider as follows. We again apply the local central limit theorem Proposition \ref{prop:local_clt} to the $j$-th term of
the sum in the definition of $K_P^{(\taf,\xf)}$ from equation (\ref{eq:NIPb_kernel}) to see uniform convergence for the Poisson weights that appear.	
Combining these asymptotics with the asymptotics for $R_{j}$ and
$\tilde{R}_{j}$ from Corollary \ref{cor:Rj_to_Pj} we have the following
limit for the $j$-th term that appears in our limit, corresponding to the $j$-th term in the sum from equation (\ref{eq:NIPb_kernel}):
\begin{align*}
 & \lim_{N\to\infty} \frac{ \sqrt{N} (-1)^j \binom{\xf+d-1}{j}}{\mu\left(\taf,\xf+d-1\right)}\tilde{R}_{j}(\ta,x)\mu(\taf-\ta,\xf-x+d-1)R_{j}(\ta^{\prime},x^{\prime})\mu(\ta^{\prime},x^{\prime})\\
 = & \lim_{N\to\infty}  \frac{ (-1)^j N^{-j}\binom{\xf+d-1}{j}}{\sqrt{N}\mu\left(\taf,\xf+d-1\right)} \\
 & \times \left( N^\frac{j}{2} \tilde{R}_{j}(x,\ta) \right) \left( N^\half\mu(\xf-x+d-1,\taf-\ta)\right) \left( N^{\frac{j}{2}} R_{j}(x^{\prime},\ta^{\prime})\right) \left( N^\half \mu(x^{\prime},\ta^{\prime}) \right)  \\
 = & (-1)^j \frac{{\tf}^j}{j!} \sqrt{2\pi \tf} \exp\Big( \frac{\zf^2}{2 \tf} \Big) \\
& \times \left(\frac{+1}{\sqrt{2}}\right)^{j} \left(\frac{t}{\tf(\tf-t)}\right)^{j/2} H_{j}\left(\Big(z-\zf\frac{t}{\tf}\Big)\dfac_{t}\right)\frac{1}{\sqrt{2\pi(\tf - t)}} \exp\left(-\frac{\left(\zf-z\right)^{2}}{2(\tf-t)}\right)\\
& \times \left(\frac{-1}{\sqrt{2}}\right)^{j} \left(\frac{\tf-t'}{\tf t'}\right)^{j/2} H_{j}\left(\Big(z^{\prime}-\zf\frac{t^{\prime}}{\tf}\Big)\dfac_{t^{\prime}}\right)\frac{1}{\sqrt{2 \pi t^\prime}}\exp\left(-\frac{z^{\prime2}}{2t^{\prime}}\right).
\end{align*}
After grouping the terms appropriately, it is verified that this is exactly equal to the corresponding $j$-th term in Definition \ref{def:K} for the kernel $K^{(\tf,\zf)}$ for non-intersecting Brownian bridges. \end{proof}


\begin{cor}
\label{cor:bounds_on_KN}Fix $\tf>0$ and $\zf\in\bR$. For any $\de,M>0$, there exist constants $C_{K}^{<}=C_{K}^{<}(\de,M), $
and $C_{K}^{\ge}=C_{K}^{\geq}(\de,M)$ so that for pairs $(t,z);(t^{\prime},z^{\prime})$
with $t,t^{\prime}\in(\de,T-\de)$, $\abs{t^{\prime}-t}>\et$ and
$z,z^{\prime}\in(-M,M)$ we have
\[
\begin{array}{lclr}
\sup_{N}\abs{K^{(N),(\tf,\zf)}\big((t,z);(t^{\prime},z^{\prime})\big)} & \leq & {C_{K}^{<}}\left({t^{\prime}-t}\right)^{-\half} & \text{if }t<t^{\prime},\\ 
\sup_{N}\abs{K^{(N),(\tf,\zf)}\big((t,z);(t^{\prime},z^{\prime})\big)} & \leq & C_{K}^{\geq} & \text{ if }t\geq t^{\prime}.
\end{array}
\]
\end{cor}
\begin{proof}
When $t\geq t^{\prime}$, the first term in the definition of $K^{(N),(\tf,\zf)}$
and $K^{(\tf,\zf)}$ is $0$, and the proof of Lemma \ref{lem:KN_to_K}
shows that regardless of $\et, $ $K^{(N),(\tf,\zf)}$ converges
uniformly to $K^{(\tf,\zf)}$ for $t,t^{\prime}\in(\de,\tf-\de)$
and $z,z^{\prime}\in(-M,M)$. Thus when $t\geq t^{\prime}$, since
$K^{(\tf,\zf)}$ is bounded by $C_{K}^{\geq}$ here by Lemma 3.4 from \cite{CorwinNica16},
and since the convergence in Lemma \ref{lem:KN_to_K} is uniform,
it follows that $K^{(N),\left(\tf,\zf\right)}$ is also bounded. Let $C_{K}^{\geq}$ be a constant large enough to bound both of them. 

To see the inequality when $t<t^{\prime}$ we must consider the
first term. By applying the bound from Corollary \ref{cor:local_clt_bound}
to the first term in $K^{(N),(\tf,\zf)}$, along with the bound $\sqrt{t^{\prime}-t}<\sqrt{\tf}$, we have by the triangle inequality that:
\[
\sup_{N}\sqrt{t^{\prime}-t}\abs{K^{(N),(\tf,\zf)}\Big((t,z);(t^{\prime},z^{\prime})\Big)}\leq\left(C_{P}+\sqrt{\tf}C_{K}^{\geq}\right).
\]
This bound gives the desired result. \end{proof}
\begin{cor}
\label{cor:bounds_on_KN_D1}Fix $\tf>0$ and $\zf\in\bR$. For all $\de,\et,M>0$, there exists a constant $C_{D_{1},K}=C_{D_{1},K}(\de,\et,M)$
such that $\abs{K^{(N),(\tf,\zf)}\Big((t,z);(t^{\prime},z^{\prime})\Big)}\leq C_{D_{1},K}$
for all pairs $(t,z);(t^{\prime},z^{\prime})$ that satisfy $t,t^{\prime}\in(\de,\tf-\de)$,
$\abs{t^{\prime}-t}>\et$ and $z,z^{\prime}\in(-M,M)$.\end{cor}
\begin{proof}
It is easily verified by Corollary \ref{cor:bounds_on_KN} that $C_{D_{1},K}=\max\left(C_{K}^{\geq},\frac{1}{\sqrt{\et}}C_{K}^{<}\right)$ will give the desired result since $\sqrt{t^{\prime}-t}>\sqrt{\et}$ when $t^\prime > t$.
\end{proof}

\begin{proof}
(Of Proposition \ref{prop:D1}) By Lemma 3.3 from \cite{CorwinNica16} and
Lemma \ref{lem:psiN_is_det_KN}, the $k$-point correlation functions $\ps_{k}^{(\tf,\zf)}$ and $\ps_{k}^{(N),(\tf,\zf)}$
are given by $k\times k$ determinants of the kernels $K^{(\tf,\zf)}$
and $K^{(N)(\tf,\zf)}$ respectively. Since determinants are polynomial functions of the matrix entries, the existence of the bound by $C_{D_{1}}(\de,\et)$
follows by the bound for $\abs{K^{(\tf,\zf)}(\cdot)}\leq C_{D_{1},K}$
in Corollary 3.5 from \cite{CorwinNica16} and the bound $\abs{K^{(N),(\tf,\zf)}(\cdot)}<C_{D_{1},K}$
in Corollary \ref{cor:bounds_on_KN_D1}. Now notice that Lemma \ref{lem:KN_to_K}
establishes uniform convergence $K^{(N),(\tf,\zf)}\big((t_{i},z_{i});(t_{j},z_{j})\big)\to K^{(\tf,\zf)}\big((t_{i},z_{i});(t_{j},z_{j})\big)$
for any pairs $\left(t_{i},z_{i}\right)$ and $\left(t_{j},z_{j}\right)$
chosen from the list $ (\vec{t},\vec{z}) \in D_{1}(\de,\et,M)$. Since the entries are bounded, this uniform convergence of the entries implies uniform convergence of the $k\times k$ determinant. 
\end{proof}

\subsection{Bound on \texorpdfstring{$D_{2}(\protect\de,\protect\et,M)$}{D2(delta,eta,M)} -- Proof of Proposition \ref{prop:D2}}

%

\begin{lem}
\label{lem:psiN_bound} Fix $\tf>0$ and $\zf\in\bR$. For any $\de,M>0$,
there exists a constant $C_{D_{2}}=C_{D_{2}}(\de,M)$ such that for
all $\big((t_{1},z_{1});\ld;(t_{k},z_{k})\big)\in D_{2}\left(\de,\et,M\right)$
we have: 
\begin{equation}
\sup_{N}\ps_{k}^{(N),(\tf,\zf)}\Big((t_{1},z_{1});\ld;(t_{k},z_{k})\Big) \leq \frac{C_{D_{2}}}{\sqrt{t_{2}-t_{1}}\sqrt{t_{3}-t_{2}}\cdots\sqrt{t_{k}-t_{k-1}}}.
\end{equation}
\end{lem}
\begin{proof}
This follows by applying Lemma 3.15 in \cite{CorwinNica16} to the bounds
on $K^{(N),\left(\tf,\zf\right)}$ from Corollary \ref{cor:bounds_on_KN}
and then finally using the fact that $K^{(N),(\tf,\zf)}$ is the determinantal
kernel for $\ps_{k}^{(N),(\tf,\zf)}$ from Lemma \ref{lem:psiN_is_det_KN}.
\end{proof}

\begin{proof}
(Of Proposition \ref{prop:D2}) Recall from Definition \ref{def:rescaled_NIPb}
that $\ps^{(N),(\tf,\zf)}$ is constant on intervals of the form $I^{(N)}(t,z)$.
Thus, as in equation (\ref{eq:integral_is_semidiscrete_sum}), we may rewrite the
integral as a semidiscrete sum. Recalling the definition of the set $D_2(\de,\et)$ in Definition \ref{def:space-time}, we apply the bound from Lemma \ref{lem:psiN_bound} on $\ps_{k}^{(N),(\tf,\zf)}(\vec{t},\vec{z})$ to see that
\begin{align}
\label{eq:D2_1} 
  & \iintop_{D_{2}(\de,\et)}\abs{\ps_{k}^{(N),(\tf,\zf)}(\vec{t},\vec{z})}^{2}\dd\vec{t}\dd\vec{z}	\nonumber \\
  =& N^{-\frac{k}{2}} \sintt{\vec{t}\in \De_k(\de,\tf-\de)\cap S_\et^c}{\vec{z} \in N^{-\half} \bZ^k} \abs{\ps_{k}^{(N),(\tf,\zf)} \left( (t_1,z_1 - \sqrt{N} t_1 ),...,(t_k,z_k - \sqrt{N} t_k) \right) }^{2} \dd \vec{t} \nonumber \\
 \leq & \quad \mathclap{\intop_{\vec{t}\in \De_k(\de,\tf-\de)\cap S_\et^c}} \quad \quad \frac{C_{D_{2}}N^{-\frac{k}{2}}}{\sqrt{t_{2}-t_{1}}\cdots\sqrt{t_{k}-t_{k-1}}}\sum_{\vec{z}\in \frac{\bZ^{k}}{\sqrt{N}}}\ps_{k}^{(N),(\tf,\zf)}\left((t_{1},z_1 - \sqrt{N}t_1),...,(t_{k},z_k - \sqrt{N} t_k)\right) \dd \vec{t}.\nonumber
\end{align}
We notice now from Definition \ref{def:rescaled_NIPb} that the scaling $N^{-\frac{k}{2}}$ makes the above exactly the probability of finding a particle occupying each position
$z_{1}+t_{1},\ld,z_{k}+t_{k}$ at the times $t_{1},\ld,t_{k}$ respectively. Summing these probabilities simply counts the $d$
paths:
\begin{equation*}
\sum_{\vec{z}\in N^{-\half}\bZ^{k}}\p\left(\bigcap_{j=1}^{k}\left\{ z_{j} -\sqrt{N} t_j \in\X^{(N),(\tf,\zf)}(t_{j})\right\} \right) = \e\left[d^{k}\right]=d^{k}.
\end{equation*}
We hence get the bound:
\begin{equation}
\iintop_{D_{2}(\de,\et)}\abs{\ps_{k}^{(N),(\tf,\zf)}(\vec{t},\vec{z})}^{2}\dd\vec{t}\dd\vec{z} \leq d^{k}C_{D_{2}}\intop_{\vec{t}\in \De_k(\de,\tf-\de)\cap S_\et^c}\frac{\dd t_{1}\dd t_{2}\ld \dd t_{k}}{\sqrt{t_{2}-t_{1}}\sqrt{t_{3}-t_{2}}\cdots\sqrt{t_{k}-t_{k-1}}}.\label{eq:D2_2}
\end{equation}
Notice that since $\left(t_{i+1}-t_{i}\right)^{-\half}$ is integrable around the singularity
at $t_{i+1}=t_{i}$, the integrand in equation (\ref{eq:D2_2}) has finite total integral when integrated over the whole range of times $\vec{t}\in\De_k(\de,\tf-\de)$. Since $\lim_{\et\to0}\one\left\{ S_\et^c \right\} =0\ a.s.$,
we have by the dominated convergence theorem that the RHS of equation
(\ref{eq:D2_2}) tends to $0$ as $\et\to0$. This gives the desired result.
\end{proof}

\section{Overlap Times} \label{sec:overlap}

In this section we extend the method of overlap times used for discrete polymers in \cite{CorwinNica16} to be able to apply them to the semi-discrete polymers studied here. This overlap time can also be thought of as the semi-discrete version of the local times between non-intersecting Brownian motions studied in Section 4 of \cite{OConnellWarren2015}.  We prove in this section that the overlap time has a property called ``weak exponential moment control''. This property is then used in Section \ref{sec:L2b} to bound the $L^2$ norm of the $k$-point correlation functions.


\begin{defn}
\label{def:overlap_times_NIW} Recall from Definition \ref{def:NIP}
the notation $\X(\ta)$, $\ta\in(0,\infty)$ for $d$ non-intersecting
Poisson processes started from $\X(0)=\vec{\de}_{d}(0)$. Let $\X^{\prime}(\ta)$,
$\ta\in(0,\infty)$ be an independent copy of the same ensemble. For
indices $1\leq k,\ell\leq d$ and times $0<\ta_{1}<\ta_{2}$, define the overlap time on $[\ta_{1},\ta_{2}]$ between the $k$-th walker of $\vec{X}$ and the $\ell$-th walker
of $X^{\prime}$ by:
\begin{equation}
O_{k,\ell}[\ta_{1},\ta_{2}]\defequal\intop_{\ta_{1}}^{\ta_{2}}\one\left\{ \X_{k}(\ta)=\X_{\ell}^{\prime}(\ta)\right\} \dd\ta.\label{eq:O_kl}
\end{equation}
Define the total overlap time on the interval $[\ta_{1},\ta_{2}]$
of these processes by
\[
O[\ta_{1},\ta_{2}]\defequal\sum_{1\leq k,\ell\leq d}O_{k,\ell}[\ta_{1},\ta_{2}]=\intop_{\ta_{1}}^{\ta_{2}}\abs{\left\{ \vec{X}(\ta)\cap\vec{X}^{\prime}(\ta)\right\} }\dd\ta,
\]
where we think of $\vec{X}(\ta)$ and $\vec{X}^{\prime}(\ta)$ as
sets and $\abs{\left\{ \X(\ta)\cap\vec{X}^{\prime}(\ta)\right\} }$
is the number of elements in their intersection. 

Similarly, for any fixed $\xf\in\bN$ and $\taf>0$, recall from Definition \ref{def:NIPb} that $\X^{(\taf,\xf)}(\ta)$,
$\ta\in[0,\taf]$ denotes an ensemble of $d$ non-intersecting random walker
bridges started from $\X^{(\taf,\xf)}(0)=\vec{\de}_{d}(0)$ and ended
at $\X^{(\taf,\xf)}(\taf)=\vec{\de}_{d}(\xf)$. Let $\X^{\prime(\xf,\nf)}(\ta)$,$\ta\in[0,\taf]$
be an independent copy of the same ensemble. For times $0<\ta_{1}<\ta_{2}$, define the total overlap time on the interval $[\ta_{1},\ta_{2}]\subset[0,\taf]$ of these processes by
\[
O^{(\taf,\xf)}[\ta_{1},\ta_{2}]\defequal\intop_{\ta_{1}}^{\ta_{2}}\abs{\left\{ \X^{(\taf,\xf)}(\ta)\cap\X^{\prime(\taf,\xf)}(\ta)\right\} }\dd\ta.
\]
For any fixed $\zf\in\bR$ and $\tf>0$, and $0<t_1<t_2<\tf$, define the rescaled overlap time by
\[
O^{(N),(\tf,\zf)}[t_{1},t_{2}]\defequal\frac{1}{\sqrt{N}}O^{\floor{N\tf},\floor{N\tf+\sqrt{N}\zf}}\left[\floor{Ns},\floor{Ns^{\prime}}\right].
\]
\end{defn}

\subsection{Weak exponential moment control -- definition and properties}
\begin{defn}
\label{def:exp_mom_control} We say that a collection of non-negative
valued processes 
\[
\left\{ Z^{(N)}(t)\ :\ t\in\left[0,\tf\right]\right\} _{N\in\bN},
\]
is ``weakly exponential moment controlled as $t\to0$'' if the following
conditions are met:

i) For any fixed $t\in[0,\tf]$,$\ga>0$, there exists $N_{\ga}\in\bN$
so that: 
\[
\sup_{N>N_{\ga}}\e\left[\exp\left(\ga Z^{(N)}(t)\right)\right]<\infty.
\]

ii) For any fixed $\ga>0$, and $\ep>0$, there exists $N_{\ga,\ep}\in\bN$
so that:
\[
\limsup_{t\to0}\left(\sup_{N>N_{\ga,\ep}}\e\left[\exp\left(\ga Z^{(N)}(t)\right)\right]\right)\leq1+\ep.
\]

iii) For any fixed $t\in[0,\tf]$, $\ep>0$ and $\ga>0$, there exists
$N_{\ga,\ep}\in\bN$ so that:
\[
\limsup_{\ell\to\infty}\left(\sup_{N>N_{\ga,\ep}}\e\left[\sum_{k=\ell}^{\infty}\frac{\ga^{k}}{k!}\left(Z^{(N)}(t)\right)^{k}\right]\right)\leq\ep.
\]
\end{defn}
\begin{rem}
The notation of ``exponential moment controlled'' without the adjective ``weak'' appears in Definition 4.3 of \cite{CorwinNica16}. Here we weaken the definition by taking the sup over $N>N_{\ga,\ep}$ rather than sup over all $N\in\bN$,
and allowing for an error of size $\ep$ in properties ii) and iii). This extension is necessary because it allows us to handle exponential rare events that arise in the continuous-time processes we study. Note that the exponential moment control defined in \cite{CorwinNica16} always
implies weak exponential moment control by setting $N_{\ga,\ep}=1$ everywhere. This relaxation is needed in the semi-discrete setting because the semi-discrete processes under consideration have the potential to be arbitrarily large in a finite amount of time (as opposed to discrete simple symmetric random walks, whose height cannot exceed the number of steps the process takes). This leads to exponential rare ``bad'' events: the $\epsilon$ of room created by the weaker definition leaves space for these errors.
\end{rem}

\begin{lem}
\label{lem:subgaussian_is_exp_mom}Suppose $\left\{ Z^{(N)}(t),t\in[0,\tf]\right\} _{N\in\bN}$
is a collection of processes so that 
\begin{equation*}
\p\left(Z^{(N)}(t)>x\right)\leq C\exp\left(-c\frac{x^{2}}{t}\right). 
\end{equation*}
Then $\left\{ Z^{(N)}(t),t\in[0,\tf]\right\} $ is weakly exponential
moment controlled.\end{lem}
\begin{proof}
In Lemma 4.7 of \cite{CorwinNica16} it is proven that such processes are
exponential moment controlled in the sense of Definition 4.3
from that paper. Since exponential moment control implies weak exponential
control, the result follows.
\end{proof}

\begin{lem}
\label{lem:exp-mom-iii}If $\left\{ Z^{(N)}(t)\ :\ t\in\left[0,\tf\right]\right\} _{N\in\bN}$
is weakly exponential moment controlled, then for any exponent $m\in\bN$
and $\ga>0$ we have that $\forall\ep>0$ , $\exists N_{\ga,\ep,m}>0$
s.t.:
\[
\limsup_{\ell\to\infty}\sup_{N>N_{\ga,\ep,m}}\e\left[\left(\sum_{k=\ell}^{\infty}\frac{\ga^{k}}{k!}\left(Z^{(N)}(t)\right)^{k}\right)^{m}\right]\leq\ep.
\]
\end{lem}
\begin{proof}
The proof is very similar to the argument from Lemma 4.4 in \cite{CorwinNica16}. Since each $Z^{(N)}(t)$ is non-negative, there is no harm in rearranging
the order of the terms in the infinite sum to arrive at:
\begin{eqnarray*}
\left(\sum_{k=\ell}^{\infty}\frac{1}{k!}\ga^{k}\left(Z^{(N)}(t)\right)^{k}\right)^{m} & = & \sum_{k_{1},\ld,k_{m}=\ell}^{\infty}\frac{1}{k_{1}!\ld k_{m}!}\ga^{k_{1}+\ldots+k_{m}}\left(Z^{(N)}(t)\right)^{k_{1}+\ld+k_{m}}\\
 & \leq & \sum_{k\geq m\ell}\left(\sum_{k_{1}+\ld+k_{m}=k}\binom{k}{k_{1},\ld,k_{m}}\frac{1}{k!}\ga^{k}Z^{(N)}(t)^{k}\right)\\
 & = & \sum_{k\geq m\ell}\left(m\ga\right)^{k}\frac{1}{k!}Z^{(N)}(t)^{k},
\end{eqnarray*}
The desired result now holds by property iii) from Definition \ref{def:exp_mom_control}
of weak exponential moment control with parameter $m\ga$ and choosing
$N_{\ga,\ep,m}=N_{m\ga,\ep}$.\end{proof}
\begin{lem}
\label{lem:sum-of-exp} Suppose $\left\{ W^{(N)}(t)\ :\ t\in\left[0,\tf\right]\right\}_{N\in\bN}$
and $\left\{ Y^{(N)}(t)\ :\ t\in\left[0,\tf\right]\right\}_{N\in\bN}$
are two collections of processes which are both weakly exponential moment controlled as $t\to0$. If $\left\{ Z^{(N)}(t)\ :\ t\in\left[0,\tf\right]\right\} _{N\in\bN}$
is a collection of non-negative valued processes so that for all $t\in[0,\tf]$
and all $N\in\bN$ we have
\[
Z^{(N)}(t)\leq W^{(N)}(t)+Y^{(N)}(t)\ a.s.
\]
then \textup{$\left\{ Z^{(N)}(t)\ :\ t\in\left[0,\tf\right]\right\} $}
is also weakly exponential moment controlled as $t\to0$.\end{lem}
\begin{proof}
The proof is very similar to the argument from Lemma 4.5 in \cite{CorwinNica16}. Property i) and ii) of the weak exponential moment control are easily verified by application of the Cauchy-Shwarz
inequality:
\begin{equation*}
\e\left[\exp\left(\ga Z^{(N)}(t)\right)\right]\leq \sqrt{\e\left[\exp\left(2\ga W^{(N)}(t)\right)\right]\cdot\e\left[\exp\left(2\ga Y^{(N)}(t)\right)\right]}.
\end{equation*}
To see property iii) for $W^{(N)}(t)$, we argue as in Lemma 4.5 in \cite{CorwinNica16} by the Cauchy-Schwarz inequality that
\begin{align}
\e\left[\sum_{k=2\ell}^{\infty}\frac{1}{k!}\ga^{k}\left(Z^{(N)}(t)\right)^{k}\right] \leq & \sqrt{\e\left[\exp\left(2\ga W^{(N)}(t)\right)\right]\e\left[\left(\sum_{b=\ell}^{\infty}\left(\frac{1}{b!}\ga^{b}\left(Y^{(N)}(t)\right)^{b}\right)\right)^{2}\right]}\nonumber  \\
    &+\sqrt{\e\left[\exp\left(2\ga Y^{(N)}(t)\right)\right]\e\left[\left(\sum_{a=\ell}^{\infty}\left(\frac{1}{a!}\ga^{a}\left(W^{(N)}(t)\right)^{a}\right)\right)^{2}\right]}, \label{eq:moment_for_sum}
\end{align}
By property i) now, we can find $N_{\ga}\in\bN$ so for all $N>N_{\ga}$ we have a uniform upper bound over  $\e\left[\exp\left(2\ga W^{(N)}(t)\right)\right]$
and $\e\left[\exp\left(2\ga Y^{(N)}(t)\right)\right]$. The desired limit as $\ell\to\infty$ of equation (\ref{eq:moment_for_sum})
follows by an application of Lemma \ref{lem:exp-mom-iii}.\end{proof}

\begin{lem}
\label{lem:weak_mom_contrl_lemma} Suppose that $\left\{ W^{(N)}(t):t\in[0,\tf]\right\} _{N\in\bN}$ is
weakly exponential moment controlled and that the collection of processes
$\left\{ Z^{(N)}(t):t\in[0,\tf]\right\} _{N\in\bN}$ have the property
that
\[
\p\left(Z^{(N)}(t)>x\right)\leq\p\left(W^{(N)}(t)>x\right)+C\exp\left(-c\sqrt{N}x \right).
\]
Then $\left\{ Z^{(N)}(t):t\in[0,\tf]\right\} _{N\in\bN}$ is weakly
exponential moment controlled.\end{lem}
\begin{proof}
By using integration by parts, we have:
\begin{align*}
\e\left[\exp\left(\ga Z^{(N)}(t)\right)\right] & = 1+\ga\intop_{0}^{\infty} e^{\ga x}\p\left(Z^{(N)}(t)>x\right)\dd x\\
 & \leq  1+\ga\intop_{0}^{\infty} e^{\ga x}\p\left(W^{(N)}(t)>x\right)\dd x+C\ga\intop_{0}^{\infty}\exp\left(-c\sqrt{N}x+\ga x\right)\dd x\\
 & = \e\left[\exp\left(\ga W^{(N)}(t)\right)\right]+\frac{C\ga}{c\sqrt{N}-\ga}.
\end{align*}
from which properties i) and ii) follow from the weak exponential
moment control of $W^{(N)}(t)$ and by choosing $N_{\ga,\ep}$ large
enough so that $\frac{C\ga}{c\sqrt{N}-\ga}<\half\ep$ for $N>N_{\ga,\ep}$.
To see property iii) consider:
\begin{align*}
\e\left[Z^{(N)}(t)^{k}\right] & \leq k\intop_{0}^{\infty}x^{k-1}\p\left(W^{(N)}(t)>x\right)\dd x+Ck\intop_{0}^{\infty}x^{k-1}\exp\left(-c\sqrt{N}x\right)\dd x\\
 & =\e\left[W^{(N)}(t)^{k}\right]+k!\left(\frac{1}{c\sqrt{N}}\right)^{k}.
\end{align*}
Thus for $N>\ga^{2}/c^{2}$ we have that the following infinite sum
is finite (again all terms are non-negative so there is no harm in
rearranging the terms of the sum):
\begin{equation}
\e\left[\sum_{k=\ell}^{\infty}\frac{1}{k!}\ga^{k}\left(Z^{(N)}(t)\right)^{k}\right]\leq\e\left[\sum_{k=\ell}^{\infty}\frac{1}{k!}\ga^{k}\left(W^{(N)}(t)\right)^{k}\right]+\left(\frac{\ga}{c\sqrt{N}}\right)^{\ell}\frac{1}{c\ga^{-1}\sqrt{N}-1}.\label{eq:mom-bound}
\end{equation}
We now notice that for $N>\ga^{2}/c^{2}$, the second term of equation
(\ref{eq:mom-bound}) goes to $0$ as $\ell\to\infty$. Along with
property iii) of the weak exponential moment control
for $W^{(N)}$, this yields property iii) for $Z^{(N)}$ as desired.
\end{proof}

\subsection{Bounds on positions of non-intersecting Poisson processes }
The bounds in this subsection are needed as an ingredient to prove weak exponential moment control for the overlap times.
\begin{lem}
\label{lem:exp-mom-for-location} Recall from Definition \ref{def:NIP}
that $\X(\ta)$, \textup{$\ta>0$} denotes an ensemble of $d$ non-intersecting
Poisson processes  started from $\X(0)=\vec{\de}_{d}(0)$. Denote by $\bar{X}_d(\ta) \defequal X_d(\ta)-\ta$ and $\bar{X}_1(\ta) \defequal X_1(\ta)-\ta$. Then there are constants
$c,C$ (which depend on $d$) so that for all $N$ and for any fixed $t>0$ we have the following inequality: 
\begin{align*}
\p\left(\sup_{0<\ta<tN}\abs{\bar{X}_{d}(\ta)}>\al\sqrt{N}\right)&\leq C\exp\left(-c\frac{\al^{2}}{t}\right)+C\exp\left(-c\sqrt{N}\al\right),\\
\p\left(\sup_{0<\ta<tN}\abs{\bar{X}_{1}(\ta)}>\al\sqrt{N}\right)&\leq C\exp\left(-c\frac{\al^{2}}{t}\right)+C\exp\left(-c\sqrt{N}\al\right).
\end{align*}
\end{lem}
\begin{proof}
The proof is by induction on $d$, using the reflected construction of non-intersecting random walkers from Section 2.1 of \cite{warren2009}. We explicitly give the argument for the top line $\bar{X}_d$ here; the proof for the bottom line $\bar{X}_1$ is analogous. 

The case $d=1$ is clear since in this case $X_{1}(\cdot)$
is simply a standard Poisson process and the needed estimate is a
consequence of the stronger general result about Poisson processes
proven in Lemma \ref{lem:Poisson_fact}. Now suppose the result holds
for $d-1$. The reflected construction in \cite{warren2009} is a coupling of
the process $\X$ of $d$ non-intersecting walkers started from $\vec{\de}_{d}(0)$
and the process $\Y$ of $d-1$ non-intersecting walkers started from
$\vec{\de}_{d-1}(0)$. In this coupling, the process $\Y$ is first
constructed, and then the top line $X_{d}$ is realized as a Poisson
process which is pushed upward by the top line of the $\vec{Y}$;
symbolically this is:
\[
X_{d}(\ta^{\prime})-X_{d}(\ta)=P(\ta^{\prime})-P(\ta)+\intop_{\ta}^{\ta^{\prime}}\de\left\{ X_{d}(s)-Y_{d-1}(s)=0\right\} \dd s,
\]
where $P(t)$ is a rate 1 Poisson process independent of the process $\Y$, and $\delta$ is the Dirac delta. Denoting by $\bar{P}(\ta)\defequal P(\ta) - \ta$ and $\bar{Y}_{d-1}(\ta)\defequal Y_{d-1}(\ta)-\ta$, we see that for any $\ga > 0$:
\begin{equation}
\label{eq:inclusion_boost}
\left\{ \sup_{0<\ta<T}\bar{X}_{d}(\ta) >\ga\right\} \subset \left\{ \sup_{0<\ta<T}\bar{Y}_{d-1}(\ta)>\frac{\ga}{2}\right\} \cup\left\{ \sup_{0<\ta<T}\bar{P}(\ta)-\inf_{0<\ta<T}\bar{P}(\ta)>\frac{\ga}{2}\right\} .
\end{equation}
This inclusion follows since if $\sup_{0<\ta<T}\bar{Y}_{d-1}(\ta)\leq\frac{\ga}{2}$,
then in order for the process $\bar{X}_{d}(\ta)$ to advance
from position $\frac{\ga}{2}$ to $\ga$, the process
$\bar{X}_{d}(\ta)$ will need a boost of at least $\frac{\ga}{2}$
which can only come from the process $\bar{P}(\ta)$. 
We also have by the inductive hypothesis, $\p\left(\sup_{0<\ta<tN}\bar{Y}_{d-1}(\ta)>\frac{\al}{2}\sqrt{N}\right)\leq C_{d-1}\exp\left(-c_{d-1}\frac{\al^{2}}{4t}\right)+C_{d-1}\exp\left(-c_{d-1}\sqrt{N}\frac{\al}{2}\right)$
for some constants $C_{d-1}$,$c_{d-1}$ (which depend on $d-1)$.
On the other hand, by Lemma \ref{lem:Poisson_fact} we know that $\p\left(\sup_{0<\ta<tN}\bar{P}(\ta)-\inf_{0<\ta<tN}\bar{P}(\ta)>\frac{\al}{2}\sqrt{N}\right)$ also obeys this type of inequality. Setting $T = tN$ and $\ga = \al \sqrt{N}$ in equation \eqref{eq:inclusion_boost} and applying a union bound then completes the desired bound on $\p\left(\sup_{0<\ta<tN}\bar{X}_{d}(\ta)>\al\sqrt{N}\right)$.
The bound on $\p\left(\sup_{0<\ta<tN}\left(-\bar{X}_{d}(\ta)\right)>\al\sqrt{N}\right)$
follows since in this coupling we have $-\bar{X}_{d}(\ta)\leq-\bar{P}(\ta)$ and  the result for $\bar{P}$ is clear by Lemma \ref{lem:Poisson_fact}. \end{proof}
\begin{cor}
\label{cor:pos_bound}For any $1\leq k\leq d$, there are constants
$c,C$ so that for all $N$ and for any fixed $t>0$, the maximum
of absolute value of the compensated $k$-th line process $\bar{X}_k(\ta) \defequal X_k(\ta)-\ta$
obeys the following inequality: 
\[
\p\left(\sup_{0<\ta<tN}\abs{\bar{X}_{k}(\ta)}>\al\sqrt{N}\right)\leq C\exp\left(-c\frac{\al^{2}}{t}\right)+C\exp\left(-c\sqrt{N}\al\right).
\]
\end{cor}
\begin{proof}
The cases $k=d$ and $k=1$ are exactly Lemma \ref{lem:exp-mom-for-location}. Now notice that for $1<k<d$, because the walkers are always ordered so
that $X_{1}(t)<X_{k}(t)<X_{d}(t)$, we have:
\[
\frac{1}{\sqrt{N}}\sup_{0<\ta<tN}\abs{\bar{X}_{k}(\ta)}\leq\frac{1}{\sqrt{N}}\sup_{0<\ta<tN}\abs{\bar{X}_{d}(\ta)}+\frac{1}{\sqrt{N}}\sup_{0<\ta<tN}\abs{\bar{X}_{1}(\ta)},
\]
 and the desired inequality then follows by a union bound. 
\end{proof}

\subsection{Inverse gaps of non-intersecting Poisson processes }
In this subsection, we prove bounds on the inverse gaps, $\abs{X_j(t)-X_i(t)}^{-1}$, $1\leq i,j\leq d$. The methods here are similar to those used for non-intersecting random walks in \cite{CorwinNica16}.

\begin{defn}
For fixed $\ta > 0$, $\ep>0$, define $\bS_{\ta,\ep}\subset\bN^{d}$
by 
\[
\bS_{\ta,\ep}\defequal\left\{ x\in\bN^{d}:\abs{x_{j}-x_{i}}>n^{\half-\ep}\ \forall1\leq i,j\leq d\right\} .
\]
\end{defn}
\begin{lem}
\label{lem:bounded_gaps}Recall from Definition \ref{def:NIP} that
$\X(\ta)$, \textup{$\ta>0$} denotes $d$ non-intersecting
Poisson processes and $\e_{\vec{x}^{0}}\left[\cdot\right]$
denotes the expectation from $\vec{X}(0)=\vec{x}^{0}$.
Recall from Definition \ref{def:NIBb} that $\vec{D}(t)$, $t\in(0,\infty)$
denotes $d$ non-intersecting Brownian motions and $\e_{\vec{0}}\left[\cdot\right]$
denotes the expectation from $\vec{D}(0)=\vec{0}$.
For any $\ep<\oo 4$, there exists an absolute constant $C_{\ep}$
so that for indices $a,b$ with $1 \leq a < b \leq d$  we have:
\begin{eqnarray}
\sup_{\ta > 0}\sup_{\vec{x}^{0}\in{\bS_{\ta,\ep}}}\e_{\x^{0}}\left[\frac{1}{\frac{1}{\sqrt{\ta}}\left(X_{b}(\ta)-X_{a}(\ta)\right)}\right] & \leq & 3^{\binom{d}{2}}\e_{\vec{0}}\left[\frac{1}{D_{b}(1)-D_{a}(1)}\right]+C_{\ep}.\label{eq:Egaps_W_n}
\end{eqnarray}
\end{lem}
\begin{proof}
Let $\vec{P}(\ta)=\left(P_{1}(\ta),\ld,P_{d}(\ta)\right)$ be $d$
iid ordinary rate $1$ Poisson processes started from $\vec{P}(0)=(0,0,\ld0)$,
and denote their expectation simply by $\e$. By Definition \ref{def:NIP}
for $\X(\ta)$ as a Doob $h$-transform using the Vandermonde determinant
$h_{d}$, the expectation on the LHS of equation ($\ref{eq:Egaps_W_n}$)
can be written as
\begin{align}
&\e_{\x^{0}}\left[\frac{1}{\frac{1}{\sqrt{\ta}}\left(X_{b}(\ta)-X_{a}(\ta)\right)}\right] \nonumber \\
  =& \frac{1}{h_{d}(\x^{0})}\e\left[\frac{h_{d}\left(\vec{P}(\ta)+\x^{0}\right)}{\frac{1}{\sqrt{\ta}}\left(P_{b}(\ta)+x_{b}^{0}-P_{a}(\ta)-x_{a}^{0}\right)}\one\left\{ \ta_{\vec{x}^{0}}^{P}>\ta\right\} \right]\nonumber \\
 =&\frac{\sqrt{\ta}}{h_{d}(\x^{0})}\e\left[\prod_{{i<j},{(i,j)\neq(a,b)}}\left(P_{j}(\ta)+x_{j}^{0}-P_{i}(\ta)-x_{i}^{0}\right)\one\left\{ \ta_{\vec{x}^{0}}^{P}>\ta\right\} \right].\label{eq:E_oo_gap}\\
\ta_{\x^{0}}^{P} & \defequal\inf_{ \si > 0}\left\{ P_{i}(\si)+x_{i}^{0}=P_{j}(\si)+x_{j}^{0}\text{ for some }1\leq i<j\leq d\right\} .\nonumber 
\end{align}
Denote by $\bar{P}_i(\ta) \defequal P_i(\ta) -\ta$. By the KMT coupling \cite{KMTcoupling}, we can couple these processes with $d$ iid Brownian motions, $\B(t)=\left(B_{1}(t),\ld,B_{d}(t)\right)$
started from $\vec{B}(0)=(0,0,\ld,0)$ so that for absolute constants
$K_{1},K_{2},K_{3}>0$ we have at integer times that: 
\[
\p\left(\sup_{1\leq j\leq d}\sup_{1\leq i\leq \floor{\ta} }\abs{\bar{P}_j(i)-B_{j}(i)}>K_{3}\ln \ta +x\right)\leq K_{1}\exp\left(-K_{2}x\right),
\]
for all $x\in\bR$. (This can be done because
each Poisson variable can be realized as a sum of iid mean zero random variables, $\bar{P}(i)\dequal\sum_{s=1}^{i}\left(\xi(s)-1\right)$
where $\xi(s)$ are iid Poisson 1 random variables).
We do not need the full power of this $O\left(\ln \ta\right)$
coupling, so we will put $x=\frac{1}{3} \ta^{\half-2\ep}$ and use an inclusion
to get the weaker inequality: 
\begin{equation}
\p\left(\sup_{1\leq j\leq d}\sup_{1\leq i\leq \floor{\ta}}\abs{\bar{P}_{j}(i)-B_{j}(i)}\geq\frac{1}{3}\ta^{\half-2\ep}\right)\leq K_{1}\exp\left(-\frac{K_{2}}{3}\ta^{\half-2\ep}\right).\label{eq:KMT}
\end{equation}

Now define the event $A_{\ep,\ta}=\left\{ \sup_{1\leq j\leq d}\sup_{t\in[0,\ta]}\abs{\bar{P}_{j}(t)-B_{j}(t)}<\ta^{\half-2\ep}\right\} $.
Notice that if $A_{\ep,\ta}$ does not happen, then either the Brownian
motions $B_j(t)$ are far from the Poisson process $\bar{P}_j(t)$
at some integer time $t=i$ or the Brownian motions or else the Poisson processes
themselves have a large amount of movement in some interval $t\in[i,i+1]$ (or in the last interval $t\in[\floor{\ta},\ta]$; in what follows we upper bound this last special exception by looking at $t\in[\floor{\ta},\floor{\ta}+1]$).
Using this inclusion, we have by a union bound that $\p\left(A_{\ep,\ta}^{c}\right)  $ is bounded by the sum 
\begin{align}
& \p\left(\bigcup_{j=1}^{d}\bigcup_{i=1}^{\floor{\ta}}\left\{\abs{\bar{P}_{j}(i)-B_{j}(i)}\geq\frac{\ta^{\half-2\ep}}{3}\right\} \right) \nonumber \\
& +\p\left(\bigcup_{j=1}^{d}\bigcup_{i=1}^{\floor{\ta}}\left\{ \sup_{0\leq t\leq1}\abs{B_{j}(i+t)-B_{j}(i)}\geq\frac{\ta^{\half-2\ep}}{3}\right\} \right) \nonumber \\
   & +\p\left(\bigcup_{j=1}^{d}\bigcup_{i=1}^{\floor{\ta}}\left\{ \sup_{0\leq t\leq1}\abs{\bar{P}_{j}(i+t)-\bar{P}_{j}(i)}\geq\frac{1}{3}\ta^{\half-2\ep}\right\} \right) \label{eq:A_bound} \\
 \leq & K_{1}\exp\left(-\frac{K_{2}}{3} \ta^{1-2\ep}\right)+2\ta d\p\left(\sup_{0\leq t\leq1}B(t)\geq \frac{1}{3} \ta^{\half-2\ep}\right)+2 \ta d\p\left(\xi\geq\frac{1}{3}\ta^{\half-2\ep}\right)\nonumber \\
  \leq & K_{1}\exp\left(-\frac{K_{2}}{3} \ta^{1-2\ep}\right)+\frac{4d\ta^{\half+2\ep}}{\sqrt{2\pi}}\exp\left(-\frac{1}{18}\ta^{1-4\ep}\right)+2 \ta d\e\left[\exp(\xi)\right]\exp\left(1-\frac{1}{3}\ta^{\half-2\ep}\right),\nonumber 
\end{align}
where $\xi$ is a Poisson(1) random variable and we have used the
exponential Chebyshev inequality to estimate the Poisson probabilities,
along with the reflection principle and Mill's ratio estimate
$\p\left(\sup_{0<t<1}B(t)\geq x\right)=2\p\left(B(1)>x\right)\leq\frac{2}{\sqrt{2\pi}}\frac{1}{x}\exp\left(-x^{2}/2\right)$
to estimate the Brownian motion probabilities. Notice in particular
that this probability is $O(\exp\left(-\ta^{\al}\right))$ for some positive power $\al>0$ of $\ta$. 

We now analyze the expectation on the RHS of equation (\ref{eq:E_oo_gap})
by separately examining the contribution on $A_{\ep,\ta}$ and $A_{\ep,\ta}^{c}$.
On the event $A_{\ep,\ta}^{c}$ consider as follows: Let $E_{\ta}$
be the event
\[
E_{\ta}=\left\{ \sup_{1\leq i\leq d}P_{i}(\ta)<2\ta\right\} .
\]
On the event $E_{\ta}\cap A_{\ep,n}^{c}$ we expand the Vandermonde
determinant and use the bound $\abs{P_{i}(\ta)-P_{j}(\ta)}<4n$ to get
the bound:
\begin{align*}
 & \frac{\sqrt{\ta}}{h_{d}(\x^{0})}\e\left[\prod_{{i<j},{(i,j)\neq(a,b)}}\left(P_{j}(\ta)+x_{j}^{0}-P_{i}(\ta)-x_{i}^{0}\right)\one\left\{ \ta_{\vec{x}^{0}}^{P}>\ta\right\} \one\left\{ A_{\ta,\ep}^{c}\right\} \one\left\{ E_{\ta}\right\} \right]\\
\leq & \frac{\sqrt{n}}{x_{b}^{0}-x_{a}^{0}}\e\left[\prod_{{i<j},{(i,j)\neq(a,b)}}\left(\frac{4\ta}{x_{j}^{0}-x_{i}^{0}}+1\right)\one\left\{ \ta_{\vec{x}^{0}}^{P}>\ta\right\} \one\left\{ A_{\ta,\ep}^{c}\right\} \one\left\{ E_{\ta}\right\} \right]\\
\leq & \sqrt{\ta}\left(4\ta+1\right)^{\binom{d}{2}}\p\left(A_{\ta,\ep}^{c}\right).
\end{align*}
Since $\p\left(A_{\ta,\ep}^{c}\right)$ is exponentially small by equation
(\ref{eq:A_bound}), this contribution tends to zero as $\ta\to\infty$. The
contribution on the event $E_{\ta}^{c}\cap A_{\ep,\ta}^{c}$ is also seen
to be negligible by the following calculation:
\begin{align*}
 & \frac{\sqrt{\ta}}{h_{d}(\x^{0})}\e\left[\prod_{{i<j},{(i,j)\neq(a,b)}}\left(P_{j}(\ta)+x_{j}^{0}-P_{i}(\ta)-x_{i}^{0}\right)\one\left\{ \ta_{\vec{x}^{0}}^{P}>\ta\right\} \one\left\{ A_{\ta,\ep}^{c}\right\} \one\left\{ E_{\ta}^{c}\right\} \right]\\
\leq & \frac{\sqrt{\ta}}{x_{b}^{0}-x_{a}^{0}}\prod_{{i<j},{(i,j)\neq(a,b)}}\e\left[\left(\frac{P_{j}(\ta)-P_{i}(\ta)}{x_{j}^{0}-x_{i}^{0}}+1\right)^{\binom{d}{2}}\right]^{\binom{d}{2}^{-1}}\p\left(E_{\ta}^{c}\right)^{\binom{d}{2}^{-1}},
\end{align*}
where we have employed the generalized Cauchy-Schwarz/Holder inequality
$\e\left[\prod_{i=1}^{m}X_{i}\right]$ $\leq\prod_{i=1}^{m}\e\left[X_{i}^{m}\right]^{1/m}$.
Since $P_i(\ta)$ is mean $\ta$, the event $E_\ta^c$ is a large deviation event. By an exponential Chebyshev inequality for Poisson random variables, we have that $\p(E_{\ta}^{c})\leq d\exp\left(-\ta\left(2\ln2-1\right)\right)$ is exponentially small. Hence this too vanishes as $\ta\to\infty$. Since the total contribution on the
event $A_{\ep,\ta}^{c}$ vanishes as $\ta\to\infty$, it must be bounded
for all $\ta$ by some constant $C_{\ep}$. 

The contribution to equation (\ref{eq:E_oo_gap}) on the event $A_{\ep,\ta}$ is seen to be bounded above by $3^{\binom{d}{2}}\e_{\vec{0}}\left[\frac{1}{D_{k}(1)-D_{\ell}(1)}\right]$ by an identical argument employed in the proof of Lemma 4.11 of \cite{CorwinNica16}. A union bound completes the result. 
\end{proof}

\begin{lem}
\label{lem:inv-gap-bound} Fix any indices $1\le a<b\leq d$ . There
is a universal constant $C_{a,b}^{g}$ that bounds the expected inverse
gap size uniformly over all initial conditions $\vec{x}^{0}\in\bW$
and all times $n\in\bN$. Namely:
\[
\sup_{\ta > 0}\sup_{\vec{x}^{0}\in\bW\cap\bN^{d}}\e_{\x^{0}}\left[\frac{1}{\frac{1}{\sqrt{\ta}}\left(X_{b}(\ta)-X_{a}(\ta)\right)}\right]\leq C_{a,b}^{g}
\]
\end{lem}
\begin{proof}
Using Lemma \ref{lem:bounded_gaps}, the proof will follow by the same method as the proof of Lemma 4.13 of \cite{CorwinNica16}. The only other ingredient in this method is the following estimate for the hitting time $\nu_{\ta,\ep} = \inf\{ t \geq 0 : \vec{x}^0 + \vec{P}(t) \in \bS_{\ta,\ep} \}$ (where $P(\ta)$ are iid ordinary Poisson processes):

\begin{equation}
\e\left[\big|{h_{d}\big(\x^{\nogt}+P(\ta)\big)}\big|\cdot\one\left\{ \nu_{\ta,\ep}>\ta^{1-\ep}\right\} \right]\leq c_{1}h_{d}(\x^{\nogt})\exp\left(-c_{2}\ta^{\ep}\right).\label{eq:denisov_bound}
\end{equation}

In the setting of random walks in discrete time, this follows directly by application of Lemma 7 and/or Lemma 8 from \cite{denisov2010conditional}. The same argument applies to random walks in continuous time as we have here. (The proof of Lemma 7 in \cite{denisov2010conditional} goes by looking at blocks of steps with $n^{\ep}$ steps and applying the central limit theorem to each block when $n \to \infty$. The argument works equally well with compensated Poisson trajectories over time $\ta^{\ep}$ in the limit $\ta \to \infty$.) 
\end{proof}

\subsection{De-poissonization -- non-intersecting multinomial random walks}\label{sec:dePoiss}

In this section we will construct ``de-Poissonized'' versions of
the random processes  $\vec{X}(\ta)$ and $\vec{X}^{\prime}(\ta)$. This is needed in order to apply the discrete Tanaka theorem in the next subsection.
\begin{defn}
\label{def:dePoiss} Recall from Definition \ref{def:NIP} that $\X(\ta)$,
$\ta>0$ denotes $d$ non-intersecting Poisson processes.
Let $\vec{X}^{\prime}(\ta)$ be an independent copy. We define a 
pair of stochastic processes in discrete time $\vec{Y}(n),\vec{Y}^{\prime}(n)$,
$n\in\bN$, which are the de-Poissonized version of $\vec{X}(\ta),\vec{X}^{\prime}(\ta)$ as follows. First let $\ta_{n}$ be the time
at which the processes have made their $n$-th jump:
\[
\ta_{n}\defequal\inf\left\{ \ta:\sum_{i=1}^{d}X_{i}(\ta)-X_{i}(0)+\sum_{i=1}^{d}X_{i}^{\prime}(\ta)-X_{i}^{\prime}(0)\geq n\right\} ,
\]
and then set $\vec{Y}(n),\vec{Y}^{\prime}(n)$ to be the position
of the processes at this time:
\begin{align*}
\vec{Y}(n) \defequal& \vec{X}(\ta_{n}),\\
\vec{Y}^{\prime}(n) \defequal& \vec{X}^{\prime}(\ta_{n}).
\end{align*}
We refer to the pair $\vec{Y}(n),\vec{Y}^{\prime}(n)$ as the de-Poissonized
version of the pair $\vec{X}(\ta),\vec{X}^{\prime}(\ta)$.\end{defn}
\begin{lem}
\label{lem:dePoiss}The time between jumps are independent of each
other and are each exponentially distributed with mean $(2d)^{-1}$,
(i.e. $\ta_{n+1}-\ta_{n}\sim Exp(2d)$) and the processes $\vec{Y},\vec{Y}^{\prime}$ are
Markov processes that evolve according to the following rules:
\begin{align*}
\p\left(\vec{Y}(n+1)-\vec{Y}(n)=\vec{e}_{i},\vec{Y}^{\prime}(n+1)-\vec{Y}^{\prime}(n)=0\given{\vec{Y}(n)=\vec{y},\vec{Y}^{\prime}(n)=\vec{y}^{\prime}}\right) & =\frac{1}{2d}\frac{h_{d}\left(\vec{y}+\vec{e}_{i}\right)}{h_{d}(\vec{y})}\\
\p\left(\vec{Y}(n+1)-\vec{Y}(n)=0,\vec{Y}^{\prime}(n+1)-\vec{Y}^{\prime}(n)=\vec{e}_{i}\given{\vec{Y}(n)=\vec{y},\vec{Y}^{\prime}(n)=\vec{y}^{\prime}}\right) & =\frac{1}{2d}\frac{h_{d}\left(\vec{y}^{\prime}+\vec{e}_{i}\right)}{h_{d}(\vec{y}^{\prime})}
\end{align*}

\end{lem}

\begin{proof}
We observe from Definition \ref{def:NIP}, by explicitly calculating the determinant that appears, that the time until
the next jump for the non-intersecting Poisson processes can be calculated by
\begin{align*}
&\p\left(\vec{X}(\ta+\De\ta)=\vec{x},\vec{X}^{\prime}(\ta+\De\ta)=\vec{x}^{\prime}\given{\vec{X}(\ta)=\vec{x},\vec{X}^{\prime}(\ta)=\vec{x}^{\prime}}\right)\\
=&\p\left(\vec{X}(\ta+\De\ta)=\vec{x}\given{\vec{X}(\ta)=\vec{x}}\right)^{2}=e^{-2d\left(\De\ta\right)}.
\end{align*}
This shows that the time until the next jump is exponentially distributed with
mean $(2d)^{-1}$. Since by definition the $\vec{X},\vec{X}^{\prime}$ random
walk is absolutely continuous with respect to iid Poisson processes,
we know that almost surely only one jump occurs at any time. Hence, we have only to consider jumps of size 1 in each individual component.
By again computing the determinant that appears in Definition \ref{def:NIP}, we find the jump rates are characterized by: 
\begin{align*}
& \p\left(\vec{X}(\ta+\De\ta)=\vec{x}+\vec{e}_{i},\vec{X}^{\prime}(\ta+\De\ta)=\vec{x}^{\prime}\given{\vec{X}(\ta)=\vec{x},\vec{X}^{\prime}(\ta)=\vec{x}^{\prime}}\right) \\
= & e^{-d\left(\De\ta\right)}\left(\De\ta\right)\frac{h_{d}(\vec{x}+\vec{e}_{i})}{h_{d}(\vec{x})}+O(\De\ta^{2}),
\end{align*}
This shows that each walker is individually a Poisson process with jump rate $h_{d}(\x)^{-1}h_{d}(\vec{x}+\vec{e}_{i})$
for the $\vec{X}$ process, and $h_{d}(\x^{\prime})^{-1}h_{d}(\vec{x}^{\prime}+\vec{e}_{i})$
for the $\vec{X}^{\prime}$ process when at the position $\left(\vec{X},\vec{X}^{\prime}\right)=(\vec{x},\vec{x}^{\prime})$. (Note that since the function $h_d$ is harmonic for the simple random walk, the total jump rate of both the $\vec{X}$ and $\vec{X}^\prime$ process is always $d$.) By the definition of the $\left(\vec{Y},\vec{Y}^{\prime}\right)$ process and the fact that only one jump occurs at a time, we get the desired result.
\end{proof}

\begin{cor}
\label{cor:dePoiss_is_multinomial}
Let $\left(\vec{\xi},\vec{\xi}^{\prime}\right)\in\bN^{d}\times\bN^{d}$
be the multinomial random vector whose probability distribution is:
\begin{align*}
\p\left(\vec{\xi}=\vec{e}_{i},\vec{\xi}^{\prime}=\vec{0}\right) & =\frac{1}{2d}\ \forall1\leq i\leq d,\\
\p\left(\vec{\xi}=\vec{0},\vec{\xi}^{\prime}=\vec{e}_{i}\right) & =\frac{1}{2d}\ \forall1\leq i\leq d.
\end{align*}
Let $\left(\vec{Z}(n),\vec{Z}^{\prime}(n)\right)\ n\in\bN$ be the
stochastic process whose increments are given by an iid sequence of
these multinomial random vectors $\vec{Z}(n)\defequal \sum_{j=1}^{n}\vec{\xi}_{j}\ \vec{Z}^{\prime}(n)\defequal\sum_{j=1}^{n}\vec{\xi}_{j}^{\prime}$.
Then law of the de-Poissonized process $\left(\vec{Y}(n),\vec{Y}^{\prime}(n)\right)$
is identical in distribution to the law of $\left(\vec{Z}(n),\vec{Z}^{\prime}(n)\right)$
conditioned on non-intersection in the sense of the Doob $h$-transform for the event 
\begin{equation*}
\left\{ Z_{i}(n)<Z_{j}(n)\ \forall1\leq i<j\leq d,\ \forall n\in\bN\right\} \cap\left\{ Z_{i}^{\prime}(n)<Z_{j}^{\pr}(n)\ \forall1\leq i<j\leq d,\ \forall n\in\bN\right\}.
\end{equation*}\end{cor}
\begin{proof}
We firstly notice that the transition probabilities for each walk individually can be calculated from the transition rates given in Lemma \ref{lem:dePoiss} and the identity 
$\frac{1}{d}\sum_{i=1}^{d}\frac{h_{d}(\vec{x}+\vec{e}_{i})}{h_{d}(\vec{x})}=1$
(see e.g. Theorem 2.1 \cite{OConnell_Roch_Konig_NonCollidingRandomWalks} for this identity). This shows that each process $\vec{Y}(n)$ and $\vec{Y}^{\prime}(n)$ taken individually
is a Markov process with respect to its own filtration with transition
probabilities given by:
\begin{align}
\p\left(\vec{Y}(n+1)-\vec{Y}(n)=\vec{e}_{i}\given{\vec{Y}(n)=\vec{y}}\right) & =\frac{1}{2d}\frac{h_{d}\left(\vec{y}+\vec{e}_{i}\right)}{h_{d}(\vec{y})}. \label{eq:Y_gen} \\
\p\left(\vec{Y}(n+1)-\vec{Y}(n)=0\given{\vec{Y}(n)=\vec{y}}\right) & =\frac{1}{2}. \nonumber
\end{align}
Moreover, we notice that the interaction between the walks $\vec{Y}$ and $\vec{Y^\prime}$ is that one jumps precisely when the other does not. The result of the Corollary then follows by noticing that the jump rates for $\vec{Y},\vec{Y^\prime}$ exactly match those of the Doob $h$-transform by the Vandermonde determinant for the multinomial walks.
\end{proof}

\begin{rem}
Corollary \ref{cor:dePoiss_is_multinomial} shows that the de-Poissonized process $\left(\vec{Y}(n),\vec{Y}^{\prime}(n)\right)$ constructed in Definition \ref{def:dePoiss} has the same law as non-intersecting multinomial random walks. It will also be convenient for us to think of the reverse construction: first constructing the non-intersecting multinomial walks $\left(\vec{Y}(n),\vec{Y}^{\prime}(n)\right)$, and then using this to build the non-intersecting Poisson processes  $\left(\vec{X}(\ta),\vec{X}^{\prime}(\ta)\right)$. Corollary \ref{cor:Poiss_from_dePoiss} records this construction. 
\end{rem}

\begin{cor}
\label{cor:Poiss_from_dePoiss}
Suppose we are given the non-intersecting multinomial processes $\left(\vec{Y}(n),\vec{Y}^{\prime}(n)\right)$ which are constructed as in Corollary \ref{cor:dePoiss_is_multinomial} 
and an independent sequence $\left\{ \xi_{i}\right\} _{i=1}^{\infty}$
,$\xi_{i}\sim Exp(2d)$ of iid mean $\left(2d\right)^{-1}$ exponential
random variables. Let $n_{\ta}=\inf\left\{ n:\sum_{i=1}^{n}\xi_{i}>\ta\right\} .$
Then the process $\left(\vec{X}(\ta),\vec{X}^{\prime}(\ta)\right)=\left(\vec{Y}(n_{\ta}),\vec{Y}^{\prime}(n_{\ta})\right)$
is a realization of the non-intersecting Poisson processes.\end{cor}
\begin{lem}
\label{lem:dePoiss_bounded_gaps} For any indices
$1\leq a<b\leq d$, there exists a constant $C_{a,b}^{g,Y}$
so that each of the de-Poissonized random walks $\vec{Y}$ or $\vec{Y^\prime}$ have:
\[
\sup_{n\in\bN}\sup_{\vec{x}^{0}\in\bW\cap\bN^{d}}\e_{\x^{0}}\left[\frac{1}{\frac{1}{\sqrt{n}}\left(Y_{a}(n)-Y_{b}(n)\right)}\right]\leq C_{a,b}^{g,Y}.
\]
\end{lem}
\begin{proof}
By the result of Lemma \ref{lem:dePoiss} of the de-Poissonized random
walks $\vec{Y},\vec{Y}^{\prime},$ we know that $\vec{Y}(n)=\vec{X}(\ta_{n})$
where the random times $\ta_{n}\dequal\sum_{i=1}^{n}\xi_{i}$ are
distributed as a sum of $n$ exponential random variables of mean
$\left(2d\right)^{-1}$. Thus, denoting by $\rho_{\ta_{n}}(\cdot)$
the density of the $\ta_{n}$ random variable and applying the bound
from Lemma \ref{lem:bounded_gaps}, we have 
\begin{align*}
\sqrt{n}\e\left[\frac{1}{Y_{b}(n)-Y_{a}(n)}\right] & = \sqrt{n}\intop_{0}^{\infty}\e\left[\frac{1}{Y_{b}(n)-Y_{a}(n)}\given{\ta_{n}=t}\right]\rh_{\ta_{n}}(t)\dd t\\
 & = \sqrt{n}\intop_{0}^{\infty}\e\left[\frac{1}{X_{b}(t)-X_{a}(t)}\right]\rho_{\ta_{n}}\left(t\right)\dd t \leq C_{a,b}^{g}\e\left[\sqrt{\frac{n}{\ta_{n}}}\,\,\right].
\end{align*}
Since $\ta_{n}$ is a sum of $n$ iid exponential random variables
of mean $\left(2d\right)^{-1}$, it is easily verified that the above
expectation is bounded. (One can explicitly
compute $\e\left[\sqrt{{n}/{\ta_{n}}}\right] = \sqrt{n}\Gamma(n-\half)\Ga(n)^{-1}$ for $n\geq 2$.)
\end{proof}

\begin{lem}
\label{lem:exp-mom-control-for-gaps} For any $\tf>0$ and any indices
$1\leq a<b\leq d$, the collection of processes
\[
\left\{ \frac{1}{\sqrt{N}}\sum_{i=1}^{\floor{tN}}\frac{1}{Y_{b}(i)-Y_{a}(i)}:t\in[0,\tf]\right\} _{N\in\bN},
\]
is weakly exponential moment controlled as $t\to0$.
\end{lem}

\begin{proof}
Using the bound from Lemma \ref{lem:dePoiss_bounded_gaps}, the proof
follows exactly in the same way as the proof of Lemma 4.14 from \cite{CorwinNica16}, which is obtained 
by estimating the moments of the random process.\end{proof}
\begin{lem}
\label{lem:position_dePoiss} Recall that $\vec{Y}_k$ denotes the position of the $k$-th line of the de-Poissonized walks. The collection of processes:
\[
\left\{ \frac{1}{\sqrt{N}}\abs{2dY_{k}\left(\floor{tN}\right)-\floor{tN}},\ t\in[0,\tf]\right\} _{N\in\bN},
\]
is weakly exponential moment controlled as $t \to 0$.
\end{lem}
\begin{proof}
Let $\{ \xi_i \}_i^\infty$ be the family of mean $(2d)^{-1}$ exponential random variables that relate $\vec{Y},\vec{Y}^\prime$ with $\vec{X},\vec{X}^\prime$. Set $\ta_n = \sum_{i=1}^n \xi_i$. Notice that to prove the lemma it suffices to show that $\left\{ \frac{1}{\sqrt{N}}\abs{Y_{k}\left(\floor{tN}\right)-\ta_{\floor{tN}}},\ t\in[0,\tf]\right\} $
is weakly exponential moment controlled, because then the result follows
by the inequality:
\[
\frac{1}{\sqrt{N}}\abs{2dY_{k}\left(\floor{tN}\right)-\floor{tN}}\leq2d\frac{\abs{Y_{k}(\floor{tN})-\ta_{\floor{tN}}}}{\sqrt{N}}+2d\frac{\abs{\ta_{\floor{tN}}-\frac{1}{2d}\floor{tN}}}{\sqrt{N}},
\]
and the fact that a sum of weakly exponential moment controlled random
variables is again weakly exponential moment controlled by Lemma \ref{lem:sum-of-exp}
(it is easily verified that collection $\left\{ \frac{1}{\sqrt{N}}\abs{2d\ta_{\floor{tN}}-\floor{tN}}\right\}_{N\in \bN} $
is weakly exponential controlled since $\ta_{\floor{tN}}-\frac{1}{2d}\floor{tN}\dequal\sum_{i=1}^{\floor{tN}}\left(\xi_{i}-\frac{1}{2d}\right)$,
$\xi_{i}\sim Exp(2d)$ is a sum of $\floor{tN}$ mean zero random
variables which have finite exponential moments.) But by the definition
of the de-Poissonized random walks $\vec{Y},\vec{Y}^{\prime},$ we
know that $\vec{Y}(\floor{tN})=\vec{X}(\ta_{\floor{tN}})$. Writing $\rho_{\ta_{tN}}$ for the density of $\ta_{\floor{tN}}$, and letting $C$ be the constant from Lemma \ref{cor:pos_bound}, make the following estimate for any $\al\in\bR$:
\begin{align*}
&\p\left(\abs{Y_{k}\left(\floor{tN}\right)-\ta_{\floor{tN}}}>\al\sqrt{N}\right)\\ 
=&\intop_{0}^{\infty}\p\left(\abs{Y_{k}\left(\floor{tN}\right)-\ta_{\floor{tN}}}>\al\sqrt{N}\given{\ta_{\floor{tN}}=s}\right)\rho_{\ta_{\floor{tN}}}(s)\dd s\\
 =&\intop_{0}^{\infty}\p\left(\abs{X_{k}\left(s\right)-s}>\al\sqrt{N}\right)\rho_{\ta_{\floor{tN}}}(s)\dd s\\
 \leq&\intop_{0}^{\infty}\left(C\exp\left(-c\frac{N\al^{2}}{s}\right)+C\exp\left(-c\sqrt{N}\al\right)\right)\rho_{\ta_{\floor{tN}}}(s)\dd s\\
 \leq& C\exp\left(-c\frac{N\al^{2}}{Nt}\right)Nt+C\p\left(\ta_{\floor{t n}}>Nt\right)+C\exp\left(-c\sqrt{N}\al\right),
\end{align*}
where we have split the integral into the contribution from $s\leq Nt$
and $s>Nt$ to get the last inequality. Notice that typically
$\ta_{\floor{tN}}\sim\frac{1}{2d}tN$, so $\p\left(\ta_{\floor{tN}}>tN\right)$
is a large deviation event; an application of the exponential Chebyshev
inequality gives $\p\left(\ta_{\floor{tN}}>tN\right)\leq \exp\left(-tN\left(1-\ln\left(\frac{2d}{2d-1}\right)\right)\right)$. Finally, the weak
exponential control follows from this bound by Lemma \ref{lem:weak_mom_contrl_lemma} and Lemma \ref{lem:subgaussian_is_exp_mom}.
\end{proof}



\begin{lem}
\label{lem:exp_mom_for_diffs}
Denote by $\De Y_{k}(i) \defequal Y_{k}(i+1)-Y_{k}(i)$. We have that the collection
of processes
\[
\left\{ \frac{1}{\sqrt{N}}\sum_{n=1}^{\floor{tN}}\e\left[2d\De Y_{k}(n)-1\given{\vec{Y}(n)}\right]:t\in[0,\tf]\right\} _{N\in\bN},
\]
is weakly exponential moment controlled.\end{lem}
\begin{proof}
By expanding the Vandermonde determinants that appear in  equation \eqref{eq:Y_gen}, we have that:
\begin{align*}
\p\left(\De Y_{k}(n)=+1\given{\vec{Y}(n)=\vec{x}}\right) & =\frac{1}{2d}\prod_{i\neq k}\left(1+\frac{1}{x_{k}-x_{i}}\right),
\end{align*}
and hence we compute
\begin{align*}
\e\left[2d\De Y_{k}(n)-1\given{\vec{Y}(n)=\vec{x}}\right] & \leq\abs{\prod_{i\neq k}\left(1+\frac{1}{x_{k}-x_{i}}\right)-\left(1+\sum_{i\neq k}\frac{1}{x_{k}-x_{i}}\right)}+\sum_{i\neq k}\abs{\frac{1}{x_{k}-x_{i}}}\\
 & \leq2^{d}\sum_{i\neq k}\abs{\frac{1}{x_{k}-x_{i}}},
\end{align*}
where we have applied the inequality from Lemma 4.15 of \cite{CorwinNica16}, which holds since
$\frac{1}{x_{k}-x_{i}} \leq 1$. Finally then:
\[
\frac{1}{\sqrt{N}}\sum_{n=1}^{\floor{tN}}\e\left[2d\De Y_{k}(n)-1\given{\vec{Y}(n)}\right]\leq2^{d}\sum_{i\neq k}\frac{1}{\sqrt{N}}\sum_{n=1}^{\floor{tN}}\frac{1}{\abs{Y_{k}(n)-Y_{i}(n)}},
\]
and the result follows by the weak exponential moment control established
in Lemma \ref{lem:exp-mom-control-for-gaps} and since weak exponential
moment control is preserved under finite sums by Lemma \ref{lem:sum-of-exp}. 
\end{proof}

\subsection{\textmd{\normalsize{}\label{sub:exp_mom_overlap}}Overlap times of
non-intersecting multinomial random walks}

In this subsection we establish the exponential moment control for overlap times by using a discrete version of Tanaka's formula. This is like Tanaka's formula in that it relates the overlap time to increments of the random walk. For our purposes, we only need an upper bound on the overlap time which simplifies the proof somewhat. This inequality will bound the overlap time by a finite sum of quantities, each of which is analyzed to establish
exponential moment control. The methods here are similar to those used for non-intersecting random walks in \cite{CorwinNica16}.

\begin{lem}
\label{lem:tanaka_prelemma}Let $u\in\bN$ be any positive integer.
Suppose that $A(n)$ is an integer-valued process whose increments
are always either $+u$ or $-1$, i.e. we have $\De A(n)\defequal A(n+1)-A(n)\in\left\{ u,-1\right\} .$
Then for any $C\in\bZ$ we have the inequality
\begin{equation}
\one\left\{ \abs{A(n)-C}=0\right\} \abs{\De A(n)}+\left(\De A(n)\right)\sgn\left(A(n)-C\right)\leq\abs{A(n+1)-C}-\abs{A(n)-C},\label{eq:tanaka_ineq}
\end{equation}
where the $\sgn$ function uses the convention that $\sgn(x)=x/\abs x$
for $x\neq0$ and $\sgn(0)=0$.\end{lem}
\begin{proof}
The proof goes by comparing the LHS and the RHS in several cases.

\noindent \uline{Case i)} $\abs{A(n)-C}=0$. 

It is easily verified that both sides of equation (\ref{eq:tanaka_ineq}) are equal to $\abs{\De A(n)}$. 

\noindent \uline{Case ii)} $\abs{A(n)-C}\neq0$, and $\sgn\left(A(n)-C\right)=\sgn\left(A(n+1)-C\right)$
or $A(n+1)-C=0$. 

We may write that $\abs{A(n)-C}=\sgn\left(A(n)-C\right)\left(A(n)-C\right)$
and $\abs{A(n+1)-C}=\sgn\left(A(n+1)-C\right)\left(A(n+1)-C\right)$.
Factor $\sgn\left(A(n)-C\right)$
from the RHS of (\ref{eq:tanaka_ineq}) to get: 
\begin{align*}
\abs{A(n+1)-C}-\abs{A(n)-C} & =\sgn\left(A(n)-C\right)\left(\left(A(n+1)-C\right)-\left(A(n)-C\right)\right)\\
 & =\sgn\left(A(n)-C\right)\De A(n).
\end{align*}
This is exactly equal to the LHS of (\ref{eq:tanaka_ineq})
since $\abs{A(n)-C}\neq0$ here.

\noindent \uline{Case iii)} $A(n)-C<0$ and $A(n+1)-C>0$.

This can happen only if $\De A(n)=u$ and $A(n)-C\in\left[-u+1,-1\right]$. In particular $A(n)-C\geq-u$. In this case, the LHS of (\ref{eq:tanaka_ineq})
is $-u$, while the RHS is
\begin{align*}
\abs{A(n+1)-C}-\abs{A(n)-C} &=\De A(n)+2\left(A(n)-C\right)\\
 & =u+2\left(A(n)-C\right) \geq -u,
\end{align*}
which verifies the desired inequality.

\noindent \uline{Case iv)} $A(n)-C>0$ and $A(n+1)-C<0$

This case is impossible by hypothesis on the process $A(n)$, since $\De A(n)\geq-1$ always.
\end{proof}
\begin{lem}
\label{lem:tanaka}Let $u\in\bN$ be any positive integer. Suppose
that $A(n)$ and $B(n)$ are integer valued processes so that the
increments are always $+u$ or $-1$, i.e. we have that $\De A(n)\in\left\{ u,-1\right\} $
and $\De B(n)\in\left\{ u,-1\right\} $. Then
\begin{align*}
\sum_{i=0}^{n}\one\left\{ A(i)=B(i)\right\} &\leq  \abs{A(n+1)-B(n+1)}-\abs{A(0)-B(0)} \\
& -\sum_{i=0}^{n}\mathrm{sgn}\left(A(i)-B(i)\right)\De A(i)+\sum_{i=0}^{n}\mathrm{sgn}\left(A(i+1)-B(i)\right)\De B(i).
\end{align*}
\end{lem}
\begin{proof}
First write that
\begin{align*}
\abs{A(i)-B(i)}-\abs{A(i-1)-B(i-1)} & = \abs{A(i)-B(i)}-\abs{A(i)-B(i-1)}\\
 & +\abs{A(i)-B(i-1)}-\abs{A(i-1)-B(i-1)}.
 \end{align*}
The result then follows by applying Lemma \ref{lem:tanaka_prelemma} twice, first to the $B$-process with $C=A(i+1)$ and then again to the $A$-process with
$C=B(i)$,  
and then summing the resulting inequality from $i=0$
to $n$.\end{proof}
\begin{lem}
\label{lem:exp_mom_dePoiss} Define the overlap time for the non-intersecting multinomial random walks between the $k$-th line $Y_{k}$ and the $\ell$-th line $Y_{\ell}^{\prime}$
by:
\[
Q_{k,\ell}[0,n]=\sum_{i=0}^{n}\one\left\{ Y_{k}(i)=Y_{\ell}^{\prime}(i)\right\} .
\]
Then, for any fixed $\tf>0$, and any indices $1\leq k < \ell\leq d$, the
collection:
\[
\left\{ \frac{1}{\sqrt{N}}Q_{k,\ell}[0,\floor{tN}]:t\in[0,\tf]\right\} _{N\in\bN},
\]
is weakly exponential moment controlled as $t\to0$.\end{lem}
\begin{proof}
For notational convenience, we use the shorthand $\De F(i)\defequal F(i+1)-F(i)$. We will apply the upper bound for the overlap time from Lemma
\ref{lem:tanaka}, to the processes:
\begin{equation}
A(i)\defequal 2dY_{k}(i)-i, \; A^\prime(i)\defequal 2dY_{\ell}^{\prime}(i)-i,
\end{equation}
which have increments of either $2d-1$ or $-1$. (Notice that the increments of these process are $\De A(i)=2d\De Y_{k}(i)-1$
and $\De A^\prime(i)=2d\De Y_{\ell}^{\prime}(i)-1$). By the definition of
$Q_{k,\ell}[0,\floor{tN}]$ and application of Lemma \ref{lem:tanaka} we have:
\begin{align}
&Q_{k,\ell}[0,\floor{tN}]\\
 =& \sum_{i=0}^{\floor{tN}}\one\left\{ A(i)=A^{\prime}(i)\right\} \leq \abs{A(\floor{tN}+1)}+\abs{A^{\prime}(\floor{tN}+1)}+\abs{S(\floor{tN})}+\abs{S^{\prime}\left(\floor{tN}\right)},\label{eq:O_kl_discrete_Tanaka}
\end{align}
where we define $S(n)$ and $S^{\prime}(n)$ by:
\begin{equation}
S\left(n\right) \defequal \sum_{i=0}^{n}\sgn\left(A(i)-A^{\pr}(i)\right)\De A(i),\; S^{\prime}\left(n\right)\defequal \sum_{i=0}^{n}\sgn\left(A(i+1)-A^{\pr}(i)\right)\De A^{\pr}(i).
\end{equation}
By Lemma \ref{lem:sum-of-exp}, to see the exponential moment control
for $N^{-\half}Q_{k,\ell}[0,\floor{tN}],$ we have only to verify that the four terms that appear on the RHS of equation (\ref{eq:O_kl_discrete_Tanaka})
are each weakly exponential moment controlled. The first two terms
on the RHS of equation (\ref{eq:O_kl_discrete_Tanaka}) are weakly exponential
moment controlled by Lemma \ref{lem:position_dePoiss}. We show
that $\left\{ \abs{N^{-\half}S\left(\floor{tN}\right)}:t\in[0,\tf]\right\} _{N\in\bN}$
and $\left\{ \abs{N^{-\half}S\left(\floor{tN}\right)}:t\in[0,\tf]\right\} _{N\in\bN}$are
weakly exponential moment controlled as $t\to0$ as follows. First notice that by the triangle inequality that
\begin{equation}
\abs{\frac{1}{\sqrt{N}}S(\floor{tN})} \leq \abs{\frac{1}{\sqrt{N}}M(\floor{tN})}+\frac{1}{\sqrt{N}}\sum_{i=1}^{\floor{tN}}\abs{\e\left[\De A(i) \given{\vec{Y}(i)}\right]},\label{eq:S}
\end{equation}
where we define
\[
M(n)\defequal\sum_{i=0}^{n}\sgn\left(A(i)-A^{\prime}(i)\right)\left(\De A(i) -\e\left[\De A(i) \given{\vec{Y}(i)}\right]\right).
\]
By Lemma \ref{lem:sum-of-exp} it suffices to check that both terms
that appear on the RHS of equation (\ref{eq:S}) are weakly exponential moment
controlled. The second term in equation (\ref{eq:S}) is weakly exponential
moment controlled by application of Lemma \ref{lem:exp_mom_for_diffs}. To handle the first term, we notice that $\left\{ M(n)\right\} _{n\in\bN}$
is a martingale with respect to the filtration $\cF_{n}\defequal\si\left(\vec{Y}(1),\vec{Y}^{\pr}(1),\ld,\vec{Y}(n+1),\vec{Y}^{\pr}(n+1)\right)$. Its increments are given by
\begin{equation}
M(n)-M(n-1)=\sgn\left(A(n)-A^{\pr}(n)\right)\left(\De A(n) -\e\left[\De A(n) \given{\vec{Y}(n)}\right]\right),\label{eq:M_inc}
\end{equation}
which have $\e\left[M(n)-M(n-1)\given{\cF_{n-1}}\right]=0$ as
$\sgn\left(A(n)-A^{\pr}(n)\right)=\sgn\left(Y_k(n)-Y_{\ell}^{\pr}(n)\right)$ is $\cF_{n-1}$ measurable
and since $\vec{Y}(\cdot)$ is a Markov process. Moreover, since $\De A (n)\in\{-1,+2d-1\}$,
we also notice from equation (\ref{eq:M_inc}) that $\abs{M(n)-M(n-1)}\leq 2d$.
We can therefore apply Azuma's inequality for martingales
with bounded differences (see e.g. Lemma 4.1 of \cite{McDiarmid}).
This gives that for any $N\in\bN$
\[
\p\left(\frac{1}{\sqrt{N}}\abs{M(\floor{tN})}>\al\right)\leq2\exp\left(-\frac{\al^{2}}{2t(2d)^{2}}\right).
\]
By Lemma \ref{lem:subgaussian_is_exp_mom}, this shows that $\left\{ \abs{N^{-\half}M\left(\floor{tN}\right)}:t\in[0,\tf]\right\} _{N\in\bN}$
is exponential moment controlled as desired. 

The proof that $\left\{ \abs{N^{-\half}S\left(\floor{tN}\right)}:t\in[0,\tf]\right\} _{N\in\bN}$
is exponential moment controlled is similar using the martingale
\[
M^\prime(n)\defequal\sum_{i=0}^{n}\sgn\left(A(i+1)-A^{\prime}(i)\right)\left(\De A^{\pr}(i) -\e\left[\De A^{\pr}(i)\given{\vec{Y^{\prime}}(i)}\right]\right),
\]
which is a martingale w.r.t. $\cF^{\prime}_{n}\defequal\si\left(\vec{Y}(1),\vec{Y}^{\pr}(1),\ld,\vec{Y}(n+1),\vec{Y}^{\pr}(n+1),\vec{Y}(n+2)\right)$.
\end{proof}

\subsection{Overlap times of non-intersecting Poisson processes  and bridges}

In this section we prove that the overlap times for non-intersecting
Poisson processes are weakly exponential moment controlled by comparison
to the overlap time for the de-Poissonized walks.
\begin{lem}
\label{lem:exp_mom_control_NIP}
Recall the definition of the overlap time $O_{k,\ell}[a,b]$ for the
Poisson random walks. For any fixed $\tf>0$, and any indices $1\leq k,\ell\leq d$,
the collection:
\[
\left\{ \frac{1}{\sqrt{N}}O_{k,\ell}[0,tN]:t\in[0,\tf]\right\} _{N\in\bN},
\]
 \textup{is weakly exponential moment controlled as $t\to0$.}\end{lem}
\begin{proof}
By Lemma \ref{lem:dePoiss}, we know that we can construct a coupling
of the non-intersecting Poisson processes  $\left(\vec{X}(t),\vec{X}^{\prime}(t)\right)$
$t>0$ and the non-intersecting multinomial walks $\left(\vec{Y}(n),\vec{Y}^{\prime}(n)\right), n\in\bN$
along with a sequence $\left\{ \xi_{i}\right\} _{i=1}^{\infty}$ of
iid mean $(2d)^{-1}$ exponential random variables so that $\vec{X}(\ta_{n})=\vec{Y}(n)$
and $\vec{X}^{\prime}(\ta_{n})=\vec{Y}^{\prime}(n)$ where $\ta_{n}\defequal\sum_{i=1}^{n}\xi_{i}$.
In this coupling, the overlap time $O_{k,\ell}$ between $X$ and
$X^{\prime}$ can be written in terms of the $\vec{Y}(n),\vec{Y}^{\prime}(n),\left\{ \xi_{i}\right\} _{i=1}^{\infty}$
as
\begin{equation*}
O_{k,\ell}[0,tN] \leq \sum_{i=1}^{\et(tN)}\intop_{\ta_{i}}^{\ta_{i+1}}\one\left\{ X_{k}(\ta)=X_{\ell}(\ta)\right\} \dd\ta =\sum_{i=1}^{\et(tN)}\xi_{i+1}\one\left\{ Y_{k}(i)=Y_{\ell}(i)\right\}, 
\end{equation*}
 where $\et(t)=\max\left\{ n:\ta_{n}\leq t\right\} $ is the number
of steps which have been taken up to time $t$. Since the $\xi_{i}$
are independent of the walk $\vec{Y}$, the only thing that is relevant for the distribution
of the above is the number of times $i$ for which $Y_{k}(i)=Y_{\ell}^{\prime}(i)$. This is exactly counted by the discrete overlap
times for the multinomial walkers $Q_{k,\ell}[0,\et(tN)]$. In particular,
if we label the indices $i$ for which $\left\{ Y_{k}(i)=Y_{\ell}(i)\right\} $
as $i_{1},i_{2},\ld$, then we have:
\begin{align}
\nonumber &\Big\{ O_{k,\ell}[0,tN]>x\sqrt{N}\Big\} \\ 
\label{eq:overlap_inclusion} \subset&\Big\{ \et(tN)>c_{1}tN\Big\} \cup\Big\{ Q_{k,\ell}\left[0,c_{1}tN\right]>c_{2}x\sqrt{N}\Big\} \cup\Big\{ \sum_{j=1}^{\floor{c_{2}x\sqrt{N}}}\xi_{i_j}>x\sqrt{N}\Big\},  
\end{align}
where $c_{1},c_{2}$ are some to-be-determined constants that depend
on $d$. Since $\{ \et(tN) > c_1 tN \} = \{ \sum_{i=1}^{c_1 t n} \xi_i < tN \}$, we can use the exponential Chebyshev inequality to estimate
\begin{equation*}
\p\left(\et(tN)>c_{1}tN\right) \leq\e\left[\exp\left(-\xi_{1}\right)\right]^{c_{1}tN}\exp\left(tN\right) =\exp\left(tN\left(c_{1}\ln\left(\frac{2d}{2d+1}\right)+1\right)\right),
\end{equation*}
and similarly we have: 
\begin{equation*}
\p\left(\sum_{i=1}^{\floor{c_{1}x\sqrt{N}}}\xi_{i_j}>x\sqrt{N}\right) \leq \exp\left(x\sqrt{N}\left(c_{1}\ln\left(\frac{2d}{2d-1}\right)-1\right)\right).
\end{equation*}
If we choose $c_{1}$and $c_{2}$ to be any constants so that $c_{1}\ln\left(\frac{2d}{2d+1}\right)+1<0$
and $c_{2}\ln\left(\frac{2d}{2d-1}\right)-1<0$, then these probabilities
are both exponentially small. Thus by the inclusion from equation \eqref{eq:overlap_inclusion}  we have:
\begin{align*}
\p\left(\frac{1}{\sqrt{N}}O_{k,\ell}[0,tN]>x\right)& \leq \p\left(\frac{Q_{k,\ell}\left[0,c_{1}tN\right]}{c_{2}\sqrt{N}}>x\right)+\exp\left(tN\left(c_{1}\ln\left(\frac{2d}{2d+1}\right)+1\right)\right)\\
&+\exp\left(x\sqrt{N}\left(c_{2}\ln\left(\frac{2d}{2d-1}\right)-1\right)\right).
\end{align*}
It is easily verified from Definition \ref{def:exp_mom_control} and the conclusion of Lemma  \ref{lem:exp_mom_dePoiss} that
for any fixed positive constants $c_{1},c_{2}$ that the process $\left\{ c^{-1}_{2} N^{-\half}Q_{k,\ell}\left[0,c_{1}tN\right]:t\in[0,\tf]\right\} $
is weakly exponential moment controlled. 
Finally the weak exponential moment control for $O_{k,\ell}$ follows
by Lemma \ref{lem:weak_mom_contrl_lemma}.
\end{proof}

\begin{prop} \label{lem:exact_formula}Recall from
Definition \ref{def:NIP} the probability function $q_{\ta}\left(\vec{x},\vec{y}\right)$
which was used in the construction of the non-intersecting Poisson
walks. We have the following exact formula:
\begin{align*}
q_{\ta}\left(\vec{\de}_{d}\left(0\right),\vec{x}\right) & =\ta^{-d(d+1)/2}\left(\prod_{i=1}^{d}\mu\left(\ta,x_{i}\right)\right)\cdot h_{d}\left(x_1,x_2,\ld,x_{d}\right).
\end{align*}
\end{prop}
\begin{proof}
The determinant that defines $q_\ta$ in this case is explicitly calculated as part of the proof of Proposition 3.3 in  \cite{OConnell_Roch_Konig_NonCollidingRandomWalks}. 
\end{proof}

\begin{lem}
\label{lem:RN_bound}
Fix any $\zf\in\bR$ and $\tf>0$. There is a constant $C_{R}^{(\tf,\zf)}<\infty$
so that the Radon-Nikodym derivative of the rescaled non-intersecting
Poisson bridge\textup{ $\X^{(N),(\tf,\zf)}(t)$} with respect to the
rescaled non-intersecting Poisson process $\frac{1}{\sqrt{N}}\X(\floor{tN})$
is uniformly bounded by $C_{R}^{(\tf,\zf)}$ over all possible positions
at all times $t$ that have $t<\frac{2}{3}\tf$:
\[
\sup_{N\in\bN}\sup_{t<\frac{2}{3}\tf}\sup_{\vec{z}\in\left(\frac{\bZ}{\sqrt{N}}\right)^{d}}\frac{\p\left(\X^{(N),(\tf,\zf)}(t)=\vec{z}\right)}{\p\left(\frac{1}{\sqrt{N}}\X(tN)=\vec{z}\right)}\leq C_{R}^{(\tf,\zf)}.
\]
\end{lem}
\begin{proof}
By Definition \ref{def:NIPb}, we have
\begin{equation}
\frac{\p\left(\X^{(\taf,\xf)}(\ta)=\x\right)}{\p\left(\X(\ta)=\x\right)}=\frac{q_{\taf-\ta}\left(\x,\vec{\de}_{d}(\xf)\right)}{q_{\taf}\left(\vec{\de}_{d}(0),\vec{\de}_{d}(\xf)\right)}\frac{h_{d}\left(\vec{\de}_{d}(0)\right)}{h_{d}(\x)}. \label{eq:P_NI}
\end{equation}
From our exact formula from Proposition \ref{lem:exact_formula} for $q_{\ta}\left(\vec{\de}_{d}\left(0\right),\vec{x}\right)$ and the time reversal of this formula which yields a similar formula for $q_{\ta}\left(\vec{x},\vec{\de}_{d}(\xf)\right)$, we have:
\begin{align*}
q_{\ta}\left(\vec{\de}_{d}\left(0\right),\vec{x}\right) & =\ta^{-d(d+1)/2}\left(\prod_{i=1}^{d}\mu\left(\ta,x_{i}\right)\right)\cdot h_{d}\left(x_{1},x_{2},\ld,x_{d}\right).\\
q_{\ta}\left(\vec{x},\vec{\de}_{d}(\xf)\right) & =\ta^{-d(d+1)/2}\left(\prod_{i=1}^{d}\mu\left(\ta,\xf+i-1-x_{i}\right)\right)\cdot h_{d}\left(\xf+d-x_{d},\ld,\xf+d-x_{1}\right) \\
& =\ta^{-d(d+1)/2}\left(\prod_{i=1}^{d}\mu\left(\ta,\xf+i-1-x_{i}\right)\right)\cdot h_{d}\left(x_{1},\ld,x_{d}\right).
\end{align*}
Thus we conclude after plugging these formulas into equation \eqref{eq:P_NI} and observing that the Vandermonde factors $h_d$ cancel out that we remain with
\begin{equation}
\frac{\p\left(\X^{(\taf,\zf)}(\ta)=\x\right)}{\p\left(\X(\ta)=\x\right)}=\left(\frac{\taf-\ta}{\taf}\right)^{-d(d+1)/2}\prod_{i=1}^{d}\frac{\mu\left(\taf-\ta,\xf+i-1-x_{i}\right)}{\mu(\taf,\xf)}.\label{eq:RN_formula}
\end{equation}
Putting in now the scaling $\ta=Nt$,$\taf=N\tf,$ $\x=Nt+\sqrt{N}\vec{z}$,
$\xf=N\tf+\sqrt{N}\zf$, we see by the local limit theorem for the
Poisson process Proposition \ref{prop:local_clt}, that
\begin{align*}
\limsup_{N\to\infty}\frac{1}{\sqrt{N}\tf\mu(\taf,\xf)} & =\sqrt{2\pi}\exp\left(\frac{\zf^{2}}{2\tf}\right),\\
\limsup_{N\to\infty}\sqrt{N}\tf\mu\left(\taf-\ta,\xf+i-1-x_{i}\right) & =\frac{1}{\sqrt{2\pi}}\exp\left(-\frac{\left(\zf-z\right)^{2}}{2\left(\tf-t\right)}\right)\leq\frac{1}{\sqrt{2\pi}},
\end{align*}
and hence, putting this result back into equation (\ref{eq:RN_formula}),
we conclude that
\begin{equation}
\limsup_{N\to\infty}\sup_{t<\frac{2}{3}\tf}\sup_{\vec{z}\in\left(\frac{\bZ}{\sqrt{N}}\right)^{d}}\frac{\p\left(\X^{(N),(\tf,\zf)}(t)=\vec{z}\right)}{\p\left(\frac{1}{\sqrt{N}}\X(tN)=\vec{z}\right)}\leq3^{d(d+1)/2}\exp\left(\frac{\zf^{2}}{2\tf}\right).\label{eq:limsup_RN}
\end{equation}
Since this limsup as $N\to\infty$ is finite, we conclude that the $\sup$ over all $N\in\bN$, as in the LHS
of equation (\ref{eq:limsup_RN}), is finite as desired.\end{proof}


\begin{prop}
\label{prop:ON_exp_mom} Recall the definition of the rescaled overlap
time $O^{(N),(\tf,\zf)}[0,t]$ from Definition \ref{def:overlap_times_NIW}.
For any $\tf>0$ and $\zf\in\bR$, the collection of rescaled overlap
times
\[
\left\{ O^{(N),(\tf,\zf)}[0,t],t\in[0,\tf]\right\} _{N\in\bN},
\]
 is weakly exponential moment controlled as $t\to0$.\end{prop}
\begin{proof}
The proof is very similar to the proof of Proposition 4.23 from \cite{CorwinNica16} using the exponential moment control for the non-intersecting Poisson processes  from Lemma \ref{lem:exp_mom_control_NIP}, the Radon-Nikodym bound between Poisson processes  and Poisson bridges from Lemma \ref{lem:RN_bound}, and the fact that weak exponential moment control is closed under addition as in Lemma \ref{lem:sum-of-exp}.
\end{proof}

\section{\texorpdfstring{$L^2$}{L2} bounds -- Proof of Propositions \ref{prop:uniform_exp_moms}, \ref{prop:D3}, \ref{prop:D4}} \label{sec:L2b}
This section uses the weak exponential moment control established in Proposition  \ref{prop:ON_exp_mom} to get bounds the $L^2$ norm of the $k$-point correlation function $\ps_k^{(\tf,\zf)}$. These arguments are a semi-discrete version of those used in Section 5 of \cite{CorwinNica16}.
\begin{lem}
\label{lem:exp-mom-moments} If $\left\{ Z^{(N)}(t)\ :\ t\in\left[0,\tf\right]\right\} _{N\in\bN}$
is weakly exponential moment controlled as $t\to0$, then for each
$t\in[0,\tf]$, there exists $N_{0}$ such that $Z^{(N)}(t)$ has
moments of all orders which are uniformly bounded in $N$:
\[
\forall k>0,\forall t\in[0,\tf]\ \sup_{N>N_{0}}\e\left[\left(Z^{(N)}(t)\right)^{k}\right]<\infty.
\]
Moreover, for any fixed $k$, the $k$-th moment can be made arbitrarily
small in the following precise sense: for any $\ep>0$, there exists
$N_{\ep,k}$ large enough so that:
\[
\limsup_{t\to0}\sup_{N>N_{\ep,k}}\e\left[\left(Z^{(N)}(t)\right)^{k}\right]<\ep.
\]
\end{lem}
\begin{proof}
Fix any $\ga>0$ and then use the inequality $x^{k}\leq\frac{k!}{\ga^{k}}e^{\ga x}$
for $x\geq0$ and property i) of the weak exponential moment control
to find $N_{\ga}\in\bN$ so large so that we have: 
\begin{eqnarray*}
\sup_{N>N_{\ga}}\e\left[\left(Z^{(N)}(t)\right)^{k}\right] & \leq & \frac{k!}{\ga^{k}}\sup_{N>N_{\ga}}\e\left[e^{\ga Z^{(N)}(t)}\right]<\infty,
\end{eqnarray*}
which is finite by property i) of the exponential moment control from
Definition \ref{def:exp_mom_control}. This establishes the first
conclusion of the lemma. To see the second point, for any fixed $k\in\bN$
and $\ep>0$, choose $\ga$ large enough so that $\ep\ga^{k}>2k!$,
and then apply property ii) of the weak exponential moment control
to find $N_{\ga,1}$ large enough so that we have the following:
\[
\limsup_{t\to0}\sup_{N>N_{\ga,1}}\e\left[\left(Z^{(N)}(t)\right)^{k}\right]\leq\frac{k!}{\ga^{k}}\limsup_{t\to0}\sup_{N>N_{\ga,\ep}}\e\left[e^{\ga Z^{(N)}(t)}\right]\leq\frac{k!}{\ga^{k}}(1+1)\leq\ep.
\]
\end{proof}
\begin{lem}
\label{cor:overlap_to_sum_bound_noscale} Recall from Definition \ref{def:overlap_times_NIW}
the overlap time $O^{(\taf,\xf)}\left[s,s^{\prime}\right]$ between
the processes $\vec{X}^{(\taf,\xf)}$ and $\vec{X}^{\prime(\taf,\xf)}$.
We have the inequality: 
\begin{equation}
\frac{1}{j!}\left(O^{(\taf,\xf)}\left[s,s^{\prime}\right]\right)^{j}\geq \quad \sintt{\vec{\ta}\in\De_j(s,s^\prime)}{\vec{x}\in \{1,\ld,\xf+d\}^k} \one\left\{ \bigcap_{i=1}^{j}\left\{ x_{i}\in\X^{(\taf,\xf)}(\ta_{i})\right\} \cap\left\{ x_{i}\in\X^{\prime(\taf,\xf)}(\ta_{i})\right\} \right\} \dd \ta_1...\dd \ta_j. \label{eq:overlap_to_l2-1}
\end{equation}

\end{lem}

\begin{proof}
By Definition \ref{def:overlap_times_NIW} we have 
\begin{equation}
\left(O^{(\taf,\xf)}\left[s,s^{\prime}\right]\right)^{j} =\left(\intop_{s}^{s^{\prime}}\sum_{x=1}^{\xf+d}\one\left\{ x\in\X^{(\taf,\xf)}(\ta)\right\} \one\left\{ x\in\X^{\prime (\taf,\xf)}(\ta)\right\} \dd\ta\right)^{j}.\label{eq:overlap_bound}
\end{equation}
The desired inequality follows by expanding the RHS of equation (\ref{eq:overlap_bound})
as a $j$-fold integral/sum. We then switch from an un-ordered integral over $\vec{\ta} \in (s,s^\prime)^j$ to an ordered integral over $\vec{\ta} \in \De_j(s,s^\prime)$ at the cost of
the factor $j!$, which completes the result.\end{proof}
\begin{cor}
\label{cor:bound_on_sum_of_P2} For $0<s<s^{\prime}<\tf$, we have that:
\[
\sintt{\vec{\ta}\in\De_j(s,s^\prime)}{\vec{x}\in \{1,\ld,\xf+d\}^k} \quad \p\left(\bigcap_{i=1}^{j}\left\{ x_{i}\in\X^{(\taf,\xf)}(\ta_{i})\right\} \right)^{2}\dd\ta_{1}\ld\dd\ta_{j}\leq\e\left[\frac{1}{j!}\left(O^{(\taf,\xf)}\left[s,s^{\prime}\right]\right)^{j}\right].
\]
\end{cor}
\begin{proof}
Notice that since the processes $\vec{X}^{(\taf,\xf)}$ and $\vec{X}^{\prime(\taf,\xf)}$
are independent, we have
\begin{align*}
\e\left[\one\left\{ \bigcap_{i=1}^{j}\left\{ x_{i}\in\X^{(\taf,\xf)}(\ta_{i})\right\} \cap\left\{ x_{i}\in\X^{\prime(\taf,\xf)}(\ta_{i})\right\} \right\} \right] & =\p\left(\bigcap_{i=1}^{j}\left\{ x_{i}\in\X^{(\taf,\xf)}(\ta_{i})\right\} \right)^{2},
\end{align*}
where we have applied the definition of $\ps_{j}^{(N),(\tf,\zf)}$
from Definition \ref{def:rescaled_NIPb}. The desired results follows by taking $\e$ of both sides of the inequality
in equation \eqref{eq:overlap_to_l2-1}. \end{proof}
\begin{cor}
\label{cor:overlap_to_sum_bound} Recall from Definition \ref{def:overlap_times_NIW}
the rescaled overlap time $O^{(N),(\tf,\zf)}\left[s,s^{\prime}\right]$ between
the processes $\vec{X}^{(N),(\tf,\zf)}$
and $\vec{X}^{\prime(N),(\tf,\zf)}$. We have the inequalities 
\begin{align}
&\frac{1}{j!}\left(\sqrt{N} O^{(N),(\tf,\zf)}\left[s,s^{\prime}\right]\right)^{j} \geq \nonumber \\
& N^j \sintt{\vec{t}\in\De_j(s,s^\prime)}{\vec{z}\in \frac{\bZ^k}{\sqrt{N}}} \one\left\{ \bigcap_{i=1}^{j}\left\{ z_{i} -\sqrt{N}t_{i} \in\X^{(N),(\tf,\zf)}(t_{i})\right\} \cap\left\{ z_{i}-\sqrt{N}t_{i}\in\X^{\prime(N),(\tf,\zf)}(t_{i})\right\} \right\} \dd t_1...\dd t_j ,\label{eq:overlap_to_l2}
\end{align}
and
\begin{equation}
\norm{\ps_{j}^{(N),(\tf,\zf)}}_{L^{2}\left( \De_j(s,s^\prime)\times \bR^j \right)}^{2}\leq\e\left[\frac{1}{j!}\left(O^{\left(N\right)(\tf,\zf)}\left[s,s^{\prime}\right]\right)^{j}\right].\label{eq:L2_bound_ps_j}
\end{equation}

\end{cor}

\begin{proof}
Equation \ref{eq:overlap_to_l2} follows immediately from Lemma \ref{cor:overlap_to_sum_bound_noscale}
and the definition of the rescaled process in Definition \ref{def:rescaled_NIPb}. Equation
\ref{eq:L2_bound_ps_j} follows immediately from Corollary \ref{cor:bound_on_sum_of_P2}
using the definition of $\ps_{j}^{(N),(\tf,\zf)}$ from equation (\ref{eq:def_ps_k_N})
and the fact that the $L^{2}$ norm of $\ps_{j}^{(N),(\tf,\zf)}$
can be written as a semi-discrete sum as in equation (\ref{eq:integral_is_semidiscrete_sum}).\end{proof}

\begin{proof} (Of Proposition \ref{prop:uniform_exp_moms}.) By Corollary \ref{cor:overlap_to_sum_bound}
applied to each term, we have for any $\ell\in\bN$ that
\begin{equation}
\sum_{k=\ell}^{\infty}\ga^{k}\norm{\ps_k^{(N),(\tf,\zf)}}^2_{L^2(\De_k(0,\tf)\times\bR^k)} \leq \e\left[\sum_{k=\ell}^{\infty}\frac{\ga^{k}}{k!}\left(O^{\left(N\right)(\tf,\zf)}\left[0,\tf\right]\right)^{k}\right]. \nonumber 
\end{equation}
The interchange of expectation with the infinite sum is justified
by the monotone convergence theorem since $O^{(N),(\tf,\zf)}[0,\tf]$
is non-negative. Finally, since the overlap time $\left\{ O^{\left(N\right),(\tf,\zf)}\left[0,t\right]:t\in[0,\tf]\right\} _{N\in\bN}$
is weakly exponential moment controlled by Proposition \ref{prop:ON_exp_mom}, we apply property i) and property iii) of weak exponential moment control from Definition \ref{def:exp_mom_control} to get the desired conclusions.
\end{proof}

\begin{defn}
Recall from Definition \ref{def:space-time} the subdivision of the
space $\De_k(0,\tf)\times\bR^k$ into sets $D_{1},D_{2},D_{3},D_{4}$.
Further subdivide $D_{3}(\de)$ as follows
\begin{equation*}
D_{3}^{0,j}(\de) \defequal \Big( \De_j(0,\de)\times\De_{k-j}(\de,\tf) \Big) \times \bR^k, \quad D_{3}^{\tf,j}(\de) \defequal \Big( \De_{k-j}(0,\tf-\de)\times\De_{j}(\tf-\de,\tf) \Big) \times \bR^k,
\end{equation*}
so that $D_{3}(\de)=\bigcup_{j=1}^{k}D_{3}^{0,j}(\de)\cup\bigcup_{j=1}^{k}D_{3}^{\tf,j}(\de)$.
\end{defn}

\begin{proof} (Of Proposition \ref{prop:D3})
It suffices to show that for that for any $1\leq j\leq k$ and for
all $\ep>0$, there exists $\de>0$ so that:
\begin{equation}
\limsup_{N\to\infty}\iintop_{D_{3}^{0,j}(\de)}\abs{\ps^{(N),(\tf,\zf)}\big(\vec{t},\vec{z}\big)}^{2}\dd\vec{t}\dd\vec{z} < \ep, \quad
\limsup_{N\to\infty}\iintop_{D_{3}^{\tf,j}(\de)}\abs{\ps^{(N),(\tf,\zf)}\big(\vec{t},\vec{z}\big)}^{2}\dd\vec{t}\dd\vec{z} < \ep,\label{eq:D3_0j_target}
\end{equation}
since once this is proven we can use a union bound and 
$D_{3}(\de)=\bigcup_{j=1}^{k}D_{3}^{0,j}(\de)\cup\bigcup_{j=1}^{k}D_{3}^{\tf,j}(\de)$
is a union of these $2k$ pieces. We will show only the bound in equation
(\ref{eq:D3_0j_target}) for $D_{3}^{0,j}$  as the result for
$D_{3}^{\tf,j}$ follows in an analogous way. We first
observe that
\begin{equation}
\iintop_{D_{3}^{0,j}(\de)}\abs{\ps_{k}^{(N),(\tf,\zf)}\big(\vec{t},\vec{z}\big)}^{2}\dd\vec{t}\dd\vec{z}\leq\e\left[\frac{1}{j!}\left(O^{(N),(\tf,\zf)}[0,\de]\right)^{j}\frac{1}{(k-j)!}\left(O^{(N),(\tf,\zf)}[\de,\tf]\right)^{k-j}\right]\label{eq:bound_on_D3_0}
\end{equation}
The justification of equation (\ref{eq:bound_on_D3_0}) follows in
the same way as the proof of Corollary \ref{cor:overlap_to_sum_bound} by applying
the inequality from equation (\ref{eq:overlap_to_l2}) and then taking
$\e$ of both sides. Applying the Cauchy-Schwarz inequality to the
RHS of equation (\ref{eq:bound_on_D3_0}) and using the fact that
$O^{(N),(\tf,\zf)}[0,t]$ is monotone increasing in $t$ gives us
\begin{equation}
\text{ LHS }\eqref{eq:bound_on_D3_0}  \leq  \frac{1}{j!(k-j)!}\sqrt{\e\left[\left(O^{(N),(\tf,\zf)}[0,\de]\right)^{2j}\right]\e\left[\left(O^{(N),(\tf,\zf)}[0,\tf]\right)^{2(k-j)}\right]}.\label{eq:D3_j_after_CS}
\end{equation}
We now use the weak exponential moment control of $\left\{ O^{(N),(\tf,\zf)}[0,t]:t\in(0,\tf)\right\} _{N\in\bN}$
from Proposition \ref{prop:ON_exp_mom}. By Lemma \ref{lem:exp-mom-moments},
we find $N_{0}\in\bN$ large enough so that
\begin{equation}
\sup_{N>N_{0}}\e\left[\left(O^{(N),(\tf,\zf)}[0,\tf]\right)^{2(k-j)}\right]<\infty,
\end{equation}
and, for each $\ep>0,$ an $N_{\ep}\in\bN$ large enough so that 
\begin{align*}
&\limsup_{\de\to0}\sup_{N>N_{\ep}}\e\left[\left(O^{(N),(\tf,\zf)}[0,\de]\right)^{2j}\right]\\
\leq&\left(j!(k-j)!\frac{\ep}{2}\right)^{2}\left(\sup_{N>N_{0}}\e\left[\left(O^{(N),(\tf,\zf)}[0,\tf]\right)^{2(k-j)}\right]\right)^{-1}.
\end{align*}
Combining this with the inequality from equation (\ref{eq:D3_j_after_CS})
we arrive at
\[
\limsup_{\de\to0}\sup_{N>N_{\ep}}\iintop_{D_{3}^{0,j}(\de)}\abs{\ps_{k}^{(N),(\tf,\zf)}\big(\vec{t},\vec{z}\big)}^{2}\dd\vec{t}\dd\vec{z}\leq\frac{\ep}{2}.
\]
Since this $\limsup$ as $\de\to0$ is less than $\ep/2$, there exists
$\de>0$ small enough to verify equation (\ref{eq:D3_0j_target})
as desired.\end{proof}
\begin{proof} (Of Proposition \ref{prop:D4})
Let $W^{(N),(\tf,\zf)}\defequal\max_{i\in\left\{ 1,\ld,d\right\} }\sup_{t\in[0,\tf]}\abs{X_{i}^{(N),(\tf,\zf)}(t)}$
be the largest absolute value achieved by the ensemble at any time
$t\in[0,\tf]$, and let $W^{\prime(N),(\tf,\zf)}$ be the same for
an independent copy $X^{\prime(N),(\tf,\zf)}.$ By the definition of the set $D_4$ we have that LHS of \eqref{eq:D4_target} is bounded above by
\begin{equation}
N^{\frac{k}{2}} \sintt{\vec{t}\in\De_k(0,\tf)}{\vec{z}\in N^{-\half}\bZ^k} \quad \one\left\{\bigcup_{i=1}^{k}\abs{z_i-\sqrt{N} t_i}>M\right\} \p\left( \bigcap_{i=1}^{k} \left\{ z_i- \sqrt{N} t_i \in \vec{X}^{(N),(\tf,\zf)}(t_i)\right\}  \right)^2 \dd t_1 \ld \dd t_k
\end{equation}
Since $\vec{X}^{(\tf,\zf)}$ and $\vec{X}^{\prime,(\tf,\zf)}$ are independent, we write this as:
\begin{align*}
   & \text{LHS \eqref{eq:D4_target}} \\
 = & N^{\frac{k}{2}} \e \Bigg[ \quad \quad \sintt{\vec{t}\in\De_k(0,\tf)}{\vec{z}\in N^{-\half}\bZ^k} \quad \one\left\{ \bigcup_{i=1}^{k}\abs{z_{i}-\sqrt{N}t_i}>M\right\} \Bigg. \\
& \quad \quad \quad \Bigg. \times  \one\left\{ \bigcap_{i=1}^{k}\left\{ z_{i}-\sqrt{N}t_i\in\X^{(N),(\tf,\zf)}(t_{i})\right\}  \bigcap_{j=1}^{k}\left\{ z_{i}-\sqrt{N}t_i\in{\X}^{\prime, (N),(\tf,\zf)}(t_{i})\right\} \right\} \Bigg]\\
  \leq & N^{\frac{k}{2}} \e\Bigg[\one\left\{ W^{(N),(\tf,\zf)}>M\right\} \one\left\{ {W}^{\prime(N),(\tf,\zf)}>M\right\} \Bigg.\\
   & \quad \Bigg.\times \sintt{\vec{t}\in\De_k(0,\tf)}{\vec{z}\in N^{-\half}\bZ^k} \quad \one\bigg\{ \bigcap_{i=1}^{k}\left\{ z_{i}-\sqrt{N}t_i\in\X^{(N),(\tf,\zf)}(t_{i})\right\} \bigcap_{i=1}^{k}\left\{ z_{i}-\sqrt{N}t_i\in\X^{\prime (N),(\tf,\zf)}(t_{i})\right\} \bigg\} \Bigg].
\end{align*}
The last inequality follows by inclusion since if $\abs{z_{i}-\sqrt{N}t_i}>M$ and
$z_{i}-\sqrt{N}t_i\in\vec{X}^{(N),(\tf,\zf)}(t_{i})$, then the maximum has $W^{(N),(\tf,\zf)}>M$. By application of Corollary \ref{cor:overlap_to_sum_bound}, and Cauchy-Schwarz we have that LHS of  equation \eqref{eq:D4_target} is bounded above by:
\begin{align}
& \e\left[\one\left\{ W^{(N),(\tf,\zf)}>M\right\} \cdot\one\left\{ W^{\pr(N),(\tf,\zf)}>M\right\} \frac{1}{k!}\left(O^{\left(N\right)(\tf,\zf)}\left[0,\tf\right]\right)^{k}\right]\nonumber \\
 \leq & \frac{1}{k!}\p\left[W^{(N),(\tf,\zf)}>M\right]\sqrt{\e\left[\left(O^{\left(N\right)(\tf,\zf)}\left[0,\tf\right]\right)^{2k}\right]}.\label{eq:D4_post_CS}
\end{align}
Finally, by the weak exponential moment control from Proposition \ref{prop:ON_exp_mom}
and Lemma \ref{lem:exp-mom-moments} we know there is an $N_{0}\in\bN$
so that $\sup_{N>N_{0}}\e\left[\left(O^{\left(N\right)(\tf,\zf)}\left[0,\tf\right]\right)^{2k}\right]<\infty$.
Since this is bounded and since $\lim_{M\to\infty}\sup_{N>N_{0}}\p\left[W^{(N),(\tf,\zf)}>M\right]=0$,
(see e.g. Lemma \ref{lem:exp-mom-for-location}) it is possible to choose an $M$ so large
so that the RHS of equation (\ref{eq:D4_post_CS}) is less than $\ep$. \end{proof}

\appendix

\section{Appendix}
\subsection{Facts about Poisson processes}
\begin{lem}
\label{lem:Poisson_fact}Let $P(\ta)$ be an ordinary unit rate Poisson process. Let $\bar{P}(\ta) \defequal P(\ta)-\ta$ denote the compensated version of this walk. Then:
\begin{equation*}
\p\left(\sup_{0<\ta<tN}\bar{P}(\ta)-\inf_{0<\ta<tN}\bar{P}(\ta)>y\sqrt{N}\right)  \leq 2\exp\left(-\frac{1}{4}\frac{y^{2}}{t}\right)+2\exp\left(-\frac{1}{4}\sqrt{N}y\right).
\end{equation*}
\end{lem}
\begin{proof}
We show that $\p\left(\sup_{0<\ta<tN}\bar{P}(\ta)>y\sqrt{N}\right)$
and $\p\left(\inf_{0<\ta<tN}\bar{P}(\ta)<-y\sqrt{N}\right)$ both separately obey this type of inequality, and the result will follow by a union bound. Fix any $T>0$, $x>0$. Since $\bar{P}(\ta)$ is a martingale, we have by Doob's inequality for the running maximum of
any sub-martingale that for any $\la>0$:
\begin{equation*}
\p\left(\sup_{0<\ta<T}\bar{P}(\ta) \geq x\right) \leq \frac{\e\left[\exp\left(\la\bar{P}(T)\right)\right]}{\exp\left(\la x\right)} \leq \exp\left(T\left(\frac{x}{T}-\ln\left(1+\frac{x}{T}\right)\left(1+\frac{x}{T}\right)\right)\right),
\end{equation*}
where we have used the minimizing value $\la=\ln\left(1+\frac{x}{T}\right)$
to get the last inequality. We now use the Taylor series inspired
bound $z-\ln(1+z)(1+z)=-\frac{1}{2}z^{2}+\frac{1}{2}\intop_{0}^{z}\frac{(z-t)^{2}}{(1+t)^{2}}\dd t\leq -\frac{1}{2}z^{2}+\frac{1}{2}\intop_{0}^{z}\frac{z^{2}}{(1+t)^{2}}\dd t = -\frac{1}{2}\frac{z^{2}}{1+z}$
to get:
\[
\p\left(\sup_{0<\ta<T}\bar{P}(\ta)\geq x\right)\leq\exp\left(-\frac{1}{2}T\frac{\left(\frac{x}{T}\right)^{2}}{1+\frac{x}{T}}\right)\le\exp\left(-\frac{1}{4}\frac{x}{T}^{2}\right)+\exp\left(-\frac{1}{4}x\right).
\]
The last inequality follows by considering the cases $x<T$ and $x\geq T$
separately. Putting $T=tN$ and $x=y\sqrt{N}$ gives the desired result of the Lemma. The
same argument works to prove the bound on $\p\left(-\inf_{0<\ta<T}\bar{P}(\ta)>y\right)$
using Doob's inequality again and the fact that $-\bar{P}(\ta)$
is also a martingale.
\end{proof}
\begin{prop}
\label{prop:local_clt}(Local Central Limit Theorem for Poisson probabilities) Recall the Poisson probability mass function $\mu$ from Definition \ref{def:NIP}. We have that
\[
\lim_{M\to\infty}\sup_{z\in\bR}\abs{\sqrt{M}\mu\left(M,M+\floor{\sqrt{M}z}\right)-\frac{1}{\sqrt{2\pi}}\exp\left(-\half z^{2}\right)}=0.
\]
\end{prop}
\begin{proof}
Notice that $\mu(M,M+y)$ can be realized as a probability:
\[
\mu\left(M,M+y\right)=\p\left(\sum_{i=1}^{M}\left(\xi_{i}-1\right)=y\right),
\]
where $\xi_{i}$ are i.i.d. Poisson random variables of rate $1$. The
result then follows by the local limit theorem for sums of mean zero
random variables, see for example Theorem 3.5.2 in \cite{durrett2010probability}.
\end{proof}
\begin{cor}
\label{cor:local_clt_bound}
There exists a constant $C_P$ so that:
\[
\sup_{M \in \bN} \sqrt{M}\mu\left(M,M+\floor{\sqrt{M}z}\right) \leq C_P.
\]
\end{cor}
\begin{proof}
This follows by applying the triangle inequality to the result from Proposition \ref{prop:local_clt} and the bound $\sup_{z\in\bR} \exp(-\half z^2)\leq 1$.
\end{proof}

\subsection{Proof of determinantal kernel for non-intersecting Poisson bridges}
\begin{prop}
\label{prop:KP_is_the_kernel} For any $\taf>0$ and any $\xf\in\bN$,
recall from Definition \ref{def:Krawtchouk-polynomials} the space-time
kernel $K_{P}^{(\taf,\xf)}$. Given any list of times $\ta_1,\ld,\ta_k \in (0,\taf)$ and coordinates $x_1,\ld,x_k \in \{1,\ld,\xf+d-1\}$, the $k$-point correlation function
for non-intersecting Poisson bridges is given by:\textup{
\begin{eqnarray}
\p\left(\bigcap_{i=1}^{k}\left\{ x_{i}\in\vec{X}^{(\taf,\xf)}(\ta_{i})\right\} \right) & = & \det\left[K_{P}^{(\taf,\xf)}\Big(\left(\ta_{i},x_{i}\right);\left(\ta_{j},x_{j}\right)\Big)\right]_{i,j=1}^{k}.\label{eq:psN_is_det_KN-1}
\end{eqnarray}
}
\end{prop}
The key ingredient in the proof of this result is the Eynard-Mehta type theorem:
\begin{prop}
\label{prop:explicit_kernel_thm} {(}Theorem 1.7. from \cite{Joh03}, see also Section 1.2 of \cite{Johansson2005-Hahn}{)} Fix $m\in\bN$. Suppose $X_{r}$, $0\leq r\leq m+1$
are subsets of $\bN$ and that $\ph_{r,r+1}:X_{r}\times X_{r+1}\to\bR,0\leq r \leq m$
are given functions. An element $\underline{x}=\left(\vec{x}^{(1)},\ld,\vec{x}^{(m)}\right)\in X_{1}^{d}\times\ld\times X_{m}^{d}\defequal\cX$
is called a configuration. We think of $\vec{x}^{(r)}=\left(x_{1}^{(r)},\ld,x_{d}^{(r)}\right)$
as the positions of $d$ particles in the set $X_{r}$. Let $\vec{x}^{(0)}\in X_{0}^{d}$ and $\vec{x}^{(m)}\in X_{m+1}^{d}$
be fixed configurations, the initial and final configurations respectively.
Define for $0\leq r<s\leq m+1$ the function $\ph_{r,s}:X_{r}\times X_{s}\to\bR$
by:
\begin{equation}
\ph_{r,s}(x,y)=\sum_{z_{1}\in X_{r+1}}\ld\sum_{z_{s-r-1}\in X_{s-1}}\ph_{r,r+1}(x,z_{1})\ph_{r+1,r+2}(z_{1},z_{2})\ld\ph_{s-1,s}(z_{s-r-1},y),\label{eq:convolutions}
\end{equation}
and define $\ph_{r,s}\equiv0$ if $r\geq s$. Consider a random configuration
$\underline{X}\in\cX$ given by the following prescription:
\[
\p\left(\underline{X}\right)=\frac{1}{Z_{d,m}}\prod_{r=0}^{m}\det\left[\ph_{r,r+1}\big(x_{i}^{(r)},x_{j}^{(r+1)}\big)\right]_{i,j=1}^{d}.
\]
Then, for any $k\in\bN$, and any list of space-time coordinates $\{(r_i,x_i)\}_{i=1}^k\in \big( \{1,\ld,m\} \times \bN \big)^k $ we have the following determinantal formula
for the probability to find these space-time points occupied:
\[
\p\left(\bigcap_{i=1}^{k}\left\{ x_{i}\in\vec{X}^{(r_{i})}\right\} \right)=\det\Big[K_{d,m}\left(r_{i},x_{i};r_{j},x_{j}\right)\Big]_{i,j=1}^{k},
\]
where the kernel $K$ (which does not depend on $k$) is explicitly
given by:
\begin{align}
K_{d,m}(r,x;r^{\prime},x^{\prime}) & \defequal-\ph_{r,s}(x,y)+\sum_{i,j=1}^{d}\ph_{r,m+1}\left(x,x_{i}^{(m)}\right)\left(A^{-1}\right)_{i,j}\ph_{0,r^{\prime}}\left(x_{j}^{(0)},x^{\prime}\right),\nonumber \\
A_{ij} & \defequal\ph_{0,m+1}\left(x_{i}^{(0)},x_{j}^{(m)}\right).\label{eq:A_defn}
\end{align}
\end{prop}
\begin{rem}
Typically this type of measure arises in the context of non-intersecting
processes as a consequence of the Lindstr{\"o}m-Gessel-Viennot/Karlin-MacGregor formula. The
difficultly in practice is inverting the matrix $A$ which appears
in the formula for the kernel. The approach we will follow goes by
using row and column manipulations to rewrite the functions $\ph$
in terms of orthogonal polynomials. Because these polynomials are
orthogonal, the matrix $A$ becomes diagonal and finding $A^{-1}$
is possible. 
\end{rem}
\begin{lem}
\label{lem:LGV}Recall the definition of the non-intersecting Poisson bridges from Definition \ref{def:NIPb}. For any fixed sequence of times $0=\ta^{(0)}<\ta^{(1)}<\ld<\ta^{(m)}<\ta^{(m+1)}=\taf$,
and any list of vectors $\left\{ \vec{x}^{(i)}\right\} _{i=0}^{m+1}$
with $\vec{x}^{(0)}=\vec{\de}_d(0)\in\bN^{d}$ and $\vec{x}^{(m+1)}=\vec{\de}_d\left(\xf\right)\in\bN^{d}$
we have: 

\begin{align*}
\p\left(\bigcap_{i=1}^{m}\left\{ \vec{X}^{(\taf,\xf)}(\ta^{(i)})=\vec{x}^{(i)}\right\} \right)= & {Z^{-1}_{\taf,\xf}} \prod_{r=0}^{m}\det\left[\mu\left(\ta^{(r+1)}-\ta^{(r)},x_{i}^{(r+1)}-x_{j}^{(r)}\right)\right]_{i,j=1}^{d},
\end{align*}
where $Z_{\taf,\xf}$ is a normalizing constant.\end{lem}
\begin{proof}
Let $\vec{P}(\ta)\in\bN^{d}$ denote an ordinary Poisson process which
consists of $d$ independent Poisson processes. By the Lindstr{\"o}m-Gessel-Viennot/Karlin-McGregor
theorem, the probability of $d$ Poisson paths to
go from an initial position $\vec{x}\in\bN^{d}$ at an original
time $\ta$ to a final position $\vec{x}^{\prime}\in\bN^{d}$ at a final
time $\ta^{\prime}$ without intersection is given by:
\[
\p\left(\left\{ \text{Non-intersecting in }(\ta,\ta^{\prime})\right\} \cap\left\{ \vec{P}(\ta^{\prime})=\vec{x}^{\prime}\right\} \given{ \vec{P}(\ta)=\vec{x} }\right)=\det\left[\mu\left(\ta^\prime-\ta,x_{i}^{\prime}-x_{j}\right)\right]_{i,j=1}^{d}
\]
The result of the Lemma then follows by the definition of $\vec{X}^{(\taf,\xf)}$
as the Markov process of Poisson processes  conditioned on non-intersection and
with initial and final conditions $\vec{x}^{(0)}$ and $\vec{x}^{(m+1)}$
respectively.
\end{proof}
\begin{defn}
Recall the definition of the polynomials $R_{j}$ and $\tilde{R}_{j}$
from Definition \ref{def:Krawtchouk-polynomials}, and the Poisson
probability mass function $\mu$ from Definition \ref{def:NIP}. Define
a family of auxiliary functions $\la_{j}:\bR\times\bN\to\bR$ and
$\tilde{\la}_{j}:\bR\times\bN\to\bR$ for $0\leq j\leq d-1$ by: 
\begin{align*}
\la_{j}(\ta,x) & \defequal R_{j}(\ta,x)\mu(\ta,x).\\
\tilde{\la}_{j}(\ta,x) & \defequal\tilde{R}_{j}(\ta,x)\mu\left(\xf-x+d-1,\taf-\ta\right).
\end{align*}
\end{defn}
\begin{lem}
\label{lem:Linear_combination}For each $j$, $\la_{j}(\ta,x)$ and
$\tilde{\la}_{j}\left(\ta,x\right)$ can be written as a linear combination
of the functions $\left\{ \mu\left(\ta,x-i\right)\right\} _{i=0}^{j}$
and $\left\{ \mu\left(\taf-\ta,\xf-x+d-1-i\right)\right\} _{i=0}^{j}$
respectively:
\begin{align*}
\la_{j}\left(\ta,x\right) & =\sum_{i=0}^{j}a_{i,j}\mu\left(\ta,x-i\right).\\
\tilde{\la}_{j}\left(\ta,x\right) & =\sum_{i=0}^{j}a_{i,j}\mu\left(\taf-\ta,\xf-x+d-1-i\right).\\
a_{i,j} & \defequal(-1)^{i}\binom{j}{i}\frac{\taf^{i}}{\left(\xf+d-1\right)_{-i}},
\end{align*}
where $(x)_{-i}$ denotes the falling factorial $(x)_{-i}=x\cdot(x-1)\cdot\ld\cdot(x-i+1)$.
\end{lem}
\begin{proof}
This is verified directly from the definition of $R_{j}(\ta,x)$ in
terms of the hypergeometric function $_{2}F_{1}$ from Definition
\ref{def:Krawtchouk-polynomials}. From this definition we have that:
\begin{equation}
R_{j}(\ta,x)\mu(\ta,x)=\sum_{i=0}^{j}(-1)^{i}\binom{j}{i}\frac{(x)_{-i}}{(\xf+d-1)_{-i}}\left(\frac{\taf}{\ta}\right)^{i}\cdot e^{-\ta}\frac{\ta^{x}}{x!}.\label{eq:Rj_mu}
\end{equation}
On the other hand, we have:
\begin{equation}
\sum_{i=0}^{j}a_{i,j}\mu\left(\ta,x-i\right)=\sum_{i=0}^{j}(-1)^{i}\binom{j}{i}\frac{\taf^{i}}{\left(\xf+d-1\right)_{-i}}e^{-\ta}\frac{\ta^{x-i}}{\left(x-i\right)!}.\label{eq:sum_aij_mu}
\end{equation}
The RHS of equations (\ref{eq:Rj_mu}) and (\ref{eq:sum_aij_mu}) are seen to
be equal by the identity $x!=(x-i)!(x)_{-i}$. A very similar
calculation holds for $\tilde{\la}_{j}(\ta,x)$.\end{proof}
\begin{cor}
\label{cor:convolution_helper} We have the identities:
\begin{equation*}
\sum_{x\in\bN}\la_{j}(\ta,x)\mu(\ta^\prime,y-x) =\la_{j}\left(\ta+\ta^\prime,y\right), \; \sum_{x\in\bN}\mu(\ta,x)\tilde{\la}_{j}(\ta^\prime,y-x) =\tilde{\la}_{j}\left(\ta+\ta^\prime,y\right).
\end{equation*}
\end{cor}
\begin{proof}
First notice that $\sum_{x\in\bN}\mu(\ta,x)\mu(\ta^\prime,y-x)=\mu\left(\ta+\ta^\prime,y\right)$. This is the well-known fact that a sum of
two independent Poisson variables is again Poisson (or in other words,
the convolution of two Poisson weights is again Poisson). The identities then follow by the observation from Lemma \ref{lem:Linear_combination} that $\la_{j}$ and $\tilde{\la}_{j}$ are linear combinations of weights $\mu$.\end{proof}
\begin{lem}
For any sequence of times $0=\ta^{(0)}<\ta^{(1)}<\ld<\ta^{(m)}<\ta^{(m+1)}=\taf$,
and any list of vectors $\left\{ \vec{x}^{(i)}\right\} _{i=1}^{m}$
with $\vec{x}^{(0)}=\vec{\de}_d(0)$ and $\vec{x}^{(m+1)}=\vec{\de}_d\left(\xf\right)$
we have: 
\begin{align*}
& \p\left(\bigcap_{i=1}^{m}\left\{ \vec{X}^{(\taf,\xf)}(\ta^{(i)})=\vec{x}^{(i)}\right\} \right) \\
 & = \frac{{Z^{-1}_{\taf,\xf}}}{\left(\prod_{i=0}^{d-1}a_{i,i}\right)^{2}}\det\left[\la_{i-1}\left(\ta^{(1)},x_{j}^{(1)}\right)\right]_{i,j=1}^{d} \\
 & \times \prod_{r=1}^{m-1}\det\left[\mu\left(\ta^{(r+1)}-\ta^{(r)},x_{i}^{(r+1)}-x_{j}^{(r)}\right)\right]_{i,j=1}^{d}\det\left[\tilde{\la}_{i-1}\left(\ta^{(m)},x_{j}^{(m)}\right)\right]_{i,j=1}^{d},
\end{align*}
where $a_{i,i}$ is as defined in Lemma \ref{lem:Linear_combination}.
\end{lem}
\begin{proof}
This follows by applying row operations on the determinants that
appear in Lemma \ref{lem:LGV} to create the linear combinations
that appear in Lemma \ref{lem:Linear_combination}. \end{proof}
\begin{lem}
\label{lem:la_orthog_helper} The polynomials $R_{j}(\ta,x)$ and
$\tilde{R}_{j}(\ta,x)$ are related by the identity:
\begin{equation} \label{eq:R_R_tilde_ident}
R_{j}(\ta,x)=(-1)^{j}\left(\frac{\taf}{\ta}-1\right)^{j}\tilde{R}_{j}(\ta,x).
\end{equation}
Moreover, for $1\leq j,k\leq d-1$, and any $\ta\in(0,\tf)$ we have:
\[
\sum_{x\in\bN}\la_{j}\left(\ta,x\right)\tilde{\la}_{k}\left(\ta,x\right)=(-1)^{j}\frac{\mu(\taf,\xf+d-1)}{\binom{\xf+d-1}{j}}\de_{j,k}.
\]
\end{lem}
\begin{proof}
We use the Euler transformation for the hypergeometric function $_{2}F_{1}$,
\[
_{2}F_{1}\binom{a,b}{c}(z)=(1-z)^{-a}{_{2}F_{1}}\binom{a,c-b}{c}\left(\frac{z}{z-1}\right).
\]
This identity applied to the polynomials $R_{j}$ and $\tilde{R}_{j}$
from Definition \ref{def:Krawtchouk-polynomials} verifies equation \eqref{eq:R_R_tilde_ident}. Once this is established, we now compute: 
\begin{align*}
& \sum_{x\in\bN}\la_{j}\left(\ta,x\right)\tilde{\la}_{k}\left(\ta,x\right) \\
 & =\frac{\mu\left(\taf,\xf+d-1\right)}{\left(\frac{\taf}{\ta}-1\right)^{j}}\sum_{x\in\bN}R_{j}(\ta,x)R_{k}(\ta,x)\binom{\xf+d-1}{x}\left(\frac{\ta}{\taf}\right)^{x}\left(1-\frac{\ta}{\taf}\right)^{\xf+d-1-x}\\
 & =\frac{\mu\left(\taf,\xf+d-1\right)}{\left(\frac{\taf}{\ta}-1\right)^{j}}\frac{1}{\binom{\xf+d-1}{k}}\left(\frac{\taf}{\ta}-1\right)^{k}\de_{j,k},
\end{align*}
where we have applied the orthogonality relation for the Krawtchouk
polynomials. \end{proof}

\begin{proof}
(Of Proposition \ref{prop:KP_is_the_kernel}) It suffices to verify the result for the process $X^{(\taf,\xf)}$ restricted to an arbitrary finite list of times $0=\ta^{(0)}<\ta^{(1)}<\ld<\ta^{(m)}<\ta^{(m+1)}=\taf$. With such a list fixed, define the initial/final conditions by $x^{(0)}\defequal \vec{\de_d}(0)$, $x^{(m)} \defequal \vec{\de_d}(\xf)$ and the functions $\ph_{r,r+1}$ for $0\leq r \leq m$  by:
\begin{align*}
\ph_{r,r+1}(x,x^{\prime}) & \defequal\mu(\ta^{(r+1)}-\ta^{(r)},x^{\prime}-x).\\
\ph_{0,1}(x_{j}^{(0)},x) & \defequal\la_{j}(\ta^{(1)},x).\\
\ph_{m,m+1}(x,x_{j}^{(m+1)}) & \defequal \tilde{\la}_{j}(\ta^{(m)},x).
\end{align*}
By Lemma \ref{lem:LGV}, the probability of any configuration is explicitly
given by equation (\ref{eq:psN_is_det_KN-1}), which matches the hypothesis of Proposition \ref{prop:explicit_kernel_thm}. By Corollary \ref{cor:convolution_helper},
we can explicitly perform the convolutions that appear in equation
(\ref{eq:convolutions}), to get 
\begin{align*}
\ph_{0,r}(x_{j}^{(0)},x) & =\la_{j}\left(\ta^{(r)},x\right)\text{ for } 1\leq r\leq m.\\
\ph_{r,m+1}(x,x_{j}^{(m+1)}) & =\tilde{\la}_{j}\left(\ta^{(r)},x\right)\text{ for } 1\leq r \leq m.\\
\ph_{r,s}(x^{\prime},x) & =\mu\left(\ta^{(s)}-\ta^{(r)},x^{\prime}-x\right)\text{ for } 1\leq r<s \leq m.
\end{align*}
It remains only to calculate $\ph_{0,m+1}$. By Lemma \ref{lem:la_orthog_helper}
we compute
\begin{equation*}
\ph_{0,m+1}(x_{j}^{(0)},x_{k}^{(m)}) =\sum_{x\in\bN}\ph_{0,r}(x_{j}^{(0)},x)\ph_{r,m+1}(x,x_{k}^{(m+1)})=(-1)^{j}\frac{\mu(\taf,\xf+d-1)}{\binom{\xf+d-1}{j}}\de_{j,k}.
\end{equation*}
(Note that the above calculation works for any $0<r<m$ and that the final result does not depend on $r$ or $m$.) Thus the
matrix $A$, defined in equation (\ref{eq:A_defn}) is diagonal! Inverting it and applying the conclusion of Proposition \ref{prop:explicit_kernel_thm}
gives the kernel $K_{P}$ as desired.\end{proof}


\providecommand{\bysame}{\leavevmode\hbox to3em{\hrulefill}\thinspace}
\providecommand{\MR}{\relax\ifhmode\unskip\space\fi MR }
\providecommand{\MRhref}[2]{%
  \href{http://www.ams.org/mathscinet-getitem?mr=#1}{#2}
}
\providecommand{\href}[2]{#2}


\ACKNO{The author is grateful to Ivan Corwin for suggesting the problem as well as for discussions and advice during the writing process. The author also thanks Jeremy Quastel for helpful discussions and for making an early draft of a related result \cite{QuastelInPrep} available. I am also grateful to the anonymous referees for their careful reading and comments which greatly improved the presentation of the article.}


\end{document}